\documentclass[12pt]{article}
\usepackage{graphicx}
\usepackage{xcolor}
\usepackage{amssymb,amsfonts,amsthm,amsmath,calligra,tikz}
\usepackage[utf8]{inputenc}
\usepackage{accents}
\usepackage{slashed}
\usepackage{yfonts}
\usepackage{mathrsfs,pifont}
\theoremstyle{definition}
\usepackage{tikz}
\usepackage[all]{xy}

\def\be{\begin{eqnarray}}
\def\ee{\end{eqnarray}}

\usepackage{geometry}

\def\matZ{{\mathbb{Z}}}
\def\matR{{\mathbb{R}}}
\def\matQ{{\mathbb{Q}}}
\def\matC{{\mathbb{C}}}

\newcommand{\bA}{\mathsf{A}}
\newcommand{\bG}{\mathsf{G}}
\newcommand{\bT}{\mathsf{T}}

\newcommand{\cB}{\mathscr{B}}
\newcommand{\cE}{\mathscr{E}}
\newcommand{\cU}{\mathscr{U}}
\newcommand{\cX}{\mathscr{X}}

\newcommand{\Ell}{\mathrm{Ell}}
\newcommand{\Pic}{\mathrm{Pic}}
\newcommand{\E}{\textsf{E}}
\newcommand{\Or}{\textsf{O}}
\newcommand{\Stab}{\mathrm{Stab}}
\theoremstyle{definition}
\newtheorem{Definition}{Definition}
\newtheorem{Proposition}{Proposition}
\newtheorem{Lemma}{Lemma}

\newtheorem{Theorem}{Theorem}

\newcommand{\hilb}{\boldsymbol{H}}
\newcommand{\ahilb}{\boldsymbol{AH}}
\newcommand{\fC}{\mathfrak{C}} 

\def\tb {{\cal{V}}}  

\newcommand{\h}{h} 

\makeatletter
\newcommand{\somespecialrotate}[3][]{%
\begingroup
\sbox\@tempboxa{#3}%
\@tempdima=.5\wd\@tempboxa
\sbox\@tempboxa{\rotatebox[#1]{#2}{\usebox\@tempboxa}}%
\advance\@tempdima by -.5\wd\@tempboxa
\mbox{\hskip\@tempdima\usebox\@tempboxa}%
\endgroup}
\linethickness{1 pt}
\newsavebox{\ybb}
\savebox{\ybb}(1,1){
	\put(0, 0){ \color{blue} \line(0, 1){ 36}} \put(2, 0){ \color{blue} \line(0, 1){ 36}} \put(4, 0){ \color{blue} \line(0, 1){ 30}} \put(6, 0){ \color{blue} \line(0, 1){ 28}} \put(8, 0){ \color{blue} \line(0, 1){ 26}} \put(10, 0){ \color{blue} \line(0, 1){ 24}} \put(12, 0){ \color{blue} \line(0, 1){ 20}} \put(14, 0){ \color{blue} \line(0, 1){ 16}} \put(16, 0){ \color{blue} \line(0, 1){ 16}} \put(18, 0){ \color{blue} \line(0, 1){ 14}} \put(20, 0){ \color{blue} \line(0, 1){ 12}} \put(22, 0){ \color{blue} \line(0, 1){ 10}} \put(24, 0){ \color{blue} \line(0, 1){ 10}} \put(26, 0){ \color{blue} \line(0, 1){ 8}} \put(28, 0){ \color{blue} \line(0, 1){ 6}} \put(30, 0){ \color{blue} \line(0, 1){ 4}} \put(32, 0){ \color{blue} \line(0, 1){ 2}} \put(34, 0){ \color{blue} \line(0, 1){ 2}} \put(36, 0){ \color{blue} \line(0, 1){ 2}} \put(0, 0){ \color{blue} \line(1, 0){ 36}} \put(0, 2){ \color{blue} \line(1, 0){ 36}} \put(0, 4){ \color{blue} \line(1, 0){ 30}} \put(0, 6){ \color{blue} \line(1, 0){ 28}} \put(0, 8){ \color{blue} \line(1, 0){ 26}} \put(0, 10){ \color{blue} \line(1, 0){ 24}} \put(0, 12){ \color{blue} \line(1, 0){ 20}} \put(0, 14){ \color{blue} \line(1, 0){ 18}} \put(0, 16){ \color{blue} \line(1, 0){ 16}} \put(0, 18){ \color{blue} \line(1, 0){ 12}} \put(0, 20){ \color{blue} \line(1, 0){ 12}} \put(0, 22){ \color{blue} \line(1, 0){ 10}} \put(0, 24){ \color{blue} \line(1, 0){ 10}} \put(0, 26){ \color{blue} \line(1, 0){ 8}} \put(0, 28){ \color{blue} \line(1, 0){ 6}} \put(0, 30){ \color{blue} \line(1, 0){ 4}} \put(0, 32){ \color{blue} \line(1, 0){ 2}} \put(0, 34){ \color{blue} \line(1, 0){ 2}} \put(0, 36){ \color{blue} \line(1, 0){ 2}}
	\put(1, 1){ \color{red} \line(0, 1){ 2}} \put(1, 3){ \color{red} \line(0, 1){ 2}} \put(1, 5){ \color{red} \line(0, 1){ 2}} \put(1, 7){ \color{red} \line(0, 1){ 2}} \put(1, 7){ \color{red} \line(1, 0){ 2}} \put(1, 9){ \color{red} \line(0, 1){ 2}} \put(1, 9){ \color{red} \line(1, 0){ 2}} \put(1, 11){ \color{red} \line(0, 1){ 2}} \put(1, 11){ \color{red} \line(1, 0){ 2}} \put(1, 13){ \color{red} \line(0, 1){ 2}} \put(1, 13){ \color{red} \line(1, 0){ 2}} \put(1, 15){ \color{red} \line(0, 1){ 2}} \put(1, 15){ \color{red} \line(1, 0){ 2}} \put(1, 17){ \color{red} \line(0, 1){ 2}} \put(1, 19){ \color{red} \line(0, 1){ 2}} \put(1, 19){ \color{red} \line(1, 0){ 2}} \put(1, 21){ \color{red} \line(0, 1){ 2}} \put(1, 21){ \color{red} \line(1, 0){ 2}} \put(1, 23){ \color{red} \line(0, 1){ 2}} \put(1, 23){ \color{red} \line(1, 0){ 2}} \put(1, 25){ \color{red} \line(0, 1){ 2}} \put(1, 25){ \color{red} \line(1, 0){ 2}} \put(1, 27){ \color{red} \line(0, 1){ 2}} \put(1, 27){ \color{red} \line(1, 0){ 2}} \put(1, 29){ \color{red} \line(0, 1){ 2}} \put(1, 29){ \color{red} \line(1, 0){ 2}} \put(1, 31){ \color{red} \line(0, 1){ 2}} \put(1, 33){ \color{red} \line(0, 1){ 2}} \put(3, 1){ \color{red} \line(0, 1){ 2}} \put(3, 1){ \color{red} \line(1, 0){ 2}} \put(3, 3){ \color{red} \line(0, 1){ 2}} \put(3, 3){ \color{red} \line(1, 0){ 2}} \put(3, 5){ \color{red} \line(0, 1){ 2}} \put(3, 5){ \color{red} \line(1, 0){ 2}} \put(3, 15){ \color{red} \line(1, 0){ 2}} \put(3, 17){ \color{red} \line(0, 1){ 2}} \put(3, 21){ \color{red} \line(1, 0){ 2}} \put(3, 23){ \color{red} \line(1, 0){ 2}} \put(3, 25){ \color{red} \line(1, 0){ 2}} \put(3, 27){ \color{red} \line(1, 0){ 2}} \put(5, 1){ \color{red} \line(1, 0){ 2}} \put(5, 3){ \color{red} \line(1, 0){ 2}} \put(5, 7){ \color{red} \line(0, 1){ 2}} \put(5, 7){ \color{red} \line(1, 0){ 2}} \put(5, 9){ \color{red} \line(0, 1){ 2}} \put(5, 11){ \color{red} \line(0, 1){ 2}} \put(5, 11){ \color{red} \line(1, 0){ 2}} \put(5, 13){ \color{red} \line(0, 1){ 2}} \put(5, 13){ \color{red} \line(1, 0){ 2}} \put(5, 17){ \color{red} \line(0, 1){ 2}} \put(5, 19){ \color{red} \line(0, 1){ 2}} \put(5, 19){ \color{red} \line(1, 0){ 2}} \put(5, 21){ \color{red} \line(1, 0){ 2}} \put(5, 25){ \color{red} \line(1, 0){ 2}} \put(7, 5){ \color{red} \line(0, 1){ 2}} \put(7, 7){ \color{red} \line(1, 0){ 2}} \put(7, 9){ \color{red} \line(0, 1){ 2}} \put(7, 11){ \color{red} \line(1, 0){ 2}} \put(7, 15){ \color{red} \line(0, 1){ 2}} \put(7, 17){ \color{red} \line(0, 1){ 2}} \put(7, 17){ \color{red} \line(1, 0){ 2}} \put(7, 23){ \color{red} \line(0, 1){ 2}} \put(7, 23){ \color{red} \line(1, 0){ 2}} \put(9, 1){ \color{red} \line(0, 1){ 2}} \put(9, 3){ \color{red} \line(0, 1){ 2}} \put(9, 3){ \color{red} \line(1, 0){ 2}} \put(9, 5){ \color{red} \line(0, 1){ 2}} \put(9, 9){ \color{red} \line(0, 1){ 2}} \put(9, 9){ \color{red} \line(1, 0){ 2}} \put(9, 11){ \color{red} \line(1, 0){ 2}} \put(9, 13){ \color{red} \line(0, 1){ 2}} \put(9, 15){ \color{red} \line(0, 1){ 2}} \put(9, 19){ \color{red} \line(0, 1){ 2}} \put(9, 19){ \color{red} \line(1, 0){ 2}} \put(9, 21){ \color{red} \line(0, 1){ 2}} \put(11, 1){ \color{red} \line(0, 1){ 2}} \put(11, 3){ \color{red} \line(1, 0){ 2}} \put(11, 5){ \color{red} \line(0, 1){ 2}} \put(11, 5){ \color{red} \line(1, 0){ 2}} \put(11, 7){ \color{red} \line(0, 1){ 2}} \put(11, 7){ \color{red} \line(1, 0){ 2}} \put(11, 9){ \color{red} \line(1, 0){ 2}} \put(11, 13){ \color{red} \line(0, 1){ 2}} \put(11, 13){ \color{red} \line(1, 0){ 2}} \put(11, 15){ \color{red} \line(0, 1){ 2}} \put(11, 15){ \color{red} \line(1, 0){ 2}} \put(11, 17){ \color{red} \line(0, 1){ 2}} \put(13, 1){ \color{red} \line(0, 1){ 2}} \put(13, 1){ \color{red} \line(1, 0){ 2}} \put(13, 5){ \color{red} \line(1, 0){ 2}} \put(13, 7){ \color{red} \line(1, 0){ 2}} \put(13, 9){ \color{red} \line(1, 0){ 2}} \put(13, 11){ \color{red} \line(0, 1){ 2}} \put(13, 13){ \color{red} \line(1, 0){ 2}} \put(13, 15){ \color{red} \line(1, 0){ 2}} \put(15, 1){ \color{red} \line(1, 0){ 2}} \put(15, 3){ \color{red} \line(0, 1){ 2}} \put(15, 5){ \color{red} \line(1, 0){ 2}} \put(15, 7){ \color{red} \line(1, 0){ 2}} \put(15, 11){ \color{red} \line(0, 1){ 2}} \put(15, 11){ \color{red} \line(1, 0){ 2}} \put(15, 13){ \color{red} \line(1, 0){ 2}} \put(17, 3){ \color{red} \line(0, 1){ 2}} \put(17, 7){ \color{red} \line(1, 0){ 2}} \put(17, 9){ \color{red} \line(0, 1){ 2}} \put(17, 11){ \color{red} \line(1, 0){ 2}} \put(19, 1){ \color{red} \line(0, 1){ 2}} \put(19, 1){ \color{red} \line(1, 0){ 2}} \put(19, 3){ \color{red} \line(0, 1){ 2}} \put(19, 3){ \color{red} \line(1, 0){ 2}} \put(19, 5){ \color{red} \line(0, 1){ 2}} \put(19, 5){ \color{red} \line(1, 0){ 2}} \put(19, 7){ \color{red} \line(1, 0){ 2}} \put(19, 9){ \color{red} \line(0, 1){ 2}} \put(19, 9){ \color{red} \line(1, 0){ 2}} \put(21, 3){ \color{red} \line(1, 0){ 2}} \put(21, 5){ \color{red} \line(1, 0){ 2}} \put(21, 7){ \color{red} \line(1, 0){ 2}} \put(21, 9){ \color{red} \line(1, 0){ 2}} \put(23, 1){ \color{red} \line(0, 1){ 2}} \put(23, 5){ \color{red} \line(1, 0){ 2}} \put(23, 7){ \color{red} \line(1, 0){ 2}} \put(25, 1){ \color{red} \line(0, 1){ 2}} \put(25, 1){ \color{red} \line(1, 0){ 2}} \put(25, 3){ \color{red} \line(0, 1){ 2}} \put(25, 5){ \color{red} \line(1, 0){ 2}} \put(27, 3){ \color{red} \line(0, 1){ 2}} \put(27, 3){ \color{red} \line(1, 0){ 2}} \put(29, 1){ \color{red} \line(0, 1){ 2}} \put(29, 1){ \color{red} \line(1, 0){ 2}} \put(31, 1){ \color{red} \line(1, 0){ 2}} \put(33, 1){ \color{red} \line(1, 0){ 2}}}
\newcommand{\yb}{  \hskip -3cm  \usebox{\ybb} \vspace{11mm} \hskip 55pt   }

\newsavebox{\ydb}
\savebox{\ydb}(1,1){ \put(4.5,2){\color{green}\rule{16pt}{16pt}}
	
	\put(0, 0){ \color{blue} \line(0, 1){ 10}} \put(2, 0){ \color{blue} \line(0, 1){ 10}} \put(4, 0){ \color{blue} \line(0, 1){ 6}} \put(6, 0){ \color{blue} \line(0,1){4}} \put(8, 0){ \color{blue} \line(0, 1){ 4}} \put(10, 0){ \color{blue} \line(0, 1){ 2}} \put(0, 0){ \color{blue} \line(1, 0){10}} \put(0, 2){ \color{blue} \line(1, 0){ 10}} \put(0, 4){ \color{blue} \line(1, 0){ 8}} \put(0, 6){ \color{blue} \line(1, 0){ 4}} \put(0, 8){ \color{blue} \line(1, 0){ 2}} \put(0, 10){ \color{blue} \line(1, 0){ 2}}
}

\newcommand{\yd}{  \hskip 5pt  \usebox{\ydb} \vspace{ 11mm} \hskip 45pt   }

\newsavebox{\exoneb}
\savebox{\exoneb}(1,1){
		\put(0, 0){ \color{blue} \line(0, 1){ 4} } \put(2, 0){ \color{blue} \line(0, 1){ 4} } \put(4, 0){ \color{blue} \line(0, 1){ 4} } \put(0, 0){ \color{blue} \line(1, 0){ 4} } \put(0, 2){ \color{blue} \line(1, 0){ 4} } \put(0, 4){ \color{blue} \line(1, 0){ 4} }
	\put(1, 1){ \color{red} \line(0, 1){ 2}}
	\put(1, 1){ \color{red} \line(1, 0){ 2}}
	\put(1, 3){ \color{red} \line(1, 0){ 2}} }
\newcommand{\exone}{  \hskip -6pt  \usebox{\exoneb} \vspace{11mm} \hskip 35pt   }

\newsavebox{\extwob}
\savebox{\extwob}(1,1){
		\put(0, 0){ \color{blue} \line(0, 1){ 4} } \put(2, 0){ \color{blue} \line(0, 1){ 4} } \put(4, 0){ \color{blue} \line(0, 1){ 4} } \put(0, 0){ \color{blue} \line(1, 0){ 4} } \put(0, 2){ \color{blue} \line(1, 0){ 4} } \put(0, 4){ \color{blue} \line(1, 0){ 4} }
	\put(1, 1){ \color{red} \line(0, 1){ 2}}
	\put(3, 1){ \color{red} \line(0, 1){ 2}}
	\put(1, 3){ \color{red} \line(1, 0){ 2}}}
\newcommand{\extwo}{  \hskip -6pt  \usebox{\extwob} \vspace{11mm} \hskip 35pt   }
\newsavebox{\rgb}

\savebox{\rgb}(1,1)
{   \put(4.5,2){\color{blue}\rule{16pt}{16pt}} 
	\put(6.5,4){\color{green}\rule{16pt}{16pt}}
	\put(8.5,6){\color{green}\rule{16pt}{16pt}}
	\put(10.5,6){\color{red}\rule{16pt}{16pt}}
	\put(8.5,4){\color{red}\rule{16pt}{16pt}}
	\put(6.5,2){\color{red}\rule{16pt}{16pt}}
	\put(2.5,0){\color{red}\rule{16pt}{16pt}}
	\put(0.5,0){\color{green}\rule{16pt}{16pt}}
	\put(2.5,2){\color{green}\rule{16pt}{16pt}}
	\put(0, 0) { \color{blue} \line(0, 1) { 12} } \put(2, 0) { \color{blue} \line(0, 1) { 12} } \put(4, 0) { \color{blue} \line(0, 1) { 10} } \put(6, 0) { \color{blue} \line(0, 1) { 10} } \put(8, 0) { \color{blue} \line(0, 1) { 10} } \put(10, 0) { \color{blue} \line(0, 1) { 8} } \put(12, 0) { \color{blue} \line(0, 1) { 8} } \put(14, 0) { \color{blue} \line(0, 1) { 4} } \put(0, 0) { \color{blue} \line(1, 0) { 14} } \put(0, 2) { \color{blue} \line(1, 0) { 14} } \put(0, 4) { \color{blue} \line(1, 0) { 14} } \put(0, 6) { \color{blue} \line(1, 0) { 12} } \put(0, 8) { \color{blue} \line(1, 0) { 12} } \put(0, 10) { \color{blue} \line(1, 0) { 8} } \put(0, 12) { \color{blue} \line(1, 0) { 2} }}

\newcommand{\rg}{  \hskip 5pt  \usebox{\rgb} \vspace{ 11mm} \hskip 45pt   }

\newsavebox{\ppb}

\savebox{\ppb}(1,1)
{   \put(0.5,10){\color{gray}\rule{16pt}{16pt}} 
	\put(2.5,8){\color{gray}\rule{16pt}{16pt}}
	 \put(0.5,8){\color{gray}\rule{16pt}{16pt}}
	 \put(0.5,6){\color{gray}\rule{16pt}{16pt}} 
	\put(0, 0) { \color{blue} \line(0, 1) { 12} } \put(2, 0) { \color{blue} \line(0, 1) { 12} } \put(4, 0) { \color{blue} \line(0, 1) { 10} } \put(6, 0) { \color{blue} \line(0, 1) { 10} } \put(8, 0) { \color{blue} \line(0, 1) { 10} } \put(10, 0) { \color{blue} \line(0, 1) { 8} } \put(12, 0) { \color{blue} \line(0, 1) { 8} } \put(14, 0) { \color{blue} \line(0, 1) { 4} } \put(0, 0) { \color{blue} \line(1, 0) { 14} } \put(0, 2) { \color{blue} \line(1, 0) { 14} } \put(0, 4) { \color{blue} \line(1, 0) { 14} } \put(0, 6) { \color{blue} \line(1, 0) { 12} } \put(0, 8) { \color{blue} \line(1, 0) { 12} } \put(0, 10) { \color{blue} \line(1, 0) { 8} } \put(0, 12) { \color{blue} \line(1, 0) { 2} }}

\begin{document}
\title{Elliptic stable envelope for Hilbert scheme of points in the plane }
\author{Andrey  Smirnov}
\date{}
\maketitle
\thispagestyle{empty}
	
\begin{abstract}
We find an explicit formula for the elliptic stable envelope in the case of the Hilbert scheme of points on a complex plane. The formula has a structure of a sum over trees in  Young diagrams.
In the limit we obtain the formulas for the stable envelope in  equivariant $K$-theory (with arbitrary slope) and equivariant cohomology. 
\end{abstract}
	
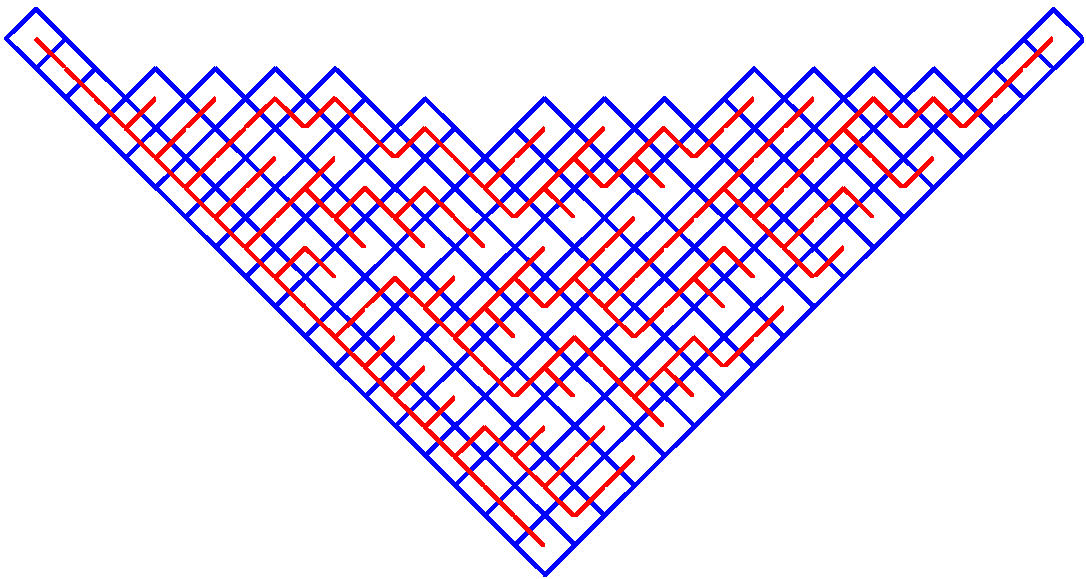
\setcounter{tocdepth}{2} \vspace{5cm}
	
$$
\somespecialrotate[origin=c]{45}{\hskip -4cm \yb}
$$
	
\newpage
\setcounter{tocdepth}{1}
\tableofcontents

\newpage 
\section{Introduction}
\subsection{}
The theory of stable envelopes was originated in \cite{MO} and is playing an increasingly important role in geometric representation theory and enumerative geometry.  It would be no exaggeration to say that stable envelopes are at the heart of many important representation theoretical constructions developed over the past few years. 

Even  before the term ``stable envelope'' was coined, this object already manifested itself under different names in sometimes unrelated areas of mathematics. For instance, the theory of canonical bases, the so called off-shell Bethe vectors \cite{OkBethe,Pushk1,Pushk2} (which are the main object of investigation in the theory of the quantum integrable systems), weight functions for  integral solutions of qKZ equations \cite{varch} are now understood as different incarnations of stable envelopes.   Stable envelopes also appear as partition functions of integrable lattice models \cite{zinnjust} and  play important role in enumerative geometry \cite{pcmilect}. The transition functions between elliptic stable envelopes corresponding to different chambers 
provide the so called quantum dynamical elliptic $R$-matrices. 
In particular, the formulas of this paper provide the elliptic dynamical version of the instanton $R$-matrix studied in \cite{InstR,GenJacks}. 
Though this list of names and applications can be extended, we believe that it is large enough to underline the importance of the object.
\subsection{} 
The elliptic version of stable envelopes was recently defined in 
\cite{AOElliptic} for a class of symplectic varieties known as Nakajima quiver varieties \cite{NakALE,NakQv}. Conjecturally, this construction admits generalization to an arbitrary symplectic resolution. Though the definition of elliptic stable envelope is pretty abstract, \cite{AOElliptic} also contains an \textit{abelianization procedure} for computing it explicitly. 
This procedure works well in a very special case, when the abelianization of a fixed point is zero-dimensional.
In particular, all known explicit examples of elliptic stable envelopes 
correspond to this case. These are hypertoric varieties and cotangent bundles to partial flag varieties of $A_n$-type \cite{varch,konno1,konno2}.
In fact, for cotangent bundles of flag varieties they were known for more than 20 years under the name of elliptic weight functions for solutions of qKZ equations for $\frak{gl}_n$. An example of the elliptic stable envelope outside of this short list is considered in the present paper for the first time.

\subsection{} 
The goal of this paper is to provide an explicit combinatorial description
of the stable envelope for the Hilbert scheme of $n$ points in~$\matC^2$. 
In this case the fixed points (of a torus acting on the Hilbert scheme by automorphisms) are labeled by Young diagrams $\lambda$ with $n$ boxes. The abelianization of a fixed point $\lambda$
 is a non-trivial hypertoric variety $\ahilb_{\lambda}$ and the abelianization procedure of \cite{AOElliptic} becomes ambiguous.

In this article we consider an auxiliary torus $\matC^{\times}$ acting on $\ahilb_{\lambda}$. We will describe a special subset in the finite set of fixed points  $\ahilb_{\lambda}^{\matC^{\times}}$, which is labeled by \textit{trees} in the Young diagram $\lambda$. An example of such a tree can be found on the title page of this article.  Our main result - Theorem \ref{mainth} gives an explicit combinatorial formula for the elliptic stable envelope of a fixed point $\lambda$ as a sum of certain elliptic weights of trees in $\lambda$. 

Let us note that the appearance of the sum over trees in Young diagrams 
is a special feature of the elliptic case. The formulas for cohomological stable envelope  for the Hilbert scheme were obtained by D. Shenfeld  in \cite{Shenfeld}. In Section \ref{shensec} we show that in the cohomological limit the sum over trees trivializes (Proposition~\ref{trsumprop}) and our result reduces to the Shenfeld's formula.    

We note that our result can be generalized in a straightforward way to Nakajima varieties associated with cyclic quivers. We believe that a similar construction of elliptic stable envelopes exists for an arbitrary Nakajima variety. We hope that the results of this article are only  a first step in this direction.

\subsection{} 
Nakajima quiver varieties have a physical interpretation 
as moduli spaces of vacua, also known as Higgs branch of certain $3d$ supersymmetric gauge theories with ${\cal{N}}=4$ supersymmetry, see  \cite{AOF} for discussion. Theories of this kind come in pairs, which are related by important duality known as $3d$-\textit{mirror symmetry} or \textit{symplectic duality}. 

The Higgs branch of the dual theory $X^{\vee}$ conjecturally coincides with the Coulomb branch of the original one. In particular, the $3d$-mirror symmetry acts by permuting the Higgs and Coulomb branches
and exchanging the roles of equivariant and K\"{a}hler parameters of the dual theories.

The new, and the most important feature of the elliptic stable envelope, compared to its cohomological and $K$-theoretic versions, is that it depends on the set of equivariant and K\"{a}hler parameters in a uniform way. This makes the elliptic stable envelope the most natural tool for mathematical description of the $3d$-mirror symmetry.  This circle of ideas was recently outlined by A.Okounkov in his talk ``\textit{Enumerative symplectic duality}'' at MSRI workshop   \textbf{Structures in Enumerative Geometry} in April 2018. 

In particular, the case of the Hilbert scheme $X=\hilb$ which we consider in this paper
is arguably the most important example of symplectic space which is \textit{selfdual}:
$$
\hilb = \hilb^{\vee}.
$$  
The ideas of  $3d$-mirror symmetry imply that our formula 
for the elliptic stable envelope has a very deep,  internal symmetry exchanging the equivariant parameter $a$ with the K\"ahler parameter $z$.

The examples of $3d$-mirror symmetry for elliptic stable envelopes (which were not yet available at the time of the first release of this paper) can be found in \cite{MirSym1,MirSym2}. In particular, the case of cotangent bundles over complete flag varieties of $A_n$ type \cite{MirSym2}
is another interesting example of 3d-selfdual symplectic variety.

We hope that the results of this paper can provide a good tool for testing and proving new results motivated by $3d$-mirror conjecture, in particular, in applications to theory of know invariants \cite{Gala}.
  
\subsection{} 
This paper is organized in the following way. 
In Section \ref{bassec} we recall  basic facts about Nakajima quiver varieties and the elliptic stable envelopes. 

In Section \ref{hilbsec} we define the Hilbert scheme $\hilb$ and recall its description as a Nakajima variety (ADHM construction).

In Section \ref{fromsec} we describe a combinatorial formula for the  elliptic stable envelope for $\hilb$. Our main result is formulated in Theorem~\ref{mainth}.

The following four technical sections are to  prove  Theorem~\ref{mainth}.  

The last two sections are  specializations of our main result to the cases of $K$-theory and cohomology. 
In particular, Section \ref{ksec} contains an explicit expression of equivariant $K$-theoretic stable envelope for arbitrary \textit{slope parameter}, see Theorem \ref{kththeor}. We also describe the set of walls in the space of slopes in Theorem \ref{wallsth}. In the last Section \ref{shensec} we show that in cohomology case our formula possess a new feature: the sum over trees can be evaluated explicitly and the formula for stable envelope can be further simplified. We show that this simplified expression coincides precisely with the cohomological formula obtained previously in \cite{Shenfeld}.

\section*{Acknowledgements}
The author thanks A.Okounkov for uncountable discussions, explanations and suggestions without which  the accomplishment of this project would not be possible. In particular his idea to look at the fixed points corresponding to trees in Young diagrams was the turning point in this work. We would like to thank   A. Osinenko and Y. Kononov for computer checks of the results of the paper and P. Pushkar for reading its preliminary version.  
The author is also grateful to M. Aganagic, I. Cherednik, D. Galakhov, S. Shakirov, A. Varchenko  and   Z. Zhou for discussions at various stages of this project.    

The work is supported in part by RFBR grant 18-01-00926.

\section{Basic facts about elliptic stable envelopes \label{bassec}}

\subsection{} 
In this section we recall the definition of the elliptic stable envelope. For more detailed exposition we refer to the original manuscript \cite{AOElliptic}. 

In \cite{AOElliptic} the elliptic stable envelopes were defined for the Nakajima quiver varieties. Though, it is possible to define these classes in more general setting, we assume that the varieties $X$ discussed in this section are the Nakajima varieties. 
Here we recall the properties of these varieties which are important for the constructions below, see \cite{Nak1,GinzburgLectures} and Section 2 in \cite{MO} for more details.

Let $X$ be a Nakajima variety. Then $X$ is a quasi-projective symplectic variety equipped with a natural linearized action of an algebraic torus~$\bT$. The linearizion means that the quasi-projective embedding may be chosen 
in the form 
\be \label{linear}
X \hookrightarrow \mathbb{P}(\bT-\textrm{module}\ \ V)
\ee
so that the action of $\bT$ on $X$ is induced from the action on the linear space $V$.

The torus $\bT$ acts on $X$ by scaling the symplectic form
$\omega\in H^{2}(X,\matC)$. We denote by $\hbar^{-1} \in \textrm{char}(\bT)$ the character of the one-dimensional $\bT$-module $\matC \omega$.  We denote by $\bA=\textrm{ker}(\hbar^{-1})\subset \bT$ the subtorus preserving the symplectic form.

The Nakajima quiver varieties are examples of symplectic resolutions and thus their cohomology are even \cite{Kaledin}. More generally
\begin{Theorem}[Section 2.3.2 in \cite{AOElliptic}] \label{nakth1}
	If $X$ is a Nakajima quiver variety and $\bT'\subset \bT$ is any subgroup then the fixed locus $X^{\bT'}$ is $\bT$-equivariantly formal, 
	$$
	H^{\bullet}_{\bT}(X^{\bT'})\cong H^{\bullet}(X^{\bT'}) \otimes H^{\bullet}_{\bT}(pt)
	$$	
	and $H^{\bullet}(X^{\bT'})$ is even. 
\end{Theorem}

The Nakajima varieties are defined as quotients by a group $G=\prod\limits_{i\in I} GL(r_i)$, where $I$ denotes the (finite) set of vertices of the corresponding quiver. This means that 
$X$ is naturally equipped with a set of rank $r_i$ tautological vector 
bundles $\tb_{i}$. 

\begin{Theorem}[\cite{kirv}] \label{tmeven}
	If $X$ is a Nakajima variety then $K^{alg}_{\bT}(X)=K^{top}_{\bT}(X)$	is generated by the classes of the tautological bundles $\tb_i$, $i\in I$. 
\end{Theorem}
We will use $K_{\bT}(X)$ to denote the $\bT$-equivariant $K$-theory ring of $X$. Thanks to the last theorem we do not  distinguish between the algebraic and the topological versions.

As a corollary of Theorem \ref{tmeven}, $\textrm{Pic}(X)$ is a finite dimensional lattice generated by the classes of the tautological line bundles $\det(\tb_i)$. The equivariant Picard group $\textrm{Pic}_{\bT}(X)$ is a lattice  given by the extension
\be \label{piclat}
0 \rightarrow \textrm{char}(\bT) \rightarrow \textrm{Pic}_{\bT}(X) \rightarrow \textrm{Pic}(X) \rightarrow 0.
\ee

\subsection{ \label{exsec}}
We consider a family of elliptic curves $E=\matC^{\times}/q^{\matZ}$ parametrized by $0<|q|<1$\footnote{We allow $q$ to be non-generic with $\matZ\subsetneq Hom(E,E)$. }. For an algebraic torus $\bT$ let (we assume $X$ has no odd cohomology here)
$$
\textrm{Ell}_{\bT} : \{ \bT -\textrm{spaces} \, X\} \longrightarrow \{ \textrm{schemes} \}
$$
be the corresponding elliptic cohomology functor such that $\textrm{Ell}_{\matC^{\times}}(pt)=E$. See \cite{ell1,ell2,ell3,ell4,ell5,ell6} for an incomplete list of expositions.  The elliptic cohomology
functor is covariant in both $X$ and $\bT$. In particular, the covarince in $\bT$ implies that 
$$
\cE_{\bT}:=\textrm{Ell}_{\bT}(\textrm{pt})=\bT/q^{\textrm{cochar}(\bT)}\cong E^{\,\dim(\bT)}.
$$
The canonical projection $X\to pt$ provides a map   $\pi :\textrm{Ell}_{\bT}(X) \to \cE_{\bT}$.  
Let $t\in \cE_{\bT}$ and $U_{t}$ be a small analytic neighborhood of $t$, which is isomorphic via exponential map to a small analytic neighborhood in $\textrm{Lie}(\bT)=\matC^{\dim \bT}$.  Locally, the map $\pi$ looks as follows:
\be
\label{abdiag}
\ee
\vspace{-1.3cm}
\[
\xymatrix{
	\textrm{Spec} H_\bT^{\bullet}(X^{\bT_{t}},\matC) \ar[d]  & (\pi)^{-1} (U_t) \ar[l] \ar[r] \ar[d] &  \ar[d]^{\pi } \textrm{Ell}_{\bT}(X)  \\
	\textrm{Lie}(\bT)  & U_t \ar[l] \ar[r] & 
	\cE_{\bT}. }
\]
where all squares are pullbacks and 
$$
\bT_t := \bigcap_{{\chi \in \mathrm{char} (\bT),} \atop {\chi(t) = 0} } \!\!\!\! \ker \chi \subset \bT, 
$$
is the intersection of kernels of all characters 
$$
\chi\in \mathrm{char} (\bT)=\mathrm{Hom}(\cE_{\bT},E)=\mathrm{Hom}(\bT,\matC^{\times})
$$
which are trivial on $t$.  By Theorem \ref{nakth1}, 
$$
H_\bT^{\bullet}(X^{\bT_{t}},\matC) = H^{\bullet}(X^{\bT_{t}},\matC) \otimes H_\bT^{\bullet}(pt)
$$ 
and therefore the fiber of $\pi$ at $t$  can be computed by sending the corresponding equivariant parameters to zero, i.e.:  
$$
\pi^{-1}(t)=\mathrm{Spec}\Big(H^{\bullet}(X^{\bT_{t}},\matC)\Big).
$$

We can use this description to constrict the scheme $\textrm{Ell}_{\bT}(X)$ 
by ``gluing'' the fibers of $\pi$:  for each $U_{t}$ we have an algebra
$$
\left.\mathscr{H} \right|_{U_{t}}:=H_\bT^{\bullet}(X^{\bT_{t}},\matC) \otimes_{H^{\bullet}_{\bT}(pt)}  {\mathcal{O}}^{an}_{U_{t}}
$$
which glue to a sheaf $\mathscr{H}$ of algebras over $\cE_{\bT}$. Then, $\textrm{Ell}_{\bT}(X)=\textrm{Spec}_{\cE_{\bT}}(\mathscr{H})$. 

\vspace{2mm}
\noindent
{\bf Example:}
Let $\bT=(\matC^{\times})^2$ act on $V=\matC^2$ by scaling the coordinates
$$
(x,y) \to (x\, a_1, y \, a_2).
$$ 
Let us consider the induced action of $\bT$ on $X=\mathbb{P}(V)$. In this case $\cE_{\bT}=E\times E$ and we denote by the same symbols $a_1,a_2$ the coordinates on the first and the second factors of $\cE_{\bT}$. 

For  generic point $t=(a_1,a_2) \in \cE_{\bT}$ the fixed 
set $X^{\bT_{t}}$ consist of two points
$$
X^{\bT_{t}}= \{ p_1=[1:0],\ \  p_2=[0:1] \}
$$
and thus the fiber $\pi^{-1}(t)=\mathrm{Spec}(H^{\bullet}(X^{\bT_{t}},\matC))$ is a disjoint union of two points. 

For points on diagonal $t=(a_1,a_2) \in \cE_{\bT}$ with $a_1=a_2$ we have
$X^{\bT_{t}}=X$ and
$$
H^{\bullet}_{\bT}(X,\matC)=\matC[c,\delta a_1,\delta a_2]/(c-\delta a_1)(c-\delta a_2)
$$
where $\delta a_i$'s denote the local coordinates on $\mathrm{Lie}(\bT)$.
The scheme $\textrm{Spec}\Big(H^{\bullet}_{\bT}(X,\matC)\Big)$, describing $\textrm{Ell}_{\bT}(X)$ in the neighborhood of $t$, is given by two intersecting hyperplanes in $\matC^3$ defined by the equations $c=\delta a_1$ and $c=\delta a_2$.

We conclude that $\textrm{Ell}_{\bT}(X)$ can be described as a union  of two intersecting  copies of~$\cE_{\bT}$:
$$
\textrm{Ell}_{\bT}(X)= \Big(\textsf{O}_{p_1}\cup \textsf{O}_{p_2}\Big)/\Delta
$$
where $\textsf{O}_{p_1}\cong \textsf{O}_{p_2}\cong \cE_{\bT}$ and $/\Delta$ denotes the gluing (more precisely, normal crossing) of $\textsf{O}_{p_1}$ and $\textsf{O}_{p_2}$ 
along the diagonal
$$
\Delta=\{(a_1,a_2): a_1=a_2 \}\subset \cE_{\bT}.
$$

\subsection{\label{pisec}}
For the lattice from (\ref{piclat}) we denote
\be
\cE_{\textrm{Pic}_{\bT}(X)}\!:=\textrm{Pic}_{\bT}(X)\otimes_{\mathbb{Z}} E
\ee
and define $\cB_{\bT,X}\!\!:=\cE_{\bT}\times \cE_{\textrm{Pic}_{\bT}(X)}$. We refer to the coordinates in the first factor of the abelian variety 
$\cB_{\bT,X}$  as  \textit{equivariant parameters}  and in the second  as \textit{K{\"a}hler parameters}.
As in the example above we will  often denote the equivariant parameters corresponding to $\bA$ by letters $a_i$, $i=1,\dots, \dim(\bA)$. The K\"ahler parameters will be denoted by 
$z_i$, $i\in I$.

We denote by 
\be \label{extcoh}
\textsf{E}_{\bT}(X):=\textrm{Ell}_{\bT}(X)\times \cE_{\textrm{Pic}_{\bT}(X)}.
\ee
the {\it extended equivariant elliptic cohomology} of $X$. The canonical map $\pi^{*}\times 1$ endows $\textsf{E}_{\bT}(X)$ with a structure of scheme over $\cB_{\bT,X}$.

\subsection{\label{thetsec}}
By definition, $\cB_{\bT,X}$ is an abelian variety isomorphic to some power of $E$. Sections of line bundles over 
$\cB_{\bT,X}$ can be explicitly expressed in terms of the elliptic theta functions associated to $E$ \cite{mumf}. In this paper we use the following multiplicative definition of the theta-function:
\be \label{thet}
\vartheta(x):=\prod\limits_{i=1}^{\infty} (1-x q^{i}) (x^{1/2}-x^{-1/2}) \prod\limits_{i=1}^{\infty} (1-x^{-1} q^{i}).
\ee
This function has the following quasi-period:
\be
\label{quasper}
\vartheta(x q)=-\dfrac{1}{\sqrt{q} x} \vartheta(x).
\ee
and satisfies $\vartheta(1)=0$. It will be convenient to use the following notation
\be
\label{phidef}
\phi(x,z)=\dfrac{\vartheta(x z)}{\vartheta(x) \vartheta(z)}.
\ee
Geometrically, this functions describes a section of the Poincaré line bundle on a product of two dual elliptic curves $E\times E^{\vee}$, i.e., this section transforms as follows:
$$
\phi(x q,z) =z^{-1} \phi(x,z),  \ \ \  \phi(x,z q)=x^{-1} \phi(x,z).
$$

\subsection{}
A rank $r$ complex vector $\tb$ bundle over $X$ defines the \textit{elliptic Chern class} map 
\be \label{chernmap}
c: ~ \textrm{Ell}_{\bT}(X) ~ \rightarrow  \textrm{Ell}_{GL(r)}(pt) = S^{r} E,
\ee
where $S^{r} E $ denotes the $r$-th symmetric power of $E$.  The coordinates on $S^{r} E$ are the symmetric functions in $x_i$, $i=1,\dots,r$ - the elliptic Chern roots of~$\tb$. 
For the definition of $c$ see Section 1.8 in \cite{GKV} or Section 5 in \cite{ell1}.

For Nakajima varieties we  have a map given by the Chern classes of the tautological bundles
$$
c:~\textsf{E}_{\bT}(X)\rightarrow
\mathscr{X}_X:=\cB_{\bT,X} \times \prod\limits_{i \in I} S^{r_i} E.
$$
Theorem \ref{tmeven} implies that this map is an embedding, see Section 2.5 of \cite{AOElliptic} for discussion. 

Many objects in the theory of stable envelopes (such as line bundles on $\textsf{E}_{\bT}(X)$ or  sections of these line bundles) are often introduced in the {\it off-shell} form.  Which means as a pullback from $\mathscr{X}_X$ by~$c^{*}$. We will denote by the superscript  $os$  the global objects living on $\mathscr{X}_X$ to distinguish them from their pullbacks to~$\textsf{E}_{\bT}(X)$ (if it is not clear from a context).  

\subsection{\label{univsec}}
Let $\tb$ be a rank $r$ complex vector bundle over $X$. Let $\Theta(\tb)^{\textrm{os}}$ be a line bundle over $\mathscr{X}_X$ associated to the section  
\be \label{elthm}
s_{\tb}=\prod\limits_{i=1}^{r}\, \vartheta(x_i).
\ee
In other words, $\Theta(\tb)^{\textrm{os}}:=\mathscr{O}(D)$ where $D$ is the divisor given by the zero locus of $s_{\tb}$.  The restriction of this line bundle to $\textsf{E}_{\bT}(X)$
is called the {\it{elliptic Thom class}} of $\tb$\footnote{We refer to Section 7  of \cite{ell1} where definitions of the elliptic Thom sheaf and the elliptic Euler class are discussed.}:
$$
\Theta(\tb):=c^{*}(\Theta(\tb)^{\textrm{os}})\in \textrm{Pic}(\textsf{E}_{\bT}(X)).
$$
Similarly let $\mathscr{U}^{\textrm{os}}$ be a line bundle over $\mathscr{X}_X$ associated to the section
$$
\prod\limits_{i\in I} \phi\Big( \prod\limits_{j=1}^{r_i} x^{(i)}_j, z_{i} \Big).
$$
where $x^{(i)}_{1},\dots, x^{(i)}_{r_i}$ denote the Grothendieck roots of $i$-th tautological bundle. 
The {\it universal line bundle over} $\E_{\bT}(X)$ is defined by: 
$$
\mathscr{U}:=c^{*}(\mathscr{U}^{\textrm{os}}) \in \textrm{Pic}(\textsf{E}_{\bT}(X)).
$$

\subsection{}
Let $\{w\}$ be a set of $\bA$-weights appearing in the normal bundle to $X^{\bA}$ in $X$. The complement of the hyperplanes  
$\textrm{Lie}_{\matR}(\bA)\setminus \{ w^{\perp}\}$ is a set of non-intersecting {\it chambers}. 

For a chamber $\fC$ and a subset 
$S\subset X^{\bA}$ we define its attracting set by
$$
\textrm{Attr}_{\fC}(S)=\{ (s,x), \lim\limits_{\fC} x=s  \} \subset X^{\bA} \times X
$$
where
\be \label{cochard}
\lim\limits_{\fC} x:= \lim\limits_{z\to 0} \sigma(z)\cdot x  
\ee
for a cocharacter $\sigma: z\in \matC^{\times}\to \bA$ 
from the chamber $\fC$. Clearly, this definition does not depend on a choice of $\sigma$. 

The full attracting set $\textrm{Attr}^{f}_{\fC}(S)$ is a minimal closed subset of $X$ which contains $S$ and is closed under taking $\textrm{Attr}_{\fC}(\cdot)$.  

Let $F_i$ denote a connected component of $X^{\bA}$. The choice of a chamber defines an ordering on these components by
\be \label{orderdef}
F_1\geq F_2 \ \ \ \Leftrightarrow \ \ \  \textrm{Attr}^{f}_{\fC}(F_1) \cap F_2 \neq \varnothing.
\ee
This ordering is well defined  by the linerization assumption  (\ref{linear}).  From now we assume that a choice of a chamber $\fC$  is fixed.

\subsection{\label{polsec}}

Recall that a polarization of $X$ is a class $T^{1/2} X \in K_{\bT}(X)$ such that 
$$
TX=T^{1/2} X + \hbar\,  (T^{1/2} X)^{*}.
$$ 
In other words, the polarization is a choice of a ``half'' of the $K$-theory class of the tangents bundle~$TX \in K_{\bT}(X)$. A natural choice of the polarization exists for the Nakajima varieties, see Section~2.2.7 in  \cite{MO}. 

The restriction of the polarization to the fixed set $X^{\bA}$ has the following decomposition
\be \label{fpdecom}
\left.T^{1/2} X\right|_{X^{\bA}} =\left.T^{1/2} X\right|_{X^{\bA},>0} + T^{1/2} X^{\bA} +\left.T^{1/2} X\right|_{X^{\bA},<0}
\ee
into terms whose $\bA$-weights are positive, zero, or negative on the chamber $\fC$. The class 
\be \label{indx}
\textrm{ind} = \left.T^{1/2} X\right|_{X^{\bA},>0} \in K_{\bT}(X^{\bA}) 
\ee
is called  {\it index}.

\subsection{}
Let us consider
$$
\mu \in \textrm{char}(\bT)=\textrm{Hom}(\cE_{\bT},E), \ \ \ \lambda \in \textrm{Pic}_{\bT}(X) = \textrm{Hom}(E,\cE_{\textrm{Pic}_{\bT}(X)}).
$$
Let $\nu:\E_{\bT}(X)\to \cB_{\bT,X}\to \cE_{\bT}$ denotes  the composition of the canonical projections. We consider a map
$$
\tau(\lambda \mu) ~:~ \textsf{E}_{\bT}(X)~\rightarrow~\textsf{E}_{\bT}(X)
$$
given by translation along K\"ahler directions:
$$
\tau(\lambda \mu) ~:(a, z) \mapsto (a, z+\lambda(\mu(\nu(a)))) 
$$
where  $(a,z) \in \Ell_{\bT}(X)\times \cE_{\Pic_{\bT}(X)}= \E_{\bT}(X)$. In words, the map $\tau(\lambda \mu)$ is a translation of K\"ahler parameters depending on the position in $\E_{\bT}(X)$.

\subsection{\label{upsec}} 
For $\det(\textrm{ind}) \in \textrm{Pic}_{\bT}(X^{\bA})$ and $\hbar \in \textrm{char}(\bT)$ we have
$$
\tau(-\hbar \det(\textrm{ind})): \textsf{E}_{\bT}(X^{\bA}) \rightarrow \textsf{E}_{\bT}(X^{\bA}).
$$ 
For the canonical inclusion of the fixed set 
$$i: X^{\bA} \to X$$
let $i^{*}:  \cE_{\textrm{Pic}_{\bT}(X)}\to \cE_{\textrm{Pic}_{\bT}(X^{\bA})}$ be the map induced by the pullback of line bundles. This gives a map
$$
1\times i^{*}: \textrm{Ell}_{\bT}(X^{\bA}) \times \cE_{\textrm{Pic}_{\bT}(X)} \rightarrow \textsf{E}_{\bT}(X^{\bA}).
$$
Following \cite{AOElliptic} we define a line bundle $\mathscr{U}'$
over $\textrm{Ell}_{\bT}(X^{\bA}) \times \cE_{\textrm{Pic}_{\bT}(X)}$:
\be \label{uprd}
\mathscr{U}'= (1\times i^{*})^{*} \tau(-\hbar \,\det \textrm{ind})^{*}  \mathscr{U}_{\textsf{E}_{\bT}(X^{\bA})}
\ee
where $ \mathscr{U}_{\textsf{E}_{\bT}(X^{\bA})}$ is the universal line bundle over  ${\textsf{E}_{\bT}(X^{\bA})}$ from Section \ref{univsec}.

We denote by (a line bundle over the same scheme $\textrm{Ell}_{\bT}(X^{\bA}) \times \cE_{\textrm{Pic}_{\bT}(X)}$):
\be \label{tpr}
\Theta(T^{1/2} X^{\bA})'= (1\times i^{*})^{*} \Theta(T^{1/2} X^{\bA})
\ee
where $T^{1/2} X^{\bA}$ is the middle term in the decomposition (\ref{fpdecom}).

\subsection{}
The  last ingredient we need to define the elliptic stable envelopes is the notion of {\it support}. 
Let $s$ be a section of a coherent sheaf on $\textrm{Ell}_{\bT}(X)$. Let 
$
Y ~ \rightarrow ~ X
$
be an inclusion of a $\bT$-equivariant set. We say that $s$ is supported on $Y$ and write 
$\textrm{supp}(s) \subset Y$ if $f^{*}(s) =0$
for the map
$$
f:~\textrm{Ell}_{\bT}(X\setminus Y)~ \to ~\textrm{Ell}_{\bT}(X)
$$
induced by the inclusion of the complement.

\subsection{}
The elliptic stable envelope (in the normalization accepted in \cite{AOElliptic}) is defined as a map of the $\mathscr{O}_{\cB_{\bT,X}}$ modules constructed in Sections \ref{univsec} and \ref{upsec}.   In particular, the source of (\ref{stdef}) below is a line bundle over  $\textrm{Ell}_{\bT}(X^{\bA}) \times \cE_{\textrm{Pic}_{\bT}(X)}$ 
and the target is a line bundle over $\E_{\bT}(X)$. Both are schemes over $\cB_{\bT,X}$ via canonical projections.

\begin{Theorem}[\cite{AOElliptic}]  \label{defth} If $X$ is a Nakajima variety then there exist {\emph{unique}} map of $\mathscr{O}_{\cB_{\bT,X}}$ modules
	\be \label{stdef}
	\textrm{Stab}_{\fC}: \Theta(\hbar)^{-\textrm{rk}( \textrm{ind})}\otimes \Theta(T^{1/2} X^{\bA})' \otimes \cU^{'}\longrightarrow  \Theta(T^{1/2}X)\otimes \cU 
	\ee
($\textrm{rk}\, (\textrm{ind})$ denotes the rank of the index bundle (\ref{indx})) which is holomorphic in coordinates of  $\cE_\bA$, meromorphic in coordinates of $\cB_{\bT,A}/\cE_\bA$ and satisfies the following conditions: 

($\star$)  If $F \subset X^{\bA}$ is a connected component of the fixed locus and $s$ is a section $\Theta(T^{1/2}X^{\bA})'\otimes \cU^{'}$ with $\textrm{supp}(s)\subset F$ then 
	$\textrm{supp}(\textrm{Stab}_{\fC}(s))\subset \textrm{Attr}^{f}_{\fC}(F)$.

($\star \star$)  If $F \subset X^{\bA}$ is a connected component of the fixed locus and
	$$
	i_{F}: \textrm{Ell}_{\bT}(F) \times \cE_{\textrm{Pic}_{\bT}(X)} \to \textsf{E}_{\bT}(X)
	$$ is the corresponding inclusion map then, for every section $s$ of $\Theta(T^{1/2}X^{\bA})'\otimes\cU^{'}$ with $\textrm{supp}(s)\subset F$ we have:
	$$
	i_{F}^{*} \, \textrm{Stab}_{\fC}(s) = i_{F}^{*} j_{*} m^{*} (s) 
	$$
	where $m^{*}$ and $j_{*}$ are the maps associated to maps induced by the natural projection and inclusion
	$$
	\Ell_{\bT}(F)\times \cE_{\Pic_{\bT}}(X) \stackrel{m}{\longleftarrow}  \Ell_{\bT}(\textrm{Attr}_{\fC}(F))\times \cE_{\Pic_{\bT}}(X)  \stackrel{j}{\longrightarrow} \E_{\bT}(X).
	$$ 
\end{Theorem}

The map $\textrm{Stab}_{\fC}$ is called  {\it elliptic stable envelope map.}

\subsection{\label{finsetsec}} 
Let us assume that $X^{\bA}$ is a finite set, which is the case of our main interest. Then, the source of the map (\ref{stdef}) is a line bundle on 
$|X^{\bA}|$ many copies of~$\cB_{\bT,X}$:
$$
\textrm{Ell}_{\bT}(X^{\bA}) \times \cE_{\Pic_{\bT}(X) } = \coprod\limits_{p\in X^{\bA}} \, \widehat{\Or}_{p} 
$$
with $\widehat{\Or}_{p}\cong \cB_{\bT,X}$.    The line bundle 
$\cU'$ is a collection of line bundles $\cU'_{p}$ on $\widehat{\Or}_{p}$ for $p\in X^{\bA}$.  It is convenient to note that the sections of the line bundle $\cU'_{p}$ have the same transformation properties as 
\be \label{uprtr}
\phi(\hbar, \det \textrm{ind}_{p})\prod\limits_{i  \in I} \phi(\left.\det\tb_i\right|_{p},z_i)
\ee
where $\left.\det\tb_i\right|_{p} \in K_{\bT}(pt)$ denotes the restriction of the  $i$-th tautological line bundle to $p\in X^{\bT}$, see Lemma 2.4 \cite{AOElliptic} .

To give a map (\ref{stdef}) is, therefore, the same as to construct a map
$$
\Theta(\hbar)^{-\textrm{rk}(\textrm{ind}_p)}\otimes \Theta(T^{1/2} X_{p}^{\bA})' \otimes \cU'_{p}~\longrightarrow~\Theta(T^{1/2}X)\otimes \cU 
$$
for each fixed point $p$, which satisfies the conditions ($\star$) and $(\star,\star)$. Equivalently, to construct a section $\Stab_{\fC}(p)$ of the line bundle
\be \label{fpline}
\Theta(\hbar)^{\textrm{rk}( \textrm{ind}_p)}\otimes \pi^{*}( \Theta(T^{1/2} X_{p}^{\bA})' \otimes \cU'_{p})^{-1}\otimes \Theta(T^{1/2}X)\otimes \cU
\ee
over $\E_{\bT}(X)$, where $\pi: \E_{\bT}(X) \to \cB_{X,\bT}\cong \widehat{\Or}_{p}$ is the canonical projection.  By abuse of language, we will refer to the section $\Stab_{\fC}(p)$ as the {\it elliptic stable envelope of a fixed point}~$p$.


Next, let us consider the scheme $\textsf{E}_{\bT}(X)$. 
Arguing as in example of Section \ref{exsec}, we obtain:
\be \label{etx}
\textsf{E}_{\bT}(X)=\Big(\coprod\limits_{p\in X^{\bA}}\, \widehat{\textsf{O}}_{p}\Big)/\Delta
\ee
where $\widehat{\textsf{O}}_{p} \cong \cB_{\bT,X}$ and $\Delta$ denotes the gluing of these abelian varieties over certain hyperplanes in $\cE_{\bT}$.  The restriction to one of the components of (\ref{etx}) 
$$
T_{p,q}:=\left.\Stab_{\fC}(p)\right|_{\widehat{\textsf{O}}_{q}}
$$
is a section of  (\ref{fpline})  over an abelian variety $\widehat{\textsf{O}}_{q}$ (more precisely of restriction of (\ref{fpline}) to the corresponding component).
In particular, $T_{p,q}$ can be described explicitly in terms of the theta functions, as discussed in Section \ref{thetsec}.  
The collection $T_{p,q}$ is usually called the {\it matrix of restrictions} of the elliptic stable envelopes.  

The property $(\star)$ in Theorem \ref{defth} implies that the matrix  $T_{p,q}$ is triangular 
if $X^{\bA}$ is ordered by (\ref{orderdef}). The property $(\star\star)$ is a condition on the diagonal of this matrix
\be  \label{diad}
T_{p,p}=\prod\limits_{{w \in \textrm{char}_{\bT}(T_p X)}\atop {\langle w, \fC\rangle<0}} \vartheta(w)
\ee
where the product runs over the $\bT$-weights of the tangent space $T_{p}X$ which are negative on the chamber $\fC$ (i.e., the repelling directions of the tangent space).

\subsection{\label{offshelstab}}
As a concluding comment for this section we note that the known explicit formulas for the  stable envelopes (cohomological, K-theoretic or elliptic) are traditionally given in the off-shell form.  Namely, instead of the section of the line bundle (\ref{fpline}) one constructs a section $\Stab_{\fC}^{os}(p)$ of the line bundle 
\be \label{fplineof}
\Theta(\hbar)^{\textrm{rk}(\textrm{ind}_p)}\otimes \pi^{*}(\Theta(T^{1/2} X_{p}^{\bA})' \otimes\cU'_{p})^{-1}\otimes \Theta(T^{1/2}X)^{os}\otimes \cU^{os}
\ee
over $\cX_{X}$ such that the elliptic stable envelope of a fixed point $p$ is the restriction of this section to $\E_{\bT}(X)$:
$$
\Stab_{\fC}(p)=c^{*} \,\Stab_{\fC}^{os}(p).
$$
The scheme $\cX_{X}$ is a (symmetric) power of $E$ and every such section can be described through the theta functions as a certain symmetric combination of the elliptic Chern roots 
of the tautological bundles. The stable envelopes of the fixed points presented in the off-shell form are also known as {\it weight functions} in literature,
see for example \cite{MirSym2} and references there.

The main result of this paper is a combinatorial formula for $\Stab_{\fC}^{os}(p)$ for $X$ given by the Hilbert scheme of points in $\matC^2$, see Theorem \ref{mainth} below.

\section{Hilbert scheme of points on $\matC^2$ \label{hilbsec}} 

\subsection{\label{torin}} 
Let us denote by $\hilb$ the Hilbert scheme of $n$ points on 
the complex plane. This is a smooth, symplectic, quasiprojective variety which parametrizes 
the polynomial ideals of codimension $n$:   
$$
\hilb=\{ {\cal{J}} \subset \matC[x,y] :\,  \dim_{\matC} (\matC[x,y]/ {\cal{J}}) =n \}. 
$$	
Let $\bT\cong (\matC^{\times})^2$ be a two-dimensional torus acting on $\matC^{2}=Spec(\matC[x,y])$ by scaling the coordinates:
\be
\label{polact}
(x,y)\to (x t_1^{-1},y t_2^{-1})
\ee
This action induces and action of $\bT$ on  $\hilb$. The one-dimensional space spanned by a symplectic form $\matC \omega \subset H^{2}(\hilb,\matC)$ is a natural $\bT$-module. We denote by $\hbar^{-1}$ the  $\bT$-character of $\matC \omega$. From our normalization (\ref{polact}) we find:
$$
\hbar = t_1 t_2 \in K_{\bT}(pt)=\matZ[t_1^{\pm 1}, t_2^{\pm 1}].
$$
We denote by 
$$
\bA=\ker(\hbar^{-1}) \subset \bT 
$$
the one-dimensional subtorus preserving the symplectic form on $\hilb$. We denote the coordinate on $\bA$ by $a$ such that\footnote{We may assume $\hbar^{1/2}$ exists by passing to the double cover of $\bT$ if needed. }:
\be \label{ttoa}
t_1 =a \hbar^{1/2}, \ \ t_2=a^{-1} \hbar^{1/2}.
\ee

\subsection{}
The set of fixed points $\hilb^{\bT}=\hilb^{\bA}$ is finite set labeled by partitions of $n$. 
A partition $\lambda=(\lambda_1,\lambda_2,\dots,\lambda_{l})$ with $\lambda_i\geq \lambda_{i+1}$ and $|\lambda|=\sum_{i=1}^{l} \lambda_i=n$ corresponds to the $\bT$-invariant ideal generated by monomials:
$$
{\cal{J}}_{\lambda}=\{x^{\lambda_1}, x^{\lambda_2} y, x^{\lambda_3} y^2, \dots,y^{l}\}.
$$
The equivariant $K$-theory of the Hilbert scheme has an explicit presentation:
\be
\label{kthh}
K_{\bT}(\hilb)=\matZ[x_1^{\pm 1},\dots,x_n^{\pm 1},t_1^{\pm 1},t_2^{\pm 1}]^{\frak{S}_n}/R
\ee
where the superscript $\frak{S}_n$ denotes the ring of symmetric Laurent polynomials in $x_i$ and $R$ is the ideal of polynomials vanishing at all fixed points $\lambda \in \hilb^{\bT}$. Below we describe the restriction of a Laurent polynomial in $x_i$'s and $t_i$'s to a fixed point labeled by $\lambda$.

Given a partition $\lambda$ we may think of it as a Young diagram with $n$ boxes. 
A box $\Box\in \lambda$ in the Young diagram with coordinates $(i,j)$ corresponds to the monomial $y^{i-1} x^{j-1}$. We denote  by $\varphi^{\lambda}_{\Box}$ the $\bT$-character of the one-dimensional space spanned by this monomial. From (\ref{polact}) we have:
\be \label{boxchar}
\varphi^{\lambda}_{\Box}=t_1^{-(j-1)} t_2^{-(i-1)} \in K_{\bT}(pt).
\ee
For a partition $\lambda$ we thus have $n$ monomials
$\varphi^{\lambda}_{\Box_1},\dots, \varphi^{\lambda}_{\Box_n}$ corresponding boxes of $\lambda$.  Let $f(x_1,\dots,x_n,t_1,t_2)$ be a polynomial representing a $K$-theory class from (\ref{kthh}). The restriction of this class  to a $\bT$-fixed point corresponding to a partition $\lambda$ is given by:
\be \label{restric}
i^{*}_{\lambda} f(x_1,\dots,x_n,t_1,t_2) = f(\varphi^{\lambda}_{\Box_1},\dots, \varphi^{\lambda}_{\Box_n},t_1,t_2) \in K_{\bT}(pt)
\ee 
where $i_{\lambda}: \lambda \to \hilb$ is the canonical inclusion of a fixed point.  The polynomial $f(x_1,\dots,x_n,t_1,t_2)$ is symmetric in $x_i$'s, so this definition does not depend on a choice of order on the set $\varphi^{\lambda}_{\Box_1},\dots, \varphi^{\lambda}_{\Box_n}$.

Denote by $\tb$ the rank $n$ tautological bundle on $\hilb$ whose fiber at a point ${\cal{J}} \in \hilb$ is defined to be $\matC[x,y]/{\cal{J}}$. As an element of (\ref{kthh}) this bundle is represented by a polynomial
\be
\label{tautbun}
\tb = x_1 + \cdots + x_n,
\ee
which means that the variables $x_i$ are the Grothendieck roots of $\tb$. We also have
\be
\label{olin}
{\mathscr{O}}(1) =\det \tb =x_1\cdots x_n.
\ee
This line bundle generates the Picard group of $\hilb$ so that 
\be \label{pich}
\textrm{Pic}(\hilb)=\matZ.
\ee
In this paper we also use the equivariant Picard group which is an extension
\be \label{picex}
\textrm{char}({\bT}) \rightarrow \textrm{Pic}_{\bT}(\hilb) \rightarrow  \textrm{Pic}(\hilb)
\ee
where the two-dimensional lattice $\textrm{char}({\bT})$ is generated by trivial line bundles
associated to $\bT$-characters $a$ and $\hbar$. 
 
\subsection{\label{hilbnak}}
The Hilbert scheme $\hilb$ is an example of a Nakajima quiver variety.
Let us outline this description of $\hilb$ here. 
More detailed expositions can be found in \cite{NakajimaLectures1, GinzburgLectures}.

For $V=\matC^n$ we set
\be \label{mrep}
M= \{(I,X)\} =\textrm{Hom}_{\matC}(\matC,V)  \bigoplus \textrm{Hom}_{\matC}(V,V). 
\ee
The vector space $T^{*}M$ is a space of matrices 
$$
\begin{array}{l}
T^{*}M=\\
\ \ \ \ \Big\{(X,Y,I,J) : X,Y \in \textrm{Hom}_{\matC}(V,V), I\in  \textrm{Hom}_{\matC}(\matC,V), J \in  \textrm{Hom}_{\matC}(V,\matC) \Big\}.
\end{array}
$$ 
This vector space carries a  natural $GL(V)$-action 
\be \label{gact}
X\to g X g^{-1}, \ \ Y\to g Y g ^{-1}, \ \ I \to  g I, \ J \to  J g^{-1}.
\ee
Let us consider the following symplectic reduction:
$$
 T^{*}M/\!\!/\!\!/\!\!/GL(V) = \mu^{-1}(0)\!/\!\!/_{\!\!\theta}GL(n)=\mu^{-1}(0)^{\theta-ss}/GL(V),
$$
where $\mu: T^{*}M \to \frak{gl}(n)^{*}$ denotes the moment map
corresponding to (\ref{gact}). In coordinates (after identification of vector spaces $\frak{gl}(n)^{*}=\frak{gl}(n)$) the moment map takes the form 
$$
\mu(X,Y,I,J)=[X,Y]+I\otimes J.
$$ 
The symbol $\!/\!\!/_{\!\!\theta}$ denotes the GIT quotient with a stability parameter $\theta \in \textrm{char}(GL(V))$. In this paper we use the following choice of stability parameter:
\be
\label{stabchoice}
\theta: g \to \det(g)^{-1}.
\ee
Finally, $\mu^{-1}(0)^{\theta-ss}$ stands for the intersection of the set 
$\mu^{-1}(0)$ with the set of $\theta$-semistable points. This set has the following  description. Recall that a vector $v \in V$ is called {\it{cyclic}} vector of  $ X,Y \in End(V)$ if 
$
\matC\langle X,Y \rangle v= V
$
where $\matC\langle X,Y \rangle$ denotes the set of $\matC$-polynomials in (possibly  non-commuting) variables $X,Y$. 

\begin{Proposition}[Proposition 5.6.5,  \cite{GinzburgLectures}]
The set of $\theta$-semistable points equals
$$
\begin{array}{l}
\mu^{-1}(0)^{\theta-ss}=\\
 \ \ \ \ \ \ \ \Big\{ (X,Y,I,J) :\, [X,Y]=0, \ \ J=0, \ \ I(1) \,\, \textrm{is a cyclic vector of}\, \, (X,Y)\Big\}.
\end{array}
$$ 	
\end{Proposition}
For any commuting pair $X,Y \in \frak{gl}(n)$ and a vector $v\in V$ we can consider the following set of polynomials:
$$
{\cal{J}}_{X,Y,v}=\{p(x,y)\in \matC[x,y]: p(X,Y) v =0\} \subset \matC[x,y].
$$
It is clear that ${\cal{J}}_{X,Y,v}$ is a polynomial ideal. This ideal  is invariant with respect to transformations (\ref{gact}) and one proves the following result. 

\begin{Proposition}[Corollary 5.6.8,\cite{GinzburgLectures}] 
The assignment $(X,Y,I) \to {\cal{J}}_{X,Y,I(1)}$ establishes an isomorphism between 
$T^{*}M/\!\!/\!\!/\!\!/GL(V)$ and $\hilb$.	 
\end{Proposition}


\subsection{}
For $\hilb$ the \textit{polarization} can be taken in the form:
\be
\label{polhilb}
T^{1/2}\hilb=\tb +t_1  \tb^{*}\, \tb  -  \tb^*\,\tb \in K_{\bT}(\hilb)
\ee 
where  $*$ stands for taking the duals in $K$-theory
$$
\tb^*=x_1^{-1}+\dots+x_{n}^{-1}\in K_{\bT}(\hilb).
$$
The $K$-theory class of the tangent bundle to $\hilb$ equals:
\be
\label{tanbun}
T\hilb= T^{1/2}\hilb+T^{1/2}\hilb^{*}  \hbar = \tb+\tb^{*} t_1 t_2 -(1-t_1)(1-t_2) \tb \tb^*. 
\ee 

\vspace{3mm}
\noindent 
{\bf Example:} Let us illustrate the above formulas in the case $n=1$. 
In this case $K_{\bT}(\hilb)=\matZ[x_1^{\pm 1},t_1^{\pm 1},t_2^{\pm 1}]/R$. By 
(\ref{boxchar}) and (\ref{restric}) we have $R=\{ x_1=1\}$ and thus
$$K_{\bT}(\hilb)=\matZ[t_1^{\pm 1},t_2^{\pm 1}].$$ 
The class of tautological bundle $\tb=x_1=1$ and thus from (\ref{tanbun}) we obtain:
$$
T\hilb=t_1+t_2.  
$$
In full agreement with the fact $\hilb \cong\matC^2$ for $n=1$.

\subsection{}
The torus $\bT$ is two-dimensional thus
$$
\cE_{\bT}\cong E^2
$$
and $t_1,t_2$ are the corresponding coordinates (equivariant parameters) . Similarly, by (\ref{pich}) we have
$$
\cE_{\textrm{Pic}(\hilb)}=\textrm{Pic}(\hilb)\otimes_{\mathbb{Z}} E \cong E
$$  
and we denote the corresponding coordinate by $z$  (the K\"{a}hler parameter).
\subsection{}
The torus $\bA$ is one-dimensional and we have the following chamber decomposition:
$$
\textrm{Lie}_{\matR}\setminus \{ \omega^{\perp} \} = \matR\setminus \{0\}=\fC_{+}\cup \fC_{-}.
$$
We choose the chamber $\fC=\fC_{+}$, corresponding to the cocharacters $\{a\to 0\}$ in (\ref{cochard}). For this chamber the ordering (\ref{orderdef}) on the set $\hilb^{\bA}$  coincides
with standard dominance ordering on the partitions.

\subsection{}
The elliptic stable envelope of a fixed point is described by a section $\Stab^{os}_{\fC}(\lambda)$ of a line bundle (\ref{fplineof}) for $X=\hilb$. Let us discuss the transformation properties of these sections.

Recall that $\cU^{os}$ is  a line bundle on $\mathscr{X}_{\hilb}$ 
whose sections transforms as
\be 
\label{univ}
\phi(x_1\cdots x_n , z)  
\ee
From (\ref{polhilb}) and (\ref{tautbun}) we see that the polarization bundle is represented by the following Laurent polynomials in $K$-theory:
\be
\label{hpol}
T^{1/2}\hilb=\sum\limits_{i=1}^{n} x_i + (t_1-1) \sum\limits_{i,j=1}^{n} \dfrac{x_i}{x_j}
\ee  
Thus,  $\Theta(T^{1/2}\hilb)^{os}$ is the line bundle on  $\mathscr{X}_{\hilb}$ whose sections transforms
as\footnote{This expression is singular   because of $\vartheta(1)$-factors in the denominator. We, however, are only interested in its quasi-periods which are well defined.}:
\be
\label{pol}
\prod\limits_{i=1}^{n}\vartheta(x_i) \prod\limits_{i,j=1}^{n} \dfrac{\vartheta(x_i/x_j t_1)}{\vartheta(x_i/x_j)} 
\ee
see Section \ref{univsec} for the definitions.

\subsection{}
By (\ref{restric}) we have:
\be \label{lamres}
i^{*}_{\lambda} \Big(T^{1/2}\hilb\Big) = \sum\limits_{i\in \lambda} \varphi^{\lambda}_{i} +
\sum\limits_{i,j\in \lambda} \dfrac{\varphi^{\lambda}_{i} t_1}{\varphi^{\lambda}_{j}} - \sum\limits_{i,j\in \lambda} \dfrac{\varphi^{\lambda}_{i} }{\varphi^{\lambda}_{j}}  \in K_{\bT}(pt).
\ee
To compute the index (\ref{indx}) of a fixed point $\lambda$ 
we use (\ref{ttoa}) to pick from this sum only the terms with positive powers of $a$ (which corresponds to
the attracting directions for the chamber $\fC$):
$$
\textrm{ind}_{\lambda}=\sum\limits_{{i \in \lambda}\atop {c_i>0}} \varphi^{\lambda}_{i}+\sum\limits_{{i,j \in \lambda}\atop {c_{i}-c_{j}+1>0}} \dfrac{\varphi^{\lambda}_{i} t_1}{\varphi^{\lambda}_{j}} -
\sum\limits_{{i,j \in \lambda}\atop {c_i-c_j>0}}\, \dfrac{\varphi^{\lambda}_{i} }{\varphi^{\lambda}_{j}} \in K_{\bT}(pt)
$$
where $c_i$ is the content of the box $i\in \lambda$ defined in (\ref{condef}). Thus
$$
\det \textrm{ind}_{\lambda}= \prod\limits_{{i\in \lambda} \atop {c_{i}>0}} \varphi^{\lambda}_{i} \prod\limits_{{i,j\in\lambda} \atop {c_{i}\geq c_{j}}} t_1.
$$
and 
$$
\textrm{rk}_{\lambda}:=\textrm{rk}(\textrm{ind}_{\lambda})=\left.\textrm{ind}_{\lambda}\right|_{t_1=t_2=1}=\sum\limits_{{i \in \lambda}\atop {c_i>0}} 1+ \sum\limits_{{i,j \in \lambda}\atop {c_{i}=c_{j}}} \, 1 \in \matZ.
$$ 
Choosing the terms in (\ref{lamres}) independent of $a$ we find:
$$
T^{1/2} \hilb^{\bA}_{\lambda}= \sum\limits_{{i \in \lambda}\atop {c_i=0}} \varphi^{\lambda}_{i}+\sum\limits_{{i,j \in \lambda}\atop {c_{i}-c_{j}+1=0}} \dfrac{\varphi^{\lambda}_{i} t_1}{\varphi^{\lambda}_{j}} -
\sum\limits_{{i,j \in \lambda}\atop {c_i-c_j=0}}\, \dfrac{\varphi^{\lambda}_{i} }{\varphi^{\lambda}_{j}}.
$$
\subsection{}   
From (\ref{fplineof}) we conclude that $\Stab^{os}_{\fC}(\lambda)$ (as the function of all variables) transforms as the following expression:
$$
\begin{array}{ll}
\prod\limits_{i=1}^{n}\vartheta(x_i) \prod\limits_{i,j=1}^{n} \dfrac{\vartheta(x_i/x_j t_1)}{\vartheta(x_i/x_j)} \times & \textrm{for} \ \  \Theta(T^{1/2} \hilb)^{os}  \\
\\
\prod\limits_{{i\in \lambda} \atop {c_{i}=0}}\vartheta(\varphi^{\lambda}_i)^{-1} \prod\limits_{{i,j\in \lambda} \atop {{c_{i}-c_{j}+1=0}}} {\vartheta(\varphi^{\lambda}_i/\varphi^{\lambda}_j t_1)}^{-1}  
\prod\limits_{{i,j\in \lambda} \atop {{c_{i}=c_{j}}}} {\vartheta(\varphi^{\lambda}_i/\varphi^{\lambda}_j )} 
\times & \textrm{for} \ \  \pi^{*}( \Theta(T^{1/2} \hilb^{\bA}_{\lambda})')^{-1}
\\
\\
\phi(x_1\cdots x_n,z)\times \ \ &  \textrm{for} \ \ \cU^{os}    \\
\\
\phi(\varphi^{\lambda}_1 \cdots  \varphi^{\lambda}_n ,z)^{-1} \phi(\det \mathrm{ind}_{\lambda},\hbar)^{-1} \times  & \textrm{for} \ \  \pi^{*}(\cU_{\lambda}')^{-1} \ \   (\textrm{see} \ \ (\ref{uprtr}) )\\ 
\\
\vartheta(\hbar)^{\textrm{rk}_{\lambda}} & \textrm{for} \ \ \Theta(\hbar)^{\textrm{rk}(\textrm{ind}_\lambda)}.
\end{array} 
$$
The quasi-period in the K\"{a}hler parameter $z$ has the form:
\be \label{ztrans} 
\left.\textrm{Stab}^{os}_{\fC}(\lambda)\right|_{z=z q} =\Big(\prod\limits_{i\in \lambda}\dfrac{\varphi^{\lambda}_{i}}{x_i} \Big)\,  \textrm{Stab}^{os}_{\fC}(\lambda).
\ee
Note also that 
\be \label{xtrans}
\left.\textrm{Stab}^{os}_{\fC}(\lambda)\right|_{x_i=x_i q} =\Big(-\dfrac{1}{\sqrt{q} x_i z} \Big)\,  \textrm{Stab}^{os}_{\fC}(\lambda).
\ee
\subsection{} The restriction matrix,
\be \label{resmat}
T_{\lambda \mu }(a,\hbar,z)= \left.\textrm{Stab}_{\fC}(\lambda)\right|_{\widehat{\Or}_{\mu}}=i^{*}_{\mu}\Big(\textrm{Stab}^{os}_{\fC}(\lambda)\Big)
\ee 
for the substitution $i^{*}_{\mu}$ given by (\ref{restric}), is triangular with respect to standard dominance ordering on partitions. The repelling weights (\ref{diad}) of the tangent space  $T_{\lambda} \hilb$ corresponding to the chamber $\fC$ are easy to compute explicitly
\be \label{diael}
T(a,\hbar,z)_{\lambda \lambda}=\prod\limits_{\Box\in\lambda} \vartheta(t_1^{-l_{\lambda}(\Box)} t_2^{a_{\lambda}(\Box)+1}),
\ee
where $a_{\lambda}(\Box)$ and $l_{\lambda}(\Box)$ stand for the standard arm and leg length of a box $\Box\in \lambda$.\footnote{ 
If a box $\Box=(i,j)$ they are defined as 
$$
 l_{\lambda}(\Box)=\lambda_i-j, \ \  a_{\lambda}(\Box)=\lambda^{'}_j-i
$$
where $\lambda^{'}$ is the transposition of $\lambda$. Note that these functions are well defined even if $\Box \notin \lambda$.
} In particular the diagonal elements of the restriction matrix do not depend on the K\"{a}hler parameter $z$. 

It is obvious from (\ref{ztrans}), that as a function of $z$ the restriction matrix transforms as
$$
T(a,\hbar, z q) = Mat_{{\mathscr{O}}(1)}  T(a,\hbar, z ) Mat_{{\mathscr{O}}(1)}^{-1}
$$ 
where $ {\mathscr{O}}(1)$ is the diagonal matrix of multiplication by the corresponding line bundle 
in $K_{\bT}(\hilb)$ in the basis of fixed points:
\be \label{omat}
Mat_{{\mathscr{O}}(1)}=\textrm{diag}\Big(\, \prod\limits_{i \in \lambda} \varphi^{\lambda}_{i} \, \Big).
\ee
After discussing all these general properties of the elliptic stable envelope, we are now ready to give the explicit combinatorial formula for it. 

\section{Formula for elliptic stable envelope \label{fromsec}}
\subsection{} 

Let $\lambda=(\lambda_1 \geq \dots \geq \lambda_k)$ be a partition of $n=\lambda_1+\cdots+\lambda_k$ and $l(\lambda)=k$.
It will be convenient to use the standard representations for a partition as a Young diagram in its French and Russian
(a diagram rotated by $45^{\circ}$) realization as in the Fig.\ref{fry}.

For a box in a partition $\Box \in \lambda$ we write $\Box=(i,j) \in \matZ_{+}^2$ for its coordinates in the corresponding Young diagram. The \textit{content} and \textit{height} of a box is defined by 
\be \label{condef}
c_\Box=i-j \ \ \ h_\Box=i+j-2
\ee 
respectively. They represent the horizontal and vertical position of a box in Russian representation.
We also note that from (\ref{ttoa}) we have:
$$
\varphi^{\lambda}_{\Box}=a^{c_\Box} \hbar^{-h_\Box/2} 
$$
\vspace{2cm}
\begin{figure}[h!]
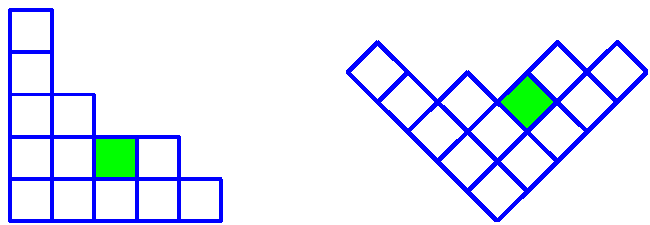

	\hskip 25mm \yd \hskip 30mm \somespecialrotate[origin=c]{45}{\yd}
	\caption{The French and Russian representation of a partition $\lambda=(5,3,2,2,1)$. The box $(3,2)\in \lambda$
		is denoted with color. \label{fry}}
\end{figure}
We will say that two boxes $\Box_1, \Box_2 \in \lambda$ are \textit{adjacent} if
$$i_{1}=i_{2}, |j_{1}-j_{2}|=1 \ \  \textrm{or} \ \ j_{1}=j_{2}, |i_{1}-i_{2}|=1.$$
Let us define the following important function of a box:
\be
\label{refc}
\rho_\Box=c_\Box-\epsilon h_\Box
\ee
where $0<\epsilon<<1$ is a infinitely small parameter \footnote{For a function $\rho_\Box$ to define the correct ordering on boxes it is enough to choose $\epsilon <\dfrac{1}{n}$. }. The value of this function defines a \textit{canonical ordering} on the boxes of $\lambda$. It orders the boxes from the left to right
and from the top to the bottom in the Russian Young diagram as in the Fig.\ref{boxord}.
\begin{figure}[h!]
	\hskip 50mm
	\begin{tikzpicture}[draw=blue]
	\draw [line width=1pt] (-7,0) -- (-9,2);
	\node [left] at (-6.75,0.5) {$5$};
	\node [left] at (-6.75,1.5) {$4$};
	\node [left] at (-7.75,1.5) {$2$};
	\node [left] at (-5.75,1.5) {$7$};
	\node [left] at (-6.25,1) {$6$};
	\node [left] at (-7.25,1) {$3$};
	\node [left] at (-8.25,2) {$1$};
	\draw [line width=1pt] (-9,2) -- (-8.5,2.5);
	\draw [line width=1pt] (-8,1) -- (-7,2);
	\draw [line width=1pt] (-8.5,1.5) -- (-8,2);
	\draw [line width=1pt] (-6,1) -- (-7,2);
	\draw [line width=1pt] (-5.5,1.5) -- (-6,2);
	\draw [line width=1pt] (-7.5,0.5) -- (-6,2);
	\draw [line width=1pt] (-6.5,0.5) -- (-8.5,2.5);
	\draw [line width=1pt] (-7,0) -- (-5.5,1.5);
	\end{tikzpicture}
	\caption{Canonical box ordering in a partition $\lambda=(4,2,1)$. \label{boxord}}
\end{figure}
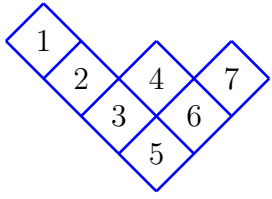

\noindent
It is convenient to introduce the following abbreviation for a combination of theta functions depending on a partition $\lambda$:
\be \label{shenpart}
\begin{aligned}
 &\textbf{S}^{Ell}_{\lambda}(x_1,\dots,x_n): =  \\ 
&  \dfrac{\prod\limits_{{\rho_j>\rho_i+1}} \vartheta(x_i x_j^{-1} t_1) \prod\limits_{{\rho_j<\rho_i+1}} \vartheta(x_j x_i^{-1} t_2) \prod\limits_{\rho_i\leq 0} \vartheta(x_i)
\prod\limits_{\rho_i>0} \vartheta(t_1 t_2 x_i^{-1})}{\prod\limits_{\rho_i<\rho_j} \vartheta(x_i x_j^{-1})\vartheta(x_i x_j^{-1} \hbar)}
\end{aligned}
\ee
where $i$ and $j$ run over the boxes of $\lambda$. 

\vspace{2mm}
\noindent 
\textbf{Example:}
$$
\begin{array}{l}
\textbf{S}^{Ell}_{[1]}(x_1)= \vartheta \left( t_{{2}} \right) \vartheta \left( x_{{1}} \right)\\
\\
\textbf{S}^{Ell}_{[1,1]}(x_1,x_2)=\dfrac{\vartheta(t_2)^2 \vartheta(x_2 x_1^{-1} t_2) \vartheta(x_1 x_2^{-1} t_2) \vartheta(x_1) \vartheta(t_1 t_2 x_2^{-1})}{\vartheta(x_1 x_2^{-1}) \vartheta(x_1 x_2^{-1} t_1 t_2)} 
\\
\\
\textbf{S}^{Ell}_{[2]}(x_1,x_2)=\dfrac{\vartheta(t_2)^2 \vartheta(x_1 x_2^{-1} t_2) \vartheta(x_1 x_2^{-1} t_1) \vartheta(x_1) \vartheta( x_2)}{\vartheta(x_1 x_2^{-1}) \vartheta(x_1 x_2^{-1} t_1 t_2)}
\end{array} 
$$

\subsection{}
\begin{Definition}
	A $\lambda$-tree is a rooted tree with
	
	($\star$) set of vertices given by the boxes of a partition $\lambda$,

	($\star,\star$) root at the box $r=(1,1)$,
	
	($\star,\star,\star$) edges connecting only the adjacent boxes.
\end{Definition}
Note that the number of $\lambda$-trees depends on a shape of $\lambda$. In particular, there is exactly one
tree for ``hooks''  $\lambda=(\lambda_1,1,\dots,1)$.

We assume that each edge of a $\lambda$-tree is oriented in a certain way. In particular, on a set of edges we have a well defined functions
$$ 
h,t : \textrm{edges of a tree} \longrightarrow \textrm{boxes of} \ \  \lambda, 
$$ which for an edge $e$ return its head $h(e)\in \lambda$ and tail $t(e)\in \lambda$  boxes respectively. In this paper we will work with a distinguished \textit{canonical orientation} on $\lambda$-trees.  
\begin{Definition}
We say that a $\lambda$-tree has canonical orientation if all edges
are oriented from the root to the leaves (the end points)  of the tree.	
\end{Definition}

For a box $i\in \lambda$ and a canonically oriented $\lambda$-tree
 $\textsf{t}$ we have a well defined canonically oriented subtree $\textsf{t}_i \subset \textsf{t}$ with a root at $i$.  Such that, for example $\textsf{t}_r=\textsf{t}$ for a root $r$ of $\textsf{t}$.  We denote
$$
[i,\textsf{t}]= \{\textrm{boxes in} \ \  \textsf{t}_i\} \subset \lambda.
$$

\subsection{}
Let $e$ be an edge of a $\lambda$-tree $\textsf{t}$. The $z$-weight of $e$ is a natural number $\textsf{w}_{e}$
which is to be defined recursively using the following relations:
\be
\label{wweight}
\textsf{w}_{e}=1+\sum\limits_{{e^{\prime}\in \textsf{t}:}\atop{t(e^{\prime})=h(e)}} \textsf{w}_{e^{\prime}}
\ee
where all edges are canonically oriented.
In other words:
$$
\textsf{w}_{e}= |[h(e),\textsf{t}]|.
$$
In particular it implies that
\be \label{westim}
1 \leq \textsf{w}_{e}\leq n-1.
\ee
We also note that 
\be \label{wsum}
n=|[r,\textsf{t}]|=1+\sum\limits_{t(e)=r} \textsf{w}_{e}
\ee
where $r=(1,1)$ is the root box.

\subsection{}
The $\hbar$-weight of an edge $e$ is a natural number $\textsf{v}_{e}$ defined by the same type of recursive relations:
\be
\label{vweight}
\textsf{v}_{e}=\beta_{\lambda}(h(e))+\sum\limits_{{e^{\prime}\in \textsf{t}:}\atop{t(e^{\prime})=h(e)}} \textsf{v}_{e^{\prime}}
\ee
where $\beta_{\lambda}(\Box)$ is a function of a box $\Box$ taking value in the set $\{0,1\}$. To define it, let us consider the upper boundary of the Young diagram in the Russian presentation.
There are four possibilities for a behaviour of the boundary at a given point:
the boundary has maxima or minima at this point or the boundary increases or decreases at this point.
We set:
\be \label{betefun}
\beta_{\lambda}(\Box)=\left\{\begin{array}{cc}
0&\textrm{if the boundary above $\Box$ has maxima or decreases, }\\
1&\textrm{if the boundary above $\Box$ has minima or increases. }\end{array} \right.
\ee
For example, the Fig.\ref{betastat} gives the values of $\beta_{\lambda}(\Box)$ for $\lambda=(4,4,4,3,3,2)$.
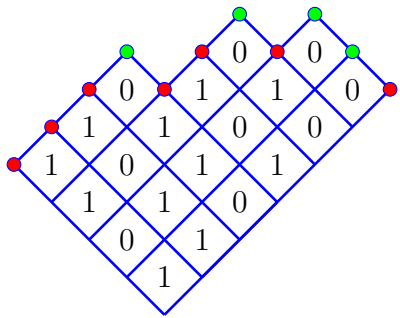
\begin{figure}[h!]
	\hskip 50mm
	\begin{tikzpicture}[draw=blue]
	\draw [line width=1pt] (-7,0) -- (-9,2);
	\node [left] at (-6.75,0.5) {$1$};
	\node [left] at (-6.75,1.5) {$1$};
	\node [left] at (-6.75,2.5) {$1$};
	\node [left] at (-5.75,1.5) {$0$};
	\node [left] at (-5.75,2.5) {$0$};
	\node [left] at (-5.75,3.5) {$0$};
	\node [left] at (-6.25,1) {$1$};
	\node [left] at (-6.25,2) {$1$};
	\node [left] at (-6.25,3) {$1$};
	\node [left] at (-7.25,1) {$0$};
	\node [left] at (-7.25,2) {$0$};
	\node [left] at (-7.25,3) {$0$};
	\node [left] at (-7.75,1.5) {$1$};
	\node [left] at (-7.75,2.5) {$1$};
	\node [left] at (-5.25,3) {$1$};
	\node [left] at (-5.25,2) {$1$};
	\node [left] at (-4.75,3.5) {$0$};
	\node [left] at (-4.75,2.5) {$0$};
	\node [left] at (-4.25,3) {$0$};

	\node [left] at (-8.25,2) {$1$};
	\draw [line width=1pt] (-9,2) -- (-7.5,3.5);
	\draw [line width=1pt] (-8,1) -- (-5,4);
	\draw [line width=1pt] (-8.5,1.5) -- (-6,4);

	\draw [line width=1pt] (-7.5,0.5) -- (-4.5,3.5);
	\draw [line width=1pt] (-6.5,0.5) -- (-4,3);
	\draw [line width=1pt] (-7,0) -- (-5.5,1.5);

	\draw [line width=1pt] (-6.5,0.5) -- (-8.5,2.5);
	\draw [line width=1pt] (-6,1) -- (-8,3);
	\draw [line width=1pt] (-5.5,1.5) -- (-7.5,3.5);
	\draw [line width=1pt] (-5,2) -- (-6.5,3.5);
	\draw [line width=1pt] (-4.5,2.5) -- (-6,4);
	\draw [line width=1pt] (-4,3) -- (-5,4);

	\draw [blue,fill=red] (-4,3) circle (0.5ex);
	\draw [blue,fill=green] (-4.5,3.5) circle (0.5ex);
	\draw [blue,fill=green] (-5,4) circle (0.5ex);
	\draw [blue,fill=red] (-5.5,3.5) circle (0.5ex);
	\draw [blue,fill=green] (-6,4) circle (0.5ex);
	\draw [blue,fill=red] (-6.5,3.5) circle (0.5ex);
	\draw [blue,fill=red] (-7,3) circle (0.5ex);
	\draw [blue,fill=green] (-7.5,3.5) circle (0.5ex);
	\draw [blue,fill=red] (-8,3) circle (0.5ex);
	\draw [blue,fill=red] (-8.5,2.5) circle (0.5ex);
	\draw [blue,fill=red] (-9,2) circle (0.5ex);

	\end{tikzpicture}
	\caption{Values of function $\beta_{\lambda}(\Box)$ for the diagram $\lambda=(4,4,4,3,3,2)$.
		The points where the boundary of the Young diagram increases or has minima are denoted by red circle.
		The points where the boundary decreases or has maximum are denoted by green circles.
		By definition, $\beta_{\lambda}(\Box)=1$ below red and $\beta_{\lambda}(\Box)=0$ below green circles.
		\label{betastat}}
\end{figure}
Similarly to (\ref{wsum}) we also define
$$
\textsf{v}_{r}=\beta_{\lambda}(r)+\sum\limits_{t(e)=r} \textsf{v}_{e}
$$

\subsection{} 
For a tree $\textbf{t}$ we define a function:
\be \label{wpartell}
\textbf{W}^{Ell}({\textbf{t}}; x_1,\dots,x_n,z)=(-1)^{\kappa_{\textbf{t}}} \phi(x_r,z^n \hbar^{\textsf{v}_{r}}) \prod\limits_{e\in \textbf{t}} \phi\Big(\dfrac{x_{h(e)} \varphi^{\lambda}_{t(e)}}{\varphi^{\lambda}_{h(e)} x_{t(e)}}, z^{\textsf{w}_{e} } \hbar^{\textsf{v}_{e}}\Big).
\ee
where $\kappa_{\textbf{t}}$ is the number of edges in $\textbf{t}$ directed  ``wrong way'', i.e., the number of vertical edges directed down plus number of horizontal edges directed to the left in the canonical orientation. 

\vspace{2mm}
\noindent
\textbf{Example:} 
Let us compute the function 
$$\textbf{W}^{Ell}\Big(\exone,x_1,x_2,x_3,x_4,z\Big).
$$
First, let us order boxes using the function $\rho$, i.e. as in the Fig.\ref{boxord}. The root box is 
$x_r=x_3$. From (\ref{boxchar}) we compute:
$$
\varphi^{\lambda}_1=t_1^{-1}, \varphi^{\lambda}_2=t_1^{-1}t_2^{-1}, \varphi^{\lambda}_3=1, \varphi^{\lambda}_4=t_2^{-1}.
$$
The indicated $\lambda$-tree has three edges with
$$
h(e_1)=3, t(e_1)=1; \ \  h(e_2)=1, t(e_2)=2; \ \ h(e_3)=3, t(e_3)=4;
$$ 
We also note that  $\kappa_{\textbf{t}}=0$. 
From (\ref{wweight}) and (\ref{vweight}) we compute
$$
\textsf{w}_{e_1}=2, \ \ \textsf{w}_{e_2}=1, \ \ \textsf{w}_{e_3}=1,
$$ 
and
$$
\textsf{v}_{e_1}=1, \ \ \textsf{v}_{e_2}=0, \ \ \textsf{v}_{e_3}=0.
$$
and thus $\textsf{v}_{r}=\textsf{v}_{e_1}+\textsf{v}_{e_3}=1$.
In the sum, we obtain:
$$
\begin{array}{ll}
\textbf{W}^{Ell}\Big(\exone,x_1,x_2,x_3,x_4,z\Big)=\\ \phi(x_3,z^4\hbar) \phi( x_1 x_3^{-1} t_1,z^2\hbar) 
 \phi({x_2}{x_1^{-1}} t_2,z)  \phi({x_4}{x_3^{-1}} t_2,z).
 \end{array}
$$
Similarly, one can compute:
$$
\begin{array}{ll}
\textbf{W}^{Ell}\Big(\extwo,x_1,x_2,x_3,x_4,z\Big)=\\ -\phi(x_3,z^4\hbar) \phi( x_1 x_3^{-1} t_1,z^3\hbar) 
\phi({x_2}{x_1^{-1}} t_2,z^2)  \phi({x_4}{x_2^{-1}} t_1^{-1},z).
\end{array}
$$
where the sign comes from and the edge with $t(e)=2$ and $h(e)=4$ which is a vertical edge directed down.
\begin{Proposition} \label{Wtrans}
	If $\textbf{t}$ is a $\lambda$-tree then the weight function has the following quasi-periods:  	
	$$
	\textbf{W}^{Ell}({\textbf{t}}; x_1,\dots,x_n,z q) = \Big(\prod\limits_{i \in \lambda} \, \dfrac{\varphi^{\lambda}_i}{ x_i} \Big)  \, \textbf{W}^{Ell}({\textbf{t}}; x_1,\dots,x_n,z) 
	$$
\end{Proposition}
\begin{proof} The function (\ref{phidef}) has the quasiperiod $\phi(x,z q)= x^{-1} \phi(x,z)$. Thus, from the definition (\ref{wpartell}) we have:
$$
\textbf{W}^{Ell}({\textbf{t}}; x_1,\dots,x_n,z q)=\Big(\prod\limits_{\Box\in \lambda}\, \dfrac{\varphi^{\lambda}_{\Box}}{x_{\Box}} 
\Big)^{m_{\Box}}\, \textbf{W}^{Ell}({\textbf{t}}; x_1,\dots,x_n,z)
$$	
for some integers $m_{\Box}$. Let us fix a box $\Box\in \lambda$. The only factors in (\ref{wpartell}) which contribute to $m_{\Box}$ correspond to the following edges of the $\lambda$-tree ${\textbf{t}}$:
 
$\bullet$ the edge $e$ with $h(e)=\Box$. The quasiperiod of
the factor of (\ref{wpartell}) corresponding to $e$ contributes $\textsf{w}_{e}$ to $m_{\Box}$. 

$\bullet$ the edges $e^{'}$ with  $t(e)=\Box$. Each such edge contributes $-\textsf{w}_{e}$ to $m_{\Box}$. 

\noindent
Overall we have
$$
m_{\Box}=\textsf{w}_{e}-\sum\limits_{{e'\in {\textbf{t}}} \atop {t(e')=\Box}}\,\textsf{w}_{e'} =1
$$
where the last equality is by (\ref{wweight}).

\end{proof}
Thus, these functions have the same quasi-periods as the elliptic stable envelope (\ref{ztrans}). 
This is exactly why we need them - the K\"{a}hler parameter $z$ enters to
the formula for elliptic stable envelope through these weight functions, see Theorem \ref{mainth}.

\subsection{ \label{upsion}} 
\begin{Definition}
	The skeleton of a partition $\Gamma_{\lambda}$ is the graph with set of vertices given by the set of boxes
	of $\lambda$ and set of edges connecting all adjacent boxes.
\end{Definition}

\begin{Definition} \label{lsdef}
	A $\reflectbox{\textsf{L}}$  -  shaped subgraph in $\lambda$ is a subgraph $\gamma\subset \Gamma_{\lambda}$ consisting of two edges
	$\gamma=\{\delta_1,\delta_2\}$ with the following end boxes:
	\be 
	\label{gshapped} \delta_{1,1}=(i,j), \ \ \delta_{2,1}=\delta_{1,2}=(i+1,j), \ \ \delta_{2,2}=(i+1,j+1).
	\ee
\end{Definition}
It is easy to see that the total number of  \reflectbox{\textsf{L}} - shaped subgraphs in $\lambda$ is equal to 
\be
\label{nofsh}
m=\sum\limits_{k \in \matZ}( \textsf{d}_{k}(\lambda)-1),
\ee
where $\textsf{d}_{k}(\lambda)$ is the number of boxes in $\lambda$ with content $k$
\be
\label{dfun}
\textsf{d}_{i}(\lambda)=\# \{i \in \lambda| c_i =k \}.
\ee 
There is a special set of $2^m$ $\lambda$-trees which can be constructed as follows.  For each $i=1\dots m$ choose one of two edges of $\gamma_i$. We have $2^m$ of such choices, that is the set:
$$
\Upsilon_\lambda=\{ (\delta_{1},\dots,\delta_{m} ): \delta_i \in \gamma_i \}. \ \ \
$$
If $\delta \in \Upsilon_\lambda$ then, it is easy to see that 
$$
\textsf{t}_{\delta}=\Gamma_\lambda \setminus \delta
$$
is a $\lambda$-tree. Indeed, the graph $\Gamma_\lambda$ has loops corresponding to $2\times 2$ squares in the Young diagram  $\lambda$. In $\Gamma_\lambda \setminus \delta$ we are removing exactly one edge for each such loop. 
Thus, $\textsf{t}_{\delta}$ is a connected graph without loops, i.e., is a tree.  Connectedness of this tree guarantees that we can find path in this tree connecting the root box $r$ with any other box $\Box\in \lambda$. This means that $\textsf{t}_{\delta}$ is a $\lambda$-tree.

 An example of a tree $\textsf{t}_{\delta}$ obtained this way can be found at the front page of this article. 

\subsection{}
With all this preparations, we can now formulate our main result.
\begin{Theorem} \label{mainth}
The off-shell elliptic stable envelope of a fixed point $\lambda\in \hilb^{\bA}$ equals
\be \label{ellipticenvelope}
\begin{array}{|c|}\hline
	\\
\ \ \ \textrm{Stab}^{os}_{\fC}(\lambda) = \textrm{Sym}\Big( \textbf{S}^{Ell}_{\lambda}(x_1,\dots,x_n) \sum\limits_{\delta\in \Upsilon_\lambda } \textbf{W}^{Ell}({\textbf{t}}_{\delta};x_1,\dots,x_n; z) \Big) \ \ \\\
\\
\hline
\end{array}
\ee
where the symbol Sym stands for symmetrization over variables $x_1,\dots,x_n$. 
\end{Theorem}
This theorem is proved in Section \ref{profsec}. 

Let us recall that the function (\ref{ellipticenvelope}) describes a section of the line bundle (\ref{fplineof})  over the scheme~$\cX_{\hilb}$.  The elliptic stable envelope of a fixed point is the restriction of this section $\textrm{Stab}_{\fC}(\lambda)=c^{*}\Big(\textrm{Stab}^{os}_{\fC}(\lambda)\Big)$ to $\E_{\bT}(\hilb)$ embedded to this scheme by the elliptic Chern class map:
$$
c: \E_{\bT}(\hilb) \to \cX_{\hilb}. 
$$

By Proposition \ref{Wtrans} this function has the expected quasi-periods (\ref{ztrans}). 
A direct calculation shows that (\ref{xtrans}) also holds.  Other expected properties of the matrix $T_{\lambda,\mu}$ such as its triangularity and (\ref{diael}) can also be checked by a direct calculation. 

\section{Elliptic stable envelope for hypertoric varieties}

Here we recall several basic facts about geometry of hypertoric varieties. For a fuller exposition we refer to \cite{ProudHyper}
and Section 3 of \cite{Shenfeld}. 
\subsection{}

Let $M$ be a complex vector space and $\bG=(\matC^{\times})^{\dim M}$ acts on $M$ by scaling the coordinates in some basis $e_i$.  Let $R_{S}:S \rightarrow\bG$ be a subtorus and
\be
\label{hypert}
X=T^*M /\!\!/\!\!/\!\!/S =\mu^{-1}(0)^{\theta-ss}/S
\ee
be the corresponding symplectic reduction, see Section 3.1 in \cite{Shenfeld}. We assume that $S$ acts freely on $\mu^{-1}(0)^{\theta-ss}$ and thus $X$ is a smooth symplectic hypertoric variety.  
Let  ${\cal{L}}_{e_i}$ be the line bundles over $X$ associated to the coordinate lines  $M_i\subset M$.  These line bundles generate the Picard group and $K$-theory of $X$. Similarly let ${\cal{L}}_{e_i^*}$ be the symplectic dual line bundles which satisfy
\be \label{dualls}
{\cal{L}}_{e_i}\otimes {\cal{L}}_{e_i^*}=\hbar \in \textrm{Pic}_{\matC^{\times}_{\hbar}}(X).
\ee 
As usual $\matC^{\times}_{\hbar}$ denotes the torus acting on $X$ by scaling the symplectic form. 
\subsection{}
For hypertoric varieties $X^{\bG}$ is always finite. A fixed point $\textbf{t}\in X^{\bG}$ defines the following data on the prequotient:  
\begin{itemize}
	\item A decomposition:
	\be \label{fpdec}
	M=M^{0}\oplus M^1
	\ee 
	with $\dim(M^0)=\dim(S)$  such that
	$$
	\textbf{t}=T^*M^0 \cap \mu^{-1}(0)^{\theta-ss}/S.
	$$
	\item A ``compensating'' map:  $R_{\textbf{t}}: \bG \rightarrow S$ such that a composed representation
	\be
	\label{compz}
	g \to g\cdot (R_{S} \circ R_{\textbf{t}}) (g)   
	\ee 
	contains $M^0$ as a trivial subrepresentation. In other words the action
	of $\bG$ on $M^0$ can be ``compensated'' by the action of $S$.  
\end{itemize}

A representative of a fixed point $\textbf{t}$ is a semi-stable vector
\be
\label{repvector}
r_{\textbf{t}}=\sum_{i\in A} \alpha_{i} e_i + \sum_{j\in B} \beta_{j} e_j^{*} \in T^*M_0 \cap \mu^{-1}(0)^{\theta-ss} 
\ee
with exactly $\dim(S)$ non-zero components, i.e., $|A|+|B|=\dim(S)$ and $A\cap B=\emptyset$. The non-intersecting sets of indexes $A,B$ are uniquely determined by the choice of $\textbf{t}$ and the stability condition $\theta$.

\subsection{}
All weights of (\ref{compz}) in $M^1$ are nontrivial. We denote by $m_i \in \textrm{char}(\bG\times  \matC^{\times}_{\hbar})$ their characters. From the definitions of ${\cal{L}}_{e_i}$ with $M_{i}\subset M$ we have
\be \label{restr1}
\left.{\cal{L}}_{e_i}\right|_{\textbf{t}} = m_i \in K_{\bG\times  \matC^{\times}_{\hbar}}(pt),
\ee
from (\ref{dualls}) we also have
\be \label{restr2}
\left.{\cal{L}}_{e_i^{*}}\right|_{\textbf{t}} = \hbar/m_i \in K_{\bG\times \matC^{\times}_{\hbar}}(pt).
\ee
Finally, for $X=T^*M /\!\!/\!\!/\!\!/S$  the $\bG\times \matC^{\times}_{\hbar}$-weights of the tangent space $T_{\textbf{t}} X$ are the non-trivial weights of (\ref{compz}) appearing in $T^*M$:
\be \label{hyperchar}
{\rm{char}}(T_{\textbf{t}} X) = \sum_{M_i\subset M^{1}} m_i +\hbar m_i^{-1} \in K_{\bG\times \matC^{\times}_{\hbar}}(pt).
\ee

\subsection{}
As explained in Section \ref{offshelstab}
the off-shell elliptic stable envelope of a fixed point is a section of line bundle (\ref{fplineof})  over $\cX_{X}$ satisfying certain defining conditions. In the hypertoric case the scheme $\cX_{X}$ is isomorphic to a power of $E$ (because all tautological bundles are of rank one). The coordinates on $\cX_{X}$ are equivariant parameters of $\bG\times \matC^{\times}_{\hbar}$, the K\"ahler parameters and elliptic Chern roots of the tautological line bundles ${\cal{L}}_{e_i}$. 
The elliptic stable envelopes can, therefore, be expressed as 
(through the theta function $\vartheta$)  certain functions of all these parameters. 
By abuse of notations we will denote by the same symbol ${\cal{L}}_{e_i}$ the elliptic Chern root of the line bundle ${\cal{L}}_{e_i}$ (i.e. the coordinate in the target of (\ref{chernmap})). 

\begin{Theorem} \label{abstabth}
	For a chamber $\fC \subset \textrm{Lie}_{\matR}(\bG)$
	the elliptic stable envelope of a fixed point $\textbf{t}\in X^{\bG}$ has the form:
	\be
	\label{hyperstabs}
	\textrm{Stab}^{os}_{\fC}(\textbf{t}) =   \textbf{S}_{\textbf{t}}  \textbf{W}_{\textbf{t}}
	\ee
	with
	\be
	\label{shgen}
	\textbf{S}_{\textbf{t}} =\prod\limits_{{M_{i}\subset M^1 } \atop {\langle  m_i,\fC \rangle < 0} } \vartheta({\cal{L}}_{e_i}) 
	\prod\limits_{{M_{i}\subset M^1 } \atop {\langle  m_i,\fC \rangle  > 0} } \vartheta({\cal{L}}_{e_i^*})
	\ee
	and
	\be
	\label{wgen}
	\textbf{W}_{\textbf{t}}= \prod\limits_{i \in A} \dfrac{\vartheta({\cal{L}}_{e_i} m_i^{-1} l_i)}{\vartheta(l_i)} \prod\limits_{i \in B} \dfrac{\vartheta( {\cal{L}}_{e_i^*} m_i \hbar^{-1} l_i)}{\vartheta( l_i)} 
	\ee
	where the sets $A$ and $B$ are as in (\ref{repvector}). The symbols $l_i$ stand for  monomials in K\"{a}hler parameters and $\hbar$ chosen such that (\ref{hyperstabs}) has the same qusiperiods as  sections of (\ref{fplineof}). 
\end{Theorem}
\begin{proof} The proof of this theorem is given by Proposition 4.1 in \cite{AOElliptic}. Here we give slightly different argument.  
	
	The set $X^{\bG}$ is finite, thus we are in the situation discussed in Section \ref{finsetsec}. 
	
	The condition $(\star,\star)$ from the definition of the elliptic stable envelope in this case  has the form (\ref{diad}):
	$$
	\left.\textrm{Stab}_{\fC}(\textbf{t})\right|_{\textbf{t}} =\prod\limits_{w \in \textrm{char} (T_{\textbf{t}}X) \atop {\langle w,\fC\rangle }<0} \vartheta(w)
	$$
	Substituting  (\ref{restr1}), (\ref{restr2}) to (\ref{hyperstabs}) we see that 
	$$
	\left.\textrm{Stab}_{\fC}(\textbf{t})\right|_{\textbf{t}} = 
	\prod\limits_{{M^1_{i}\subset M^1 } \atop {\langle  m_i,\fC \rangle < 0} } \vartheta(m_i) 
	\prod\limits_{{M^1_{i}\subset M^1 } \atop {\langle  m_i,\fC \rangle  > 0} } \vartheta(\hbar/m_i)
	$$
	which is exactly what we need by (\ref{hyperchar}). Thus (\ref{hyperstabs}) satisfies the condition $(\star,\star)$. 
	
	The condition $(\star)$ in the definition of elliptic stable envelope is the condition on its support. This condition is the same for all elliptic, K-theoretic or cohomological versions of the stable envelopes \cite{MO}. Thus, it is enough to check that (\ref{hyperstabs}) has correct support in the cohomological limit. 
	
	In this limit (see Section \ref{shensec} of this paper) the K\"ahler part $\textbf{W}_{\textbf{t}}$ becomes trivial.  To compute the cohomological limit of $\textbf{S}_{\textbf{t}}$  we need to substitute each factor $\vartheta({\cal{L}}_{e_i})$ by the corresponding cohomological first Chern class. We denote
	$$
	c_{1}({\cal{L}}_{e_i}) = u_i \in H^{\bullet}_{\bG\times \matC^{\times}_{\hbar}}(X)
	$$	
	the first Chern class of the tautological line bundle. The condition (\ref{dualls}) means that 
	$$
	c_{1}({\cal{L}}_{e_i^{*}}) = h-u_i \in H^{\bullet}_{\bG\times\matC^{\times}_{\hbar} }(X)
	$$
	where  $h=c_{1}(\hbar)$ for $\hbar\in \Pic_{\matC^{\times}_{\hbar}}(X)$. 
	We conclude that the cohomological limit of (\ref{hyperstabs}) has the form
	$$
	\prod\limits_{{M^1_{i}\subset M^1 } \atop {\langle  m_i,\fC \rangle > 0} } u_i 
	\prod\limits_{{M^1_{i}\subset M^1 } \atop {\langle  m_i,\fC \rangle  < 0} } h-u_i \in H_{\bG\times \matC^{\times}_{\hbar}}(X).
	$$ 
	Comparing it with  Theorem 3.3.5 in \cite{Shenfeld} we see that this is exactly the cohomological stable envelope of a fixed point in cohomology, and thus has correct support.  
	Thus, the condition $(\star)$ is also satisfied by  (\ref{hyperstabs}). 
	The condition on $l_{i}$ means that $\textrm{Stab}_{\fC}(\textbf{t})$
	is a section of the correct line bundle. The result follows from uniqueness of the elliptic stable envelopes. 
\end{proof}

\section{Abelianization of Hilbert Scheme}
Let $M$ be a $n^2+n$-dimensional vector space (\ref{mrep}) spanned by the  matrix elements:
\be \label{matel}
M=\bigoplus_{i,j=1}^{n} X_{ij} \oplus \bigoplus_{i=1}^{n} I_{i}
\ee
Recall that the Hilbert scheme $\hilb$ can be defined as a symplectic reduction of $M$ by the action of $GL(V)$, 
see Section \ref{hilbnak}:
$$
\hilb = T^{*} M/\!\!/\!\!/\!\!/GL(V).
$$
  The {\it abelianization of the Hilbert scheme} $\hilb$ is, by definition,  the hypertoric variety
given by the symplectic reduction by a maximal torus $S\subset GL(V)$: 
$$
\ahilb = T^{*} M/\!\!/\!\!/\!\!/S=\mu_S^{-1}(0)^{\theta-ss}/S,
$$
for the same choice of the stability parameter (\ref{stabchoice}). This hypertoric variety was first considered in Section 6 of \cite{Shenfeld}. 
The torus $S$ acts on $V$ by
$$
\left\{\left(\begin{array}{cccc}
x_1&&&\\
&x_2&&\\
&&.&\\
&&&.
\end{array}\right) \right\} \subset \textrm{End}(V).
$$
We denote by the same symbols $x_1,\dots, x_n \in \textrm{Pic}(\ahilb)$ the tautological line bundles associated to the corresponding 
one-dimensional representations of $S$. These line bundles generate $K$-theory of $\ahilb$. We also denote by $z_1,\dots, z_n$ the corresponding coordinates on $\cE_{\textrm{Pic}_{\bT}(\ahilb)}$, i.e. the dual K\"{a}hler parameters.    

Let ${\cal{L}}_{X_{i j}}$, ${\cal{L}}_{I_{i}} \in \Pic_{\bT}(\ahilb)$ be the tautological line bundles over $\ahilb$ associated to matrix elements $(\ref{matel})$. In terms of the tautological line bundles they and their duals have the following form
\be
\label{LtoX}
{\cal{L}}_{X_{i j}}=\dfrac{x_i}{x_j} t_1, \ \ \ {\cal{L}}_{I_{i}}=x_i; \ \ {\cal{L}}_{Y_{i j}}=\dfrac{x_i}{x_j} t_2, \ \ \ {\cal{L}}_{J_{i}}=x_i^{-1} \hbar
\ee 
Such that 
$$
{\cal{L}}_{X_{i j}}\otimes {\cal{L}}_{Y_{j i}}=\hbar, \ \ {\cal{L}}_{I_{i}}\otimes {\cal{L}}_{J_{i}} =\hbar.
$$
The $\theta$-semistable points in this case have the following  description \cite{Shenfeld}: 
\begin{Proposition} \label{abstab} A point $(X,Y,I,J)\in T^*M$ is $\theta$-semistable if and only if any subspace of $V$ containing $\textrm{im}(I)$ and stable under $X$ and $Y$ is not contained in any coordinate hyperplane. 
\end{Proposition}

\subsection{\label{restrsec}}
An important difference between $\hilb$ and $\ahilb$ is that the set of $\bA$-fixed points for the last one is not necessarily finite.  Indeed, recall that the torus $\bA$ acts on the prequotient 
by $X_{i j}\rightarrow X_{ij} a$, $I_{i}\rightarrow I_{i}$.
Let $\lambda\in \hilb^{\bA}$ be a fixed point. The corresponding ``compensating'' map $\bA \rightarrow GL(V)$ is of the form:
\be
\label{AtoS}
a \rightarrow \textrm{diag}\Big( a^{c_1},\cdots,a^{c_n} \Big).
\ee
where we abbreviate $c_k=c_{\Box_k}$ for the content of $k$-th box in $\lambda$ defined by~(\ref{condef}).

We conclude that the total torus action on $X_{i j}$ and $I_i$ is given by the following formula:
\be
\label{Aact}
X_{ij} \rightarrow X_{i j} a^{c_i-c_j+1}, \ \ I_{i} \rightarrow I_{i} a^{c_i}. 
\ee
Thus, the  hypertoric subvariety 
\be 
\label{Ah} 
\ahilb_{\lambda} \subset \ahilb
\ee
defined by
$$
\ahilb_{\lambda} = T^{*} M_{\lambda}/\!\!/\!\!/\!\!/S,
$$
for
\be
\label{ML}
M_{\lambda}=\bigoplus_{c_{j}=c_{i}+1} X_{i j} \oplus \bigoplus_{c_i=0} I_{i} \subset M
\ee
is $\bA$-fixed. We see that
$$
\dim \ahilb_{\lambda}= 2 \dim M_{\lambda} -2 n \geq 0,
$$
with $\dim \ahilb_{\lambda}=0$ if and only if  $\lambda$ is a hook Young diagram.  One can see that $(\ref{AtoS})$ 
is invariant with respect to subgroup 
\be
\frak{S}_{\lambda}=\prod_i \frak{S}_{\textsf{d}_{i}(\lambda)} \subset \frak{S}_{n}
\ee
 and thus, the number of the fixed components in $\ahilb^{\bA}$ which are isomorphic to hypertoric variety $\ahilb_{\lambda}$ equals:
$$
\dfrac{n!}{\prod_i \textsf{d}_{i}(\lambda)!},
$$  
and the total number of connected components in $\ahilb^{\bA}$ is 
$$
|\ahilb^{\bA}|=\sum\limits_{|\lambda|=n}\, \dfrac{n!}{\prod_i \textsf{d}_{i}(\lambda)!}.
$$
We will denote by the same symbols $x_i$ the restrictions of the tautological line bundles  to $\ahilb_{\lambda}$ induced by inclusion (\ref{Ah}).

\subsection{}
 For a $\lambda$-tree $\textbf{t}$ let  $\matC^{\times}_{\textbf{t}}$ be a one-dimensional torus acting on $M_{\lambda}$ by scaling the coordinates 
\be
\label{act1}
X_{ij}\to 
\left\{\begin{array}{ll}
	X_{ij} \epsilon^{-h_i+h_j}, & \textrm{if} \ \  (ij) \in {\textbf{t}} \\
	\\
	X_{ij}, & \textrm{else}. 
\end{array}\right.
\ee
where the height function $\h_{i}$ is defined by (\ref{condef}).
This action induces an action of $\matC^{\times}_{\textbf{t}}$ on 
$\ahilb_{\lambda}$. We denote by the same symbol $\textbf{t} \in \ahilb_{\lambda}^{\matC^{\times}_{\textbf{t}}}$ a fixed point corresponding to the compensating map
\be
\label{act2}
R_{\textbf{t}}: \epsilon  \rightarrow \textrm{diag}(\epsilon^{\h_1},\cdots, \epsilon^{\h_n}  ).
\ee
 This map defines a decomposition (\ref{fpdec}) associated to the fixed point $\textbf{t}$: 
$$
M_{\lambda}=M_{\lambda}^{0}\oplus M_{\lambda}^{1} 
$$ 
with
$$
M_{\lambda}^{0}=  I_{r} \oplus \bigoplus\limits_{{c_j=c_i+1,} \atop{(i,j) \in  \textbf{t}}} X_{i j}, \ \ \ M_{\lambda}^{1}= \bigoplus\limits_{{c_i=0} \atop {i\neq r}} I_{i} \oplus \bigoplus\limits_{{c_j=c_i+1,} \atop {(i,j) \notin  \textbf{t}}} X_{i j}.
$$
where $r=(1,1) \in\lambda$ is the root of $\textbf{t}$, the only box
in the Young diagram with zero height.  A representative of a fixed point is a semi-stable vector
$$
r_{\textbf{t}} \in T^{*} M_{\lambda}^{0}= I_{r} \oplus \bigoplus\limits_{{c_j=c_i+1,} \atop{(i,j) \in  \textbf{t}}} X_{i j} \oplus J_{r}\oplus \bigoplus\limits_{{c_j=c_i+1,} \atop{(i,j) \in  \textbf{t}}} Y_{j i}
$$
with exactly $n$  components which are not zero identically.  
\begin{Proposition} \label{stabper}
The $n$ non-trivial components of the representative vector 
$r_{\textbf{t}}$ are given by: the component of $I_r$, the components 
$X_{h(e),t(e)}$	 for $c_{h(e)}=c_{t(e)}-1$ and the components $Y_{h(e),t(e)}$ for $c_{h(e)}=c_{t(e)}+1$  where the symbols 
$h(e)$ and $t(e)$ denote the head and tail boxes of the edges $e$ of the $\lambda$-tree $\textbf{t}$ in the canonical orientation.
	
\end{Proposition}

\begin{proof}
The stable points in the abelianization are described by the Proposition \ref{abstab}.  First, the matrix element $I_r$ must be non-zero such that $\textrm{im}(I)\neq 0$. Thus $J_r=0$. Second, if $(i,j)\in \textbf{t}$ such that $c_i=c_j-1$ then we have two possibilities 
\be \label{choice}
\{X_{i j}\neq 0 ,  Y_{j i}=  0\} \ \ \ \textrm{or}  \ \ \ \{X_{i j}= 0,  Y_{j i}\neq 0\}
\ee
because the sums in (\ref{repvector}) run over non-intersecting sets.
The condition (\ref{choice}) obviously defines an orientation on the edge:
$$
X_{i j}\neq 0  \Leftrightarrow (h(e)=i, t(e)=j) \ \ \textrm{or} \ \ Y_{j i}\neq 0  \Leftrightarrow (h(e)=j, t(e)=i).
$$
The choice (\ref{choice}) is uniquely determined from the stability conditions. We conclude that the stability condition endows a tree with certain orientation on its edges. To determine  this orientation we note that for the canonical orientation of the tree, i.e. when each edge is oriented from the root, the only $X,Y$ stable subspace in $V$ containing  $\textrm{im}(I)=r$ is
$$
Span\{ X^a Y^b (r) \} =V
$$
and thus by Proposition \ref{abstab} the corresponding point is stable. 
\end{proof}

\subsection{} 
\label{chamsec}

A fixed point $\textbf{t}\in \ahilb_{\lambda}^{\matC^{\times}_{\textbf{t}}}$ may also be viewed as a fixed point 
$\textbf{t}\in \ahilb^{\bA\times \matC^{\times}_{\textbf{t}}}$. We are interested in the elliptic stable envelope of $\textbf{t}$ for both cases. 
Let $\fC^{'}$ be a chamber in the real Lie  algebra of the torus
$\bA\times \matC^{\times}_{\textbf{t}}$
spanned by the elements pairing positively with cocharacter $\sigma = (-1,\epsilon)$ with  $0<\epsilon\ll 1$. Let $\fC^{''}\subset \fC^{'}$ be the one-dimensional face which is the chamber in the real Lie algebra of $\matC^{\times}_{\textbf{t}}$. 
This  chamber is spanned by the elements
pairing positively with $\epsilon>0$.


\begin{Proposition} \label{propstabs}
	Up to a shift of the K\"{a}hler parameters by a powers of $\hbar$ the elliptic stable envelopes of  $\textbf{t}$ take the form:
	\be 	\label{stabA}
	\textrm{Stab}_{\fC^{''}}(\textbf{t})=\ {S}^{'}_{\lambda}\, W_{\textbf{t}}\,,\ \ \ 
	\textrm{Stab}_{\fC^{'}}(\textbf{t})=  {S}_{\lambda} \, W_{\textbf{t}} 
	\ee
	with 
	\be
	\label{shp}
	{S}^{'}_{\lambda}=\prod\limits_{{c_j=c_i+1}\atop {\h_{i}>\h_{j}}} \vartheta(x_i/x_j t_1) \prod\limits_{{c_j=c_i+1}\atop{\h_{i}<\h_{j}}} \vartheta(x_j/x_i t_2) \prod\limits_{{c_i=0}} \vartheta(x_i)
	\ee
	and
	\be
	\label{sh}
	{S}_{\lambda}=\prod\limits_{{\rho_j>\rho_i+1}} \vartheta(x_i/x_j t_1) \prod\limits_{{\rho_j<\rho_i+1}} \vartheta(x_j/x_i t_2) \prod\limits_{\rho_i\leq0} \vartheta(x_i)
	\prod\limits_{\rho_i>0} \vartheta(t_1 t_2 /x_i)
	\ee
for $\rho_i$ given by (\ref{refc})	and
	\be
	\label{Wpart}
	W_{\textbf{t}} =(-1)^{\kappa_{\textbf{t}}} \phi(x_r, \prod\limits_{i=1}^{n}  z_i) \prod\limits_{e\in \textbf{t}} \phi\Big(\dfrac{x_{h(e)} \varphi^{\lambda}_{t(e)}}{x_{t(e)}\varphi^{\lambda}_{h(e)}}, \prod\limits_{i \in [h(e),\textbf{t}]} z_i\Big)
	\ee
	where $\kappa_{\textbf{t}}$ is the number of 
	vertical edges in the tree directed down plus number of horizontal edges directed to the left in the canonical orientation.

\end{Proposition}
\begin{proof}
	To compute the elliptic stable envelopes we use Theorem \ref{abstabth}. This theorem describes the elliptic stable envelope of a fixed point in the hypertoric variety up to the polarization-dependent shift $\tau^{*}$
	of the K\"{a}hler parameters. 
	
	 From  (\ref{ML}) we find that the one-dimensional subrepresentations of $M_{\lambda}$ are spanned by
	$X_{ij}$ with $c_j=c_i+1$ and $I_{i}$ with $c_{i}=0$ with the following characters:
	$$
	X_{ij}\to 
	\left\{\begin{array}{ll}
	X_{ij} , & \textrm{if} \ \  (i,j) \in {\textbf{t}} \\
	\\
	X_{ij} \epsilon^{h_i -h_j}, & \textrm{else} 
	\end{array}\right.,   \ \ \ I_{i} \to I_{i} \epsilon^{h_i}.
	$$
Thus, the subspaces of $M^{1}_{\lambda}$ with nontrivial  $\epsilon$-weights correspond to $(i,j)\notin {\textbf{t}} $  and for (\ref{shgen}) we obtain:
$$
	\textbf{S}_{\textbf{t}}=\prod\limits_{{{c_j=c_i+1}\atop {\h_{i}>\h_{j}}} \atop {(i,j)\notin \textbf{t} } }\vartheta({\cal{L}}_{X_{i j}}) \prod\limits_{{{c_j=c_i+1}\atop {\h_{i}<\h_{j}}} \atop  {(i,j)\notin \textbf{t} }   }\vartheta({\cal{L}}_{Y_{j i}} ) 
	\prod\limits_{{{c_{i}=0} \atop {h_i>0}} } \vartheta({\cal{L}}_{I_{i}})
	$$ 
which together with (\ref{LtoX}) and $\hbar =t_1 t_2$ gives:

\be \label{numer}
\begin{array}{l}
\textbf{S}_{\textbf{t}}=\prod\limits_{{{c_j=c_i+1}\atop {\h_{i}>\h_{j}}} \atop {(i,j)\notin \textbf{t} } }\vartheta(x_i/x_j t_1) \prod\limits_{{{c_j=c_i+1}\atop {\h_{i}<\h_{j}}} \atop  {(i,j)\notin \textbf{t} }   }\vartheta(x_j/x_i t_2) 
\prod\limits_{{{c_{i}=0} \atop {h_i>0}}  } \vartheta(x_i)=\\
\\
(-1)^{\kappa_{\textbf{t}}}\dfrac{\prod\limits_{{{c_j=c_i+1}\atop {\h_{i}>\h_{j}}}  }\vartheta(x_i/x_j t_1) \prod\limits_{{{c_j=c_i+1}\atop {\h_{i}<\h_{j}}}   }\vartheta(x_j/x_i t_2) 
\prod\limits_{{{c_{i}=0} } } \vartheta(x_i)}{\vartheta(x_r) \prod\limits_{e\in \textbf{t} } \vartheta(x_{h(e)} \varphi^{\lambda}_{t(e)}/(x_{t(e)}\varphi^{\lambda}_{h(e)}))}.
\end{array}
\ee	
Note that to cancel the denominator with the corresponding factors in the numerator we have to invert the argument in the theta functions for vertical edges in the tree which are directed down and horizontal edges which are directed to the left
in the canonical orientation. Each of these factors contributes a sign
$\vartheta(x^{-1})=-\vartheta(x)$, which in total gives   $(-1)^{\kappa_{\textbf{t}}}$. Note that the numerator of this formula gives (\ref{shp}).

The stable representative of the fixed point corresponding to the $\lambda$-tree is given by Proposition \ref{stabper}. Thus, formula (\ref{wgen}) in this case gives:
$$
\textbf{W}_{\textbf{t}}(z_i)=\dfrac{\vartheta({\cal{L}}_{I_{r}} m_r)}{\vartheta(m_r)} \prod\limits_{{r\in \textbf{t}}\atop{c_{h(e)}=c_{t(e)}-1}} \dfrac{\vartheta({\cal{L}}_{X_{h(e),t(e)}} m_e)}{\vartheta(m_e)} 
\prod\limits_{{r\in \textbf{t}}\atop{c_{h(e)}=c_{t(e)}+1}} \dfrac{\vartheta({\cal{L}}_{Y_{h(e),t(e)}} m_e)}{\vartheta(m_e)}
$$	
which after substitution (\ref{LtoX})  can be conveniently written as
\be
\label{wthpart}
\textbf{W}_{\textbf{t}}(z_i)=\dfrac{\vartheta(x_r m_r)}{\vartheta(m_r)} \prod\limits_{{e\in \textbf{t}}} \dfrac{\vartheta\Big(\dfrac{x_{h(e)} \varphi^{\lambda}_{t(e)}}{x_{t(e)} \varphi^{\lambda}_{h(e)}} m_e \Big)}{\vartheta(m_e)} .
\ee
This, together with denominator in (\ref{numer}) gives (\ref{Wpart}).	The dependence of monomials $m_e$  on K\"{a}hler parameters is fixed by the quasiperiods of universal line bundle,  which in this case is:
$$
\textbf{W}_{\textbf{t}}(z_1,\dots ,z_{i} q,\dots ,z_n)= \dfrac{\varphi^{\lambda}_{i}}{x_i} \textbf{W}_{\textbf{t}}(z_1,\dots ,z_{i},\dots,z_n)
	$$
One checks that the choice
$$
m_r=z_1\cdots z_n, \ \ \ m_e=\prod\limits_{i \in [h(e),\textbf{t}]} z_i
$$
satisfies these conditions. 
	
For  $\textrm{S}_{\lambda}$ we have and extra action of $\bA$ 	by (\ref{Aact}) such that in total we have:
$$
X_{i j}\to 
\left\{\begin{array}{ll}
X_{i j} , & \textrm{if} \ \  (i,j) \in {\textbf{t}} \\
\\
	X_{i j} a^{c_i-c_j+1} \epsilon^{h_i -h_j}, & \textrm{else} 
	\end{array}\right.,   \ \ \ I_{i} \to I_{i} a^{c_i} \epsilon^{h_i}.
	$$
Thus, the value of the cocharacter $\sigma$ on the character of
	$X_{i j}$
	is $c_j-c_i-1+\epsilon (h_i-h_j)=\rho_{j}-\rho_{i}-1$. Similarly, the value of $\sigma$ on the character of $I_{i}$ is $-\rho_{i}$. Thus, for  (\ref{shgen}) we obtain:
$$
\textbf{S}_{\textbf{t}}=\prod\limits_{{(i,j)\notin \textbf{t}}\atop {\rho_j>\rho_i+1}}\vartheta({\cal{L}}_{X_{i j}}) \prod\limits_{{(i,j)\notin \textbf{t}}\atop {\rho_j<\rho_i+1}}\vartheta({\cal{L}}_{Y_{j i}}) 
\prod\limits_{\rho_{i}<0} \vartheta({\cal{L}}_{I_{i}}) \prod\limits_{\rho_i>0} \vartheta({\cal{L}}_{J_{i}})
$$ 
Which in tautological classes takes the form:
$$
\begin{array}{c}
\textbf{S}_{\textbf{t}}={\prod\limits_{{(i,j)\notin \textbf{t}}\atop{\rho_j>\rho_i+1}} \vartheta(x_i/x_j t_1) \prod\limits_{{(i,j)\notin \textbf{t}}\atop {\rho_j<\rho_i+1}} \vartheta(x_j/x_i t_2) \prod\limits_{\rho_i<0} \vartheta(x_i)
	\prod\limits_{\rho_i>0} \vartheta(t_1 t_2 /x_i)}=\\
\\
(-1)^{\kappa_{\textbf{t}}} \dfrac{\prod\limits_{{\rho_j>\rho_i+1}} \vartheta(x_i/x_j t_1) \prod\limits_{{\rho_j<\rho_i+1}} \vartheta(x_j/x_i t_2) \prod\limits_{\rho_i\leq 0} \vartheta(x_i)
	\prod\limits_{\rho_i>0} \vartheta(t_1 t_2 /x_i)}{\vartheta(x_r) \prod\limits_{e\in \textbf{t} } \vartheta(x_{h(e)} \varphi^{\lambda}_{t(e)}/(x_{t(e)}\varphi^{\lambda}_{h(e)}))}
\end{array}
$$
The  computation of the K\"{a}hler part $W_{\textbf{t}}(z_i)$ remains the same.
\end{proof}

\section{Abelianization of stable envelope}
The main reference for this section is Section 4.3 of \cite{AOElliptic}, where the proof of the existence of the elliptic stable envelopes for Nakajima varieties is given. 

\subsection{}
Let $B\subset GL(V)$ be a Borel subgroup with Lie algebra $\frak{b}$.
The Hilbert scheme and its abelianization  fit into the following diagram (See Section 4.3 in \cite{AOElliptic} for definitions. In notations of \cite{AOElliptic} $X=\hilb$, $X_{S}=\ahilb$, $F=\{\lambda\}$ and  $F_{S}=\ahilb_{S}$):
\[
\xymatrix{\textsf{Fl} \ar[rr]^{\j_+ \ \ \ \ } \ar[d]^{\pi}&& \mu^{-1}(\frak{b}^{\perp})/\!\!/S \ar[rr]^{\j_{-}} & & \ahilb\\
	\hilb & & }
\]
where $\textsf{Fl}=T^{*} M/\!\!/\!\!/\!\!/B$ is a flag fibration over $\hilb$, such that the fiber of $\pi$ is isomorphic to a flag variety $GL(V)/B$.

If $\lambda \in \hilb^{\bA}$ is a fixed point, we also have a similar diagram for the $\bA$-fixed hypertoric subvariety $\ahilb_{\lambda}$:
\[
\xymatrix{\textsf{Fl}^{'} \ar[rr]^{\j^{'}_+ \ \ \ \ } \ar[d]^{\pi^{'}}&& M_{\lambda} \cap \mu^{-1}(\frak{b}^{\perp})/\!\!/S \ar[rr]^{\j^{'}_{-}} && \ahilb_{\lambda}\\
	\{\lambda\} && }
\]
In this case $\textsf{Fl}^{'}$ is a $\bA$-fixed component of $\textsf{Fl}$ which  itself is a product of flag varieties:
\be
\label{flagprod}
\textsf{Fl}^{'}\cong \prod\limits_{i \in \matZ} GL( \textsf{d}_{i}(\lambda))/B(\textsf{d}_i(\lambda))
\ee
where $\textsf{d}_i(\lambda)$ are defined by (\ref{dfun}) and $B(\textsf{d}_i(\lambda))\subset GL(\textsf{d}_i(\lambda))$ is the corresponding Borel subgroup. Among all possible fixed components of $\textsf{Fl}^{'}$ we choose the one for which the normal $a$-weights to  $\textsf{Fl}^{'}$ 
in $\pi^{-1}(\lambda)$ are negative. 

\subsection{}
In Section 4.3.10 of \cite{AOElliptic}, the following diagram  
\be
\label{abdiag1}
\ee
\vspace{-1.3cm}
\[
\xymatrix{{\mathscr{U}}^{'} \ar[rrr]^{  \j^{'}_{- *}\circ(\j^{'*}_{+})^{-1}\circ\pi^{'-1}_{*} } \ar[d]^{\textrm{Stab}_{\fC} }&&&  \ar[d]^{\textrm{Stab}_{\fC}^{'}} \Theta(T^{1/2}\ahilb_{\lambda} ) \otimes {\mathscr{U}}^{'}  \\
	\Theta(T^{1/2} \hilb) \otimes {\mathscr{U}}  &&& 
	\ar[lll]_{  \pi_{*}\circ \j^{*}_{+}\circ (\j_{- *})^{-1}}
	\Theta(T^{1/2} \ahilb) \otimes {\mathscr{U}} }
\]
of maps of line bundles is used to define the stable envelope 
$\textrm{Stab}_{\fC}$ as a composition\footnote{More precisely, in Section 4.3, \cite{AOElliptic} the authors show that the right side of (\ref{abdef}) is well defined and satisfies all defining properties for stable envelopes. Thus, by uniqueness it coincides with the left side, if it exists. The map $\textrm{Stab}_{\fC}^{'}$ is the elliptic stable envelope for hypertoric varieties and thus it is well defined and exists. This way, the existence of the elliptic stable envelopes for the Nakajima varieties is proven. }
\be  \label{abdef}
\textrm{Stab}_{\fC}=\pi_{*}\circ \j^{*}_{+}\circ (\j_{- *})^{-1}\circ \textrm{Stab}_{\fC}^{'}  \circ \j^{'}_{- *}\circ(\j^{'*}_{+})^{-1}\circ\pi^{'-1}_{*}.
\ee
The maps $\pi_{*}$, $\pi^{'}_{*}$ and $\j^{'*}_{+}$ are not isomorphisms but only surjective. 

\noindent
By the map $ \j^{'}_{- *}~\circ~(\j^{'*}_{+})^{-1}~\circ~\pi^{'-1}_{*}$ in (\ref{abdiag1})  a choice of a \textit{formal inverse} is understood, i.e., a map $\textsf{r}$ which satisfies
\be
\label{frominv}
\pi^{'}_{*}\circ \j^{'*}_{+} \circ (\j^{'}_{-*})^{-1} \circ \textsf{r} =id.
\ee
The composition (\ref{abdef}) is independent on this choice as discussed in 
Section 4.3.11 \cite{AOElliptic}.

As shown in Section 4.3.12, \cite{AOElliptic} there exist a well defined map $m$ such that
$$
\j_{-*} \circ  m = \textrm{Stab}_{\fC}^{'} \circ \j^{'}_{-*}
$$
i.e. the map  $\textrm{Stab}_{\fC}^{'} \circ \j^{'}_{-*}$ factors through $\j_{-*}$. This map
is denoted by $(\j_{-*})^{-1}\circ \textrm{Stab}_{\fC}^{'}\circ \j^{'}_{-*}$ in the diagram above.

With all these, the right side of (\ref{abdef}) is defined.

\subsection{} 
Note that the stable maps in the right and left parts of  (\ref{abdiag})
are defined for varieties with non-isomorphic Picard groups. In particular
they depend on different number of K\"{a}hler parameters. Let us comment on this seeming discrepancy here.  

The Picard group of $\hilb$ is a subgroup of its abelianization 
with embedding given by the inclusion of characters:
\be \label{picemb}
\textrm{Pic}(\hilb) \cong \textrm{char}( GL(V)) \rightarrow 
\textrm{char}( S ) \cong \textrm{Pic}(\ahilb)
\ee
which extends to an embedding $\cB_{\hilb,\bT} \hookrightarrow \cB_{\ahilb,\bT}$. All maps in (\ref{abdiag})  are understood after tensoring with the universal bundles and restricting to
the image of this embedding.   

We note that $\textrm{Pic}(\hilb)\cong \matZ$ 
with generator ${\cal{O}}(1)$. Explicitly (\ref{picemb}) has the form:
$$
{\cal{O}}(1) \rightarrow x_1 x_2 \dots x_n.
$$  
This induced the following restriction map on K\"{a}hler variables:
\be \label{zref}
z_i  \rightarrow z, \ \ \ i=1\dots n.
\ee

\subsection{}
\begin{Proposition} \label{pushpro}
	Let $f(x_1,\cdots, x_n)$ be a section of the line bundle 	$\Theta(T^{1/2} \ahilb) \otimes {\mathscr{U}}$ then 
	\be\label{fpush}  
	\pi_{*} \circ \j_{+}^{*} \circ (\j_{-*})^{-1} \Big( f(x_1,\cdots, x_n) \Big)=
	\sum\limits_{\sigma \in \frak{S}_{n}}\, \dfrac{f(x_{\sigma(1)},\cdots, x_{\sigma (n)})}{\prod\limits_{\rho_i<\rho_j} \vartheta(x_{\sigma (i)}/x_{\sigma(j)}) \vartheta(x_{\sigma (i)}/x_{\sigma (j)} \hbar)}
	\ee
	Similarly, for a section of $\Theta(T^{1/2} \ahilb_{\lambda}) \otimes {\mathscr{U}^{'}}$ we have
	\be \label{secpush}
	\pi^{'}_{*} \circ \j_{+}^{'*} \circ (\j^{'}_{-*})^{-1} \Big( f(x_1,\cdots, x_n) \Big)=
	\sum\limits_{\sigma \in \frak{S}_{\lambda}}\, \dfrac{f(\varphi^{\lambda}_{\sigma(1)},\cdots, \varphi^{\lambda}_{\sigma (n)})}{\prod\limits_{
			{h_i>h_j} \atop {c_{i}=c_{j}} } \vartheta(\varphi^{\lambda}_{\sigma (i)}/\varphi^{\lambda}_{\sigma (j)})\vartheta(\varphi^{\lambda}_{\sigma (i)}/\varphi^{\lambda}_{\sigma (j)} \hbar)}
	\ee
\end{Proposition}
\begin{proof}
	Let $\frak{n}=[\frak{b},\frak{b}]$ and denote by $\mathscr{N}$ the corresponding tautological bundle on $\ahilb$ associated to
	$S$ action on $\frak{n}$.  The fiber of $\pi$ is isomorphic to $GL(V)/B$ and therefore (for the tangent map $d \pi$)
	$$
	\textrm{Ker} (d \pi)  \cong  \mathscr{N}^{\vee}=\sum\limits_{i<j} \dfrac{x_i}{x_j}
	$$
is the sum of line bundles corresponding to positive roots associated to $\frak{b}$.
 By assumption, the normal $a$-weights
 to $\textsf{Fl}^{'}$ in $\pi^{-1}(\lambda)$ are negative. This corresponds to the following choice of the order
\be \label{pker}
\textrm{Ker} (d \pi)  \cong  \mathscr{N}^{\vee}=\sum\limits_{\rho_i<\rho_j} \dfrac{x_i}{x_j}
\ee
Such that upon restriction (\ref{restric}) all $a$-weights (\ref{ttoa}) appearing in this sum are non-positive.  Computation of normal bundles to $\j_{-}$ gives:
\be
\label{jker}
\textrm{normal bundle to} \ \ \j_{-}=\hbar \mathscr{N}^{\vee}=\hbar \sum\limits_{\rho_i<\rho_j} \dfrac{x_i}{x_j}.
\ee
see 4.3.4 in \cite{AOElliptic}.
Thus, the push-forward $\pi_{*}$ from flag variety (the elliptic version of Weyl character formula)  contributes as symmetrization over 
the Weyl group $\frak{S}_{n}$ together with the Thom class of the normal bundle $\Theta(\mathscr{N}^{\vee})^{-1}$ which is the first factor in the denominator of (\ref{fpush}).  Similarly, the second factor in the denominator is 
	$\Theta(\hbar\mathscr{N}^{\vee})^{-1}$ which is the Thom class coming from the inversion of $\j_{- *}$.
	
For the second formula the consideration is exactly the same 
	with $\mathscr{N}$ replaced by its $\bA$ fixed part. The torus $\bA$
	acts on $\mathscr{N}$ through the inclusion (\ref{AtoS}), from which we see that:
	$$
	\bA - \textrm{fixed part of}  \ \ \mathscr{N}= \sum\limits_{{\rho_i>\rho_j} \atop {c_{i}=c_{j}}} x_i/x_j
	$$
	In localization, the push forward $\pi^{'}_{*}$ is the sum over the fixed points on (\ref{flagprod}), i.e, the sum over the Weyl group
	of $\prod_{i} GL(\textsf{d}_i(\lambda))$ which is $\frak{S}_{\lambda}$.  
	Finally, as $\lambda$ is a point, all line bundles $x_i$ evaluate to the corresponding $\bT$ characters $x_i=\varphi^{\lambda}_{i} \in K_{\bT}(\lambda)$ as in (\ref{restric}).
\end{proof}
\section{Proof of Theorem \ref{mainth} \label{profsec}}
\subsection{}
Let us consider the functions given explicitly by
\be \label{num}
N_{\lambda}(x_1,\dots,x_n)
={\prod\limits_{{{c_j=c_i+1}\atop {\h_{i}>\h_{j}}} \atop {(i,j)\notin \Gamma_{\lambda} } }\vartheta(x_i/x_j t_1) \prod\limits_{{{c_j=c_i+1}\atop {\h_{i}<\h_{j}}} \atop  {(i,j)\notin \Gamma_{\lambda} }   }\vartheta(x_j/x_i t_2) 
	\prod\limits_{{{c_{i}=0} \atop {h_i>0}}} \vartheta(x_i)},
\ee
and
\be \label{den}
D_{\lambda}(x_1,\dots,x_n)={\prod\limits_{{c_i=c_j}
		\atop {h_i>h_j}}\vartheta(x_i/x_j) \prod\limits_{{{c_i=c_j}
			\atop {h_i>h_j+2}}}\vartheta(x_i/x_j \hbar)},
\ee
where we assume that products run over boxes $i,j\in \lambda$. For a permutation $\sigma \in {\frak{S}}_{\lambda}$ we also set
$$
N^{\sigma}_{\lambda}(x_1,\dots,x_n)=N_{\lambda}(x_{\sigma(1)},\dots,x_{\sigma(n)}), \ \ D^{\sigma}_{\lambda}(x_1,\dots,x_n)=D_{\lambda}(x_{\sigma(1)},\dots,x_{\sigma(n)}).
$$
\begin{Proposition} \label{vanprop1}
	If $N^{\sigma}_{\lambda}(\varphi^{\lambda}_1,\dots,\varphi^{\lambda}_n)\neq 0$ then $\sigma = 1$. 
\end{Proposition}
\begin{proof}
	Let us denote by $b^{i}_r \in \lambda$ the boxes with content $r$ ordered by the value of height (\ref{condef}), such that 
 $h_{b^{1}_r}>h_{b^{1}_r}>\dots >h_{b^{m_r}_r}$  and $m_r$ is the number of boxes in $\lambda$ with content $r$, see Fig.\ref{boxespic}.

	The proof is by induction on the value of $c_i$ for $i\in \lambda$.
	Assume that from $N^{\sigma}_{\lambda}(\varphi^{\lambda}_1,\dots,\varphi^{\lambda}_n)\neq 0$ follows that $\sigma$ acts trivially on all boxes $i\in \lambda$ with
	content $c_i<k<0$.

\begin{figure}[h!] 
\vspace{-0cm}
	\center
	\begin{tikzpicture}[draw=blue,baseline={(-7,-5.5)}]
	\draw [line width=1pt] (-7,0) -- (-11,4);
	\draw [line width=1pt] (-7,0) -- (-2.5,4.5);
	\draw [line width=1pt] (-6.5,0.5) -- (-10.5,4.5);
	\draw [line width=1pt] (-6,1) -- (-9.5,4.5);
	\draw [line width=1pt] (-5.5,1.5) -- (-8.5,4.5);	
	\draw [line width=1pt] (-5.0,2) -- (-7.5,4.5);
	\draw [line width=1pt] (-4.5,2.5) -- (-6,4);
	\draw [line width=1pt] (-4.0,3) -- (-5.5,4.5);
	\draw [line width=1pt] (-3.5,3.5) -- (-4.5,4.5);
	\draw [line width=1pt] (-3,4) -- (-3.5,4.5);
	\draw [line width=1pt] (-2.5,4.5) -- (-3,5);
	
	\node [left] at (-6.65,0.5) {$b_0^4$};
	\node [left] at (-6.65,1.5) {$b_0^3$};
	\node [left] at (-6.65,2.5) {$b_0^2$};
	\node [left] at (-6.65,3.5) {$b_0^1$};
	
	\node [left] at (-7.05,1) {$b^{4}_{-1}$};
	\node [left] at (-7.05,2) {$b^{3}_{-1}$};
	\node [left] at (-7.05,3) {$b^{2}_{-1}$};
	\node [left] at (-7.05,4) {$b^{1}_{-1}$};

	\node [left] at (-7.57,1.5) {$b_{-2}^{3}$};
	\node [left] at (-7.57,2.5) {$b_{-2}^{2}$};
	\node [left] at (-7.57,3.5) {$b_{-2}^{1}$};
	
	\node [left] at (-8.05,2) {$b^{3}_{-3}$};
	\node [left] at (-8.05,3) {$b^{2}_{-3}$};
	\node [left] at (-8.05,4) {$b^{1}_{-3}$};	
	
	\node [left] at (-8.57,2.5) {$b^{2}_{-4}$};	
	\node [left] at (-8.57,3.5) {$b^{1}_{-4}$};	
	
	\node [left] at (-9.07,3) {$b^{2}_{-5}$};	
	\node [left] at (-9.07,4) {$b^{1}_{-5}$};
	
	\node [left] at (-9.57,3.5) {$b^{1}_{-6}$};
	\node [left] at (-10.07,4) {$b^{1}_{-7}$};

	\node [left] at (-6.15,1) {$b_1^{3}$};	
	\node [left] at (-6.15,2) {$b_1^{2}$};
	\node [left] at (-6.15,3) {$b_1^{1}$};

	\node [left] at (-5.65,1.5) {$b_2^3$};
	\node [left] at (-5.65,2.5) {$b_2^2$};
	\node [left] at (-5.65,3.5) {$b_2^1$};
	
	\node [left] at (-5.15,2) {$b_3^3$};
	\node [left] at (-5.15,3) {$b_3^2$};
	\node [left] at (-5.15,4) {$b_3^1$};
	
	\node [left] at (-4.65,2.5) {$b_4^2$};
	\node [left] at (-4.65,3.5) {$b_4^1$};
	
	\node [left] at (-4.15,3) {$b_5^2$};
	\node [left] at (-4.15,4) {$b_5^1$};
	
	\node [left] at (-3.65,3.5) {$b_6^1$};

	\node [left] at (-3.15,4) {$b_7^1$};
	
	\node [left] at (-2.65,4.5) {$b_8^1$};

	\draw [line width=1pt] (-8,1) -- (-4.5,4.5);

	\draw [line width=1pt] (-5.5,1.5) -- (-6,2);
	\draw [line width=1pt] (-7.5,0.5) -- (-3,5);
	\draw [line width=1pt] (-8.5,1.5) -- (-5.5,4.5);
	\draw [line width=1pt] (-9,2) -- (-7,4);
	
	\draw [line width=1pt] (-9.5,2.5) -- (-7.5,4.5);
	\draw [line width=1pt] (-10,3) -- (-8.5,4.5);
	
	\draw [line width=1pt] (-10.5,3.5) -- (-9.5,4.5);
	\draw [line width=1pt] (-11,4) -- (-10.5,4.5);
	
		\end{tikzpicture}
\vspace{-5cm}
\caption{\label{boxespic}}
\end{figure}

	\underline{Step 1}: Assume that $\sigma$ acts non-trivially on the boxes with content $k$ and  $N^{\sigma}_{\lambda}(\varphi^{\lambda}_1,\dots,\varphi^{\lambda}_n)\neq 0$, then $\sigma(b^{1}_k)=b^{1}_k$ or $\sigma(b^{1}_k)=b^{m_k}_k$. Indeed, if $\sigma(b^{1}_k)=(i,j)\neq b^{1}_k, b_k^{m_k}$ then $N_{\lambda}$ contains the factor $\vartheta(x_{b^{1}_k}/x_{(i-1,j)} t_2)$ and thus $N^{\sigma}_{\lambda}$ contains the factor $\vartheta(x_{(i,j)}/x_{(i-1,j)} t_2)$ which vanishes at $x_\Box=\varphi^{\lambda}_{\Box}$.  Similarly, if $\sigma(b^{1}_k)=b^{1}_k$ then  $\sigma(b^{2}_k)=b^{2}_k$ or $\sigma(b^{2}_k)=b^{m_k}_k$ and so on. 
	We conclude that if $\sigma$ acts non-trivially on boxes with content $k$ then $\sigma(b_k^{m_k})\neq b_k^{m_k}$.

	\underline{Step 2}:
	We note that if $\sigma(b_k^{m_k})\neq b_k^{m_k}$, $k<0$ and $N^{\sigma}_{\lambda}(\varphi^{\lambda}_1,\dots,\varphi^{\lambda}_n)\neq 0$ then $\sigma(b_{k+1}^{m_{k+1}})\neq b_{k+1}^{m_{k+1}}$. Indeed, if  $\sigma(b_k^{m_k})\neq b_k^{m_k}$ then there exists a box $b\in \lambda$ such that $\sigma(b)=b_k^{m_k}$. If $k<0$ then $N_{\lambda}$ contains a factor
	$\vartheta(x_{b}/x_{b_{k+1}^{m_{k+1}}} t_1)$. If $\sigma(b_{k+1}^{m_{k+1}})= b_{k+1}^{m_{k+1}}$ then $N^{\sigma}_{\lambda}$  contains a vanishing factor $\vartheta(x_{b_k^{m_{k}}}/x_{b_{k+1}^{m_{k+1}}} t_1)$ which contradicts the assumption $N^{\sigma}_{\lambda}(\varphi^{\lambda}_1,\dots,\varphi^{\lambda}_n)\neq 0$.  
	Similarly, if  $\sigma(b_{k+1}^{m_{k+1}})\neq b_{k+1}^{m_{k+1}}$, $k+1<0$ and $N^{\sigma}_{\lambda}(\varphi^{\lambda}_1,\dots,\varphi^{\lambda}_n)\neq 0$ then $\sigma(b_{k+2}^{m_{k+2}})\neq b_{k+2}^{m_{k+2}}$ and so on.  
	
	By induction on $k$ we see that if $\sigma(b_k^{m_k})\neq b_k^{m_k}$ and $N^{\sigma}_{\lambda}(\varphi^{\lambda}_1,\dots,\varphi^{\lambda}_n)\neq 0$ then $\sigma(b_{0}^{m_0})\neq b_{0}^{m_0}$.   
	This is, however, impossible. Indeed, if $\sigma(b_{0}^{m_0})\neq b_{0}^{m_0}$ then there exists a box $b$ with $c_b=0$ and $h_{b}>0$ such that
	$\sigma(b)=b_{0}^{m_0}$ and $N_{\lambda}$ contains a factor $\vartheta(x_b)$.  Therefore, $N_{\lambda}^{\sigma}$ has a vanishing factor $\vartheta(x_{b_0^{m_0}})$ which contradicts the assumption $N^{\sigma}_{\lambda}(\varphi^{\lambda}_1,\dots,\varphi^{\lambda}_n)\neq 0$.
	
	\textit{We conclude that $\sigma(b)=b$ for all $b\in \lambda$ with $c_b\leq 0$.}
	
	\underline{Step 3}:
	
	The final step is to note that if $\sigma(b)=b$ for all $b\in\lambda$ with $c_b=r\geq 0$ then  $\sigma(b)=b$ for all
		$b\in \lambda$  with $c_b=r+1$. 
		Indeed, assume that $\forall b\in \lambda$ with $c_b=r$ we have $\sigma(b)=b$, then $\sigma(b^1_{r+1})=b^1_{r+1}$. If it is not true then $\sigma(b^1_{r+1})=(i,j) \in \lambda$ and $N_{\lambda}$ contains a factor 
	$\vartheta(x_{b^{1}_{k+1}}/x_{(i-1,j)} t_2)$.  Thus, $N^{\sigma}_{\lambda}$ 
	contains a vanishing factor $\vartheta(x_{(i,j)}/x_{(i-1,j)} t_2)$.   
	The same argument shows that $\sigma(b^2_{r+1})=b^2_{r+1}$ and so on. We conclude $\sigma(b)=b$ for all $b$ with $c_b=r+1$. 
	The proposition follows by induction on $r$.  
\end{proof}

Both $N^{\sigma}_{\lambda}$ and $D^{\sigma}_{\lambda}$ are  
products of theta functions. They vanish if one or more of the corresponding factors vanishes.
\begin{Proposition} \label{vanprop2}
	If $D^{\sigma}_{\lambda}$ has $n$-factors vanishing at $x_i=\varphi^{\lambda}_{i}$ then for $\sigma \neq 0$ the function  $N^{\sigma}_{\lambda}$ 
	has at least $n+1$  vanishing factor. 	
\end{Proposition}		
\begin{proof}
	First, we note that the only factors of $D_{\lambda}^{\sigma}$ that can vanish at $x_i=\varphi^{\lambda}_{i}$ are of the form $\vartheta(x_{(i+1,j+1)}/x_{(i,j)} \hbar)$.  
	Assume  $D_{\lambda}^{\sigma}$ has a vanishing factor $\vartheta(x_{(i+1,j+1)}/x_{(i,j)} \hbar)$. Then there are boxes $a,b \in \lambda$ with $h_{a}>h_{b}+2$ such that $\sigma(a)=(i+1,j+1)$ and 
	$\sigma(b)=(i,j)$. 
	
	Let us consider boxes $c,d\in \lambda$ such that $\sigma(c)=(i,j+1)$ and $\sigma(d)=(i+1,j)$. 
	If $h_a>h_c$ then $N_{\lambda}$ contains a factor $\vartheta(x_a/x_c t_2)$ and thus $N^{\sigma}_{\lambda}$ has a vanishing factor 
	$ \vartheta(x_{(i+1,j+1)}/x_{i,j+1} t_2)$. In the opposite situation
	$h_a<h_c$ (which also implies $h_b<h_c$) the function $N_{\lambda}$ contains a factor $\vartheta(x_c/x_b t_1)$ and thus  
	$N^{\sigma}_{\lambda}$ has a vanishing factor 
	$ \vartheta(x_{(i,j+1)}/x_{(i,j)} t_1)$. 
	
	We proved that if $\vartheta(x_{\sigma(a)}/x_{\sigma(b)} \hbar)$ is a vanishing factor of $D^{\sigma}_{\lambda}$ then either $\vartheta(x_{\sigma(a)}/x_{\sigma(c)} t_2)$ or $\vartheta(x_{\sigma(c)}/x_{\sigma(b)} t_1)$ is a vanishing factor of 
	$N^{\sigma}_{\lambda}$. Exactly same argument shows that
	$\vartheta(x_{\sigma(a)}/x_{\sigma(d)} t_1)$ or $\vartheta(x_{\sigma(d)}/x_{\sigma(b)} t_2)$ is also a vanishing factor of $N^{\sigma}_{\lambda}$.

	We conclude that with for every vanishing factor of $D^{\sigma}_{\lambda}$ one has two associated vanishing factors of $N^{\sigma}_{\lambda}$. From the same consideration as above one checks that two different vanishing factors of $D^{\sigma}_{\lambda}$ can not have more then one common associated vanishing factors of $N^{\sigma}_{\lambda}$. This implies that if 
	$D^{\sigma}_{\lambda}$ has $n$ vanishing factors then $D^{\sigma}_{\lambda}$ has minimum $n+1$ vanishing factor. 
\end{proof}
Next, let us consider a function
\be \label{sfun}
S^{\sigma}_{\Gamma_{\lambda}}(x_1,\dots,x_n)=\dfrac{N^{\sigma}_{\lambda}(x_1,\dots,x_n)}{D^{\sigma}_{\lambda}(x_1,\dots,x_n)}.
\ee
\begin{Proposition} \label{svalues}
	The functions $S^{\sigma}_{\Gamma_{\lambda}}$ is non-singular at $\varphi^{\lambda}_{i}$ with values
	$$
	S^{\sigma}_{\Gamma_{\lambda}}(\varphi^{\lambda}_1,\dots,\varphi^{\lambda}_n)=\left\{\begin{array}{ll}
	1&\sigma =1\\
	0&\sigma \neq 1 
	\end{array}\right.
	$$
\end{Proposition}
\begin{proof}
	For $\sigma\neq 1$ this is follows immediately from Propositions \ref{vanprop1}	and \ref{vanprop2}. For $\sigma=1$, it is easy to see that both the numerator and denominator are nonzero and 
	$$
	N_{\lambda}(\varphi^{\lambda}_1,\dots,\varphi^{\lambda}_n)=
	D_{\lambda}(\varphi^{\lambda}_1,\dots,\varphi^{\lambda}_n).
	$$ 
	To prove it, we consider a  tangent space at a point $\textbf{t}$ corresponding to a $\lambda$-tree:
	$$
	T_{\textbf{t}} \ahilb_{\lambda} = T_{>} \oplus T_{0} \oplus T_{<}
	$$
	where $T_{>}$ denotes the subspace with positive $\hbar$ characters, i.e., characters of the form $\hbar^k$ with $k>0$.  Similarly $T_<$ and $T_0$ denote negative and zero characters. 
	As $\hbar^{-1}$ is the weight of the symplectic form, the space $T_{>}$ is symplectic dual to $T_{0} \oplus T_{<}$ such that
	$$
	T_{>}=\hbar\otimes (T_{0}^{\vee} \oplus T_{<}^{\vee})
	$$ 
	and thus $\dim T_{>}=\dim T_{0} +\dim T_{<} =\dim \ahilb_{\lambda} /2$.

	The elliptic stable envelope of $\textbf{t}$ given explicitly by (\ref{stabA}). By the property of ($\star,\star$), defining the elliptic stable envelope, at $x_i=\varphi^{\lambda}_{i}$ 
	it is equal to $\Theta(N_{-})$ where $N_{-}$ is a half of the tangent space corresponding to repelling directions for the chamber $\fC^{''}$. From explicit formula (\ref{stabA}) we see that the elliptic stable envelope is a multiple of $N_{\lambda}$. Note, that all $\hbar$ characters in the numerator of $N_{\lambda}$ at  $x_i=\varphi^{\lambda}_{i}$ are all negative (and the rest of the factors in the stable envelope are zero).   Therefore, 
	$N_{\lambda}(\varphi^{\lambda}_1,\dots,\varphi^{\lambda}_n)=\Theta(T_{<}).$ 
	
The analog of (\ref{pker}) and (\ref{jker}) for $\ahilb_{\lambda}$ 
	gives:
	$$ \textrm{Ker} (d \pi^{'})  \cong   \mathscr{N}^{\vee}_{\bA}, \ \ \ 
	\label{jker}
	\textrm{normal bundle to} \ \ \j_{-}^{'}=\hbar \mathscr{N}^{\vee}_{\bA}.
	$$
	where $\mathscr{N}^{\vee}_{\bA}$ is  $\bA$-fixed part of $\mathscr{N}^{\vee}$ or: 
	$$
	\mathscr{N}^{\vee}_{\bA} =\sum\limits_{{c_{i}=c_{j}}\atop {h_{i}>h_{j}}} \dfrac{x_i}{x_j} 
	$$
	therefore the virtual tangent space has the corresponding summands:
	$$
	T_{\textbf{t}} \ahilb_{\lambda}=\mathscr{N}^{\vee}_{\bA} \oplus \mathscr{N}^{\vee}_{\bA} \otimes \hbar \oplus...
	$$
	We compute that the $\hbar$ weights of these summands are non-positive, with negative weights corresponding to 
	$$
	\sum\limits_{{c_{i}=c_{j}}\atop {h_{i}>h_{j}}} \dfrac{x_i}{x_j} + \hbar \sum\limits_{{c_{i}=c_{j}}\atop {h_{i}>h_{j}+2}} \dfrac{x_i}{x_j}  
	$$  
	as in the denominator of $D_{\lambda}(x_1,\dots,x_n)$. Finally, counting the number of summands in this expression we find that it is equal to the dimension of $T_{<}$ and therefore it is equal to $T_{<}$.  
	We conclude that 
	$D_{\lambda}(\varphi^{\lambda}_1,\dots,\varphi^{\lambda}_n)=\Theta(T_{<}).$ 
\end{proof}
Next, we consider the following function
\be 
\label{Ffun}
F_{\lambda}(x_1,\dots,x_n)=\sum\limits_{\delta\in \Upsilon_{\lambda}} S_{\textbf{t}_{\delta}}(x_1,\dots,x_n)  \,W_{\textbf{t}_{\delta}}(z_i)
\ee
with where for a $\lambda$-tree $\textbf{t}_{\sigma}$ the function  $W_{\textbf{t}_{\sigma}}(z_i)$ given by (\ref{Wpart}) and
$$
S_{\textbf{t}_{\delta}}(x_1,\dots,x_n) =\dfrac{\prod\limits_{{{c_j=c_i+1}\atop {\h_{i}>\h_{j}}} \atop {(i,j)\notin \textbf{t}_{\delta} } }\vartheta(x_i/x_j t_1) \prod\limits_{{{c_j=c_i+1}\atop {\h_{i}<\h_{j}}} \atop  {(i,j)\notin \textbf{t}_{\delta} }   }\vartheta(x_j/x_i t_2) 
	\prod\limits_{{{c_{i}=0} \atop {h_i>0}} \atop {i\neq r} } \vartheta(x_i)}{\prod\limits_{{c_i=c_j}
		\atop {h_i>h_j}}\vartheta(x_i/x_j)\vartheta(x_i/x_j \hbar)}
$$ 
For $\sigma\in {\frak{S}}_{\lambda}$ we also denote 
$
 F_{\lambda}^{\sigma}(x_1,\dots,x_n)=F_{\lambda}(x_{\sigma(1)},\dots,x_{\sigma(n)}).
$ 

\begin{Proposition} \label{vanpro}
The functions $F^{\sigma}_{\lambda}(x_1,\dots,x_n)$ are regular at the point $x_i=\varphi^{\lambda}_{i}$ and take the following values:
	$$
	F^{\sigma}_{\lambda}(\varphi^{\lambda}_1,\dots,\varphi^{\lambda}_n)=\left\{\begin{array}{ll}
	1&\sigma =1\\
	0&\sigma \neq 1 
	\end{array}\right.
	$$  	
\end{Proposition}
\begin{proof}
First we note that $F^{\sigma}_{\lambda}(\varphi^{\lambda}_1,\dots,\varphi^{\lambda}_n)$ is a function of parameters $\hbar$ and $z_i$ only i.e. it does not depend on the equivariant parameter $a$. Thus we denote
$$
Q^{\sigma}(\hbar,z_i)=F^{\sigma}_{\lambda}(\varphi^{\lambda}_1,\dots,\varphi^{\lambda}_n).
$$
We also note that all the poles of $Q^{\sigma}(\hbar,z_i)$ in $\hbar$ are at the points $\hbar=q^i$.  A long but straightforward calculation shows that this function  has the following quasi-periods in $\hbar$
$$
Q^{\sigma}(\hbar q,z_i)=z^{\lambda}_{\sigma}\,  Q^{\sigma}(\hbar,z_i) 
$$	
where
$$
z^{\lambda}_{\sigma}=\prod\limits_{{{i,j\in \lambda:}\atop {\rho_i <\rho_j,}}\atop{\rho_{\sigma(i)} >\rho_{\sigma(j)}}}\dfrac{z_{j}}{z_i}.
$$
Thus, the function 
\be \label{Qshift}
Q^{\sigma}(\hbar,z_i) \dfrac{\vartheta(\hbar z^{\lambda}_{\sigma})}{\vartheta(\hbar)}
	\ee
	is double-periodic elliptic function of $\hbar$. It may only have poles at $\hbar=q^i$. Let us show that this function is actually regular at this points and thus poles free. Enough to show it for $\hbar=1$. First, from (\ref{wthpart}) we for $\sigma=1$ we have
	\be
	\label{Wval}
	W_{\textbf{t}_{\delta}}(\varphi^{\lambda}_{1},\dots,\varphi^{\lambda}_{n})=1.
	\ee  
	For general $\sigma$  in the limit $\hbar\to 1$ we have
	\be
	\label{Wvals}
	W_{\textbf{t}_{\delta}}(\varphi^{\lambda}_{\sigma(1)},\dots,\varphi^{\lambda}_{\sigma(n)})=1+o(\varepsilon),
	\ee
	where $\varepsilon=\hbar-1$ is a small parameter. Second, we write:
	$$
	S_{\textbf{t}_{\delta}}(x_1,\dots,x_n)=S_{\Gamma_{\lambda}}(x_1,\dots,x_n) \tilde{S}_{\textbf{t}_{\delta}}(x_1,\dots,x_n)
	$$
	where the first factor is independent on a $\lambda$-tree (only depends on $\lambda$) and is given by (\ref{sfun}). For $\delta=(\delta_1,\dots,\delta_m)\in \Upsilon_{\lambda}$ (in the notations of Section \ref{upsion}) the second factor takes the form:
	$$
	\tilde{S}_{\textbf{t}_{\delta}}(x_1,\dots,x_n)
	= \dfrac{\prod\limits_{r=1}^{m}\Big(\prod\limits_{{{c_{\delta_{r,2}}=c_{\delta_{r,1}-1}}} }\vartheta(x_{\delta_{r,2}}/x_{\delta_{r,1}} t_1) \prod\limits_{{{c_{\delta_{r,2}}=c_{\delta_{r,1}+1}}} }\vartheta(x_{\delta_{r,2}}/x_{\delta_{r,1}} t_2)\Big)}{\prod\limits_{{{c_i=c_j}
				\atop {h_i=h_j+2}}}\vartheta(x_i/x_j \hbar)}
	$$
	Recall that the set $\Upsilon$ is a set of $2^m$ elements, where $m$ is the number of  $\reflectbox{\textsf{L}}$-shaped subgraphs in $\lambda$. Combining $\delta_i$ from the same  $\reflectbox{\textsf{L}}$-shaped subgraphs we can write the sum over trees as the product:
	\be
	\label{sexp}
	\sum_{\sigma\in \Upsilon_{\lambda}}\tilde{S}_{\textbf{t}_{\sigma}}(x_1,\dots,x_n)=\prod\limits_{\{\delta_1,\delta_2\}\,\,  is \,\,  \reflectbox{\textsf{L}}-shaped }\left( \dfrac{ \vartheta(x_{\delta_{1,2}}/x_{\delta_{1,1}} t_2) + \vartheta(x_{\delta_{2,2}}/x_{\delta_{2,1}} t_1)}{ \vartheta(x_{\delta_{2,2}}/x_{\delta_{1,1}} \hbar) } \right)
	\ee 
	Let us show that each multiple of this product 
	is regular at $x_i=\varphi^{\lambda}_{\sigma(i)}$ for all $\sigma$.
	By assumption $\{\delta_1,\delta_2\}$ is a $\reflectbox{\textsf{L}}$-shaped subgraph and thus the corresponding vertices have the following coordinates (\ref{gshapped}):
	$$
	\delta_{1,1}=(a,b), \ \ \delta_{1,2}=\delta_{2,1}=(a+1,b), \ \ \delta_{2,2}=(a+1,b+1).
	$$
	If $\sigma=1$ then both denominator and numerator are vanishing and we need to evaluate the corresponding limit:
	$$
	\lim_{x_i\to\varphi^{\lambda}_i} \dfrac{ \vartheta(x_{\delta_{1,2}}/x_{\delta_{1,1}} t_2) + \vartheta(x_{\delta_{2,2}}/x_{\delta_{2,1}} t_1)}{ \vartheta(x_{\delta_{2,2}}/x_{\delta_{1,1}} \hbar) } =1
	$$ 
Thus, for $\sigma=1$ we have:
	\be 
	\label{slim1}
	\left.\Big(\sum_{\delta\in \Upsilon_{\lambda}}\tilde{S}_{\textbf{t}_{\delta}}(x_1,\dots,x_n)\Big)\right|_{x_i=\varphi^{\lambda}_{i}}=1.
	\ee  
	For $\sigma\neq 1 $ we have
	$$
	\varphi^{\lambda}_{\sigma(\delta_{2,2})}=\varphi^{\lambda}_{\sigma(\delta_{1,1})} \hbar^{p}, \ \ \ 
	\varphi^{\lambda}_{\sigma(\delta_{2,1})}=
	\varphi^{\lambda}_{\sigma(\delta_{1,2})}=\varphi^{\lambda}_{\sigma(\delta_{1,1})} \hbar^{q}t_2^{-1}, 
	$$
	for some integers $p,q$, thus
	$$
	\begin{array}{l}
	\dfrac{ \vartheta(\varphi^{\lambda}_{\sigma(\delta_{1,2})}/\varphi^{\lambda}_{\sigma(\delta_{1,1})} t_2) + \vartheta(\varphi^{\lambda}_{\sigma(\delta_{2,2})}/\varphi^{\lambda}_{\sigma(\delta_{2,1})} t_1)}{ \vartheta(\varphi^{\lambda}_{\sigma(\delta_{2,2})}/\varphi^{\lambda}_{\sigma(\delta_{1,1})} \hbar) } =\dfrac{\vartheta(\hbar^{q})+\vartheta(\hbar^{p-q+1})}{\vartheta(\hbar^{p+1})}\\
	\\
	\cong \dfrac{q}{p+1} +\dfrac{p-q+1}{p+1} +o(\varepsilon)=1+o(\varepsilon).
	\end{array}
	$$
	We conclude that
	\be \label{slims}
	\left.\Big(\sum_{\delta\in \Upsilon_{\lambda}}\tilde{S}_{\textbf{t}_{\delta}}(x_{\sigma(1)},\dots,x_{\sigma(n)})\Big)\right|_{x_i=\varphi^{\lambda}_{i}}=1+o(\varepsilon).
	\ee
	For $\sigma=1$ the equalities (\ref{Wval}), (\ref{slim1}) together with Proposition \ref{svalues} give
	$$
	Q^{\sigma=1}(\hbar,z_i)=F_{\lambda}(\varphi^{\lambda}_1,\dots,\varphi^{\lambda}_n)=1. 
	$$
	For $\sigma\neq 1$ the equations (\ref{Wvals}), (\ref{slims}) together with Proposition \ref{svalues} imply that the function (\ref{Qshift}) does not have poles in $\hbar$. As a pole-free double periodic function it does not depend on $\hbar$:
	$$
	Q^{\sigma}(\hbar,z_i) \dfrac{\vartheta(\hbar z^{\lambda}_{\sigma} )}{\vartheta(\hbar)}  =C(z_i). 
	$$
To compute $C(z_i)$ enough to substitute $\hbar=1/z^{\lambda}_{\sigma}$
(note that $z^{\lambda}_{\sigma}\neq 1$ if $\sigma\neq 1$) which gives	
$C(z_i)=0$. Thus $Q^{\sigma}(\hbar,z_i) =0$.	
\end{proof}

\subsection{}
We would like to rewrite this result in the following form.
Let $\textbf{t}_{\delta}$ be a set of $2^m$ $\lambda$-trees with  $\delta\in \Upsilon_{\lambda}$ as in Section \ref{upsion}. As above we denote by the same symbol $\textbf{t}_{\delta} \in \ahilb_{\lambda}^{\bG}$ the corresponding fixed point. In the notations of Proposition \ref{propstabs} we have:

\begin{Theorem}
The map 	
$$
{\mathscr{U}}^{'} \longrightarrow \Theta(T^{1/2}\ahilb_{\lambda} ) \otimes {\mathscr{U}}^{'} 
$$
defined explicitly by
$$
\textsf{r}=\sum\limits_{\delta \in \Upsilon_{\lambda}} \textrm{Stab}_{\fC^{''}}(\textbf{t}_{\delta}) 
$$
is a formal inverse in the sense of (\ref{frominv}), i.e.: 	
\label{choicet}
\be	 \label{stabcompdef}
\pi^{'}_{*} \circ \j_{+}^{'*} \circ (\j^{'}_{-*})^{-1} \circ \textsf{r} =1
	\ee
\end{Theorem}
\begin{proof}
	The elliptic stable envelopes of the fixed points $\textrm{Stab}_{\fC^{''}}(\textbf{t}_{\delta})$ 
	are given as explicit functions of $x_i$ by Proposition \ref{propstabs}. The map $\pi^{'}_{*} \circ \j_{+}^{'*} \circ (\j^{'}_{-*})^{-1}$ is described explicitly by Proposition \ref{pushpro}. Therefore, the statement of the proposition 
	is equivalent to the identity
	$$
	\sum\limits_{\sigma \in {\frak{S}}_{\lambda}} F^{\sigma}_{\lambda}(\varphi^{\lambda}_{1},\dots,\varphi^{\lambda}_{n})=1.
	$$
which is immediate from the Proposition \ref{vanpro}. 	
\end{proof}

\subsection{} 
{\bf{Proof of the Theorem \ref{mainth}:}}

\vspace{2mm}

\noindent
First,  by  (\ref{abdef}) and (\ref{stabcompdef}) we have:
	\be
	\label{absquare}
	\textrm{Stab}_{\fC}=\pi_{*}\circ \j^{*}_{+}\circ (\j_{- *})^{-1}\circ \textrm{Stab}_{\fC}^{'} \circ \textsf{r}
	\ee
By (\ref{fpush}) the last three maps $\pi_{*} \circ \j_{+}^{*} \circ (\j_{-*})^{-1}$ 
	give exactly the denominator in (\ref{shenpart}) and symmetrization
	of (\ref{ellipticenvelope}).

Seconds, let $\fC$, $\fC^{'}$ and $\fC^{''}$ are the chambers defined in Section \ref{chamsec}.  	
By Theorem~\ref{choicet} we have
$$
\textrm{Stab}^{'}_{\fC}  \circ \textsf{r}=\textrm{Stab}^{'}_{\fC}  \Big(\sum\limits_{\delta\in \Upsilon_{\lambda}} \textrm{Stab}_{\fC^{''}}(\textbf{t}_{\delta}) \Big)=\sum\limits_{\delta\in \Upsilon_{\lambda}} \textrm{Stab}_{\fC^{'}}(\textbf{t}_\delta),
	$$
	where the last step is by triangle Lemma \ref{tlemma} below. The exact expression  for $\textrm{Stab}_{\fC^{'}}(\textbf{t}_{\delta})$ is given by Proposition \ref{propstabs} and (\ref{sh})  gives the numerator in (\ref{shenpart}).

The Proposition \ref{propstabs} provides explicit formulas for elliptic stable envelopes up to some unknown factors of $\hbar$,
arising from the shifts of K\"ahler parameters
$$
\tau^{*} : z_i\to z_i \hbar^{m_i} 
$$
and depending on a choice of polarization.
In particular, the K\"{a}hler part of the elliptic envelope takes the form:
$$
W_{\textbf{t}}(z_i) =(-1)^{\kappa_{\textbf{t}}} \phi(x_r, \hbar^{\textsf{v}_{r}}\prod\limits_{i=1}^{|\lambda|}  z_i) \prod\limits_{e\in \textbf{t}} \phi\Big(\dfrac{x_{h(e)} \varphi^{\lambda}_{t(e)}}{x_{t(e)}\varphi^{\lambda}_{h(e)}}, \hbar^{\textsf{v}_{e}} \!\!\!\!\prod\limits_{i \in [h(e),\textbf{t}]} z_i\Big).
$$
where $\textsf{v}_{r}$ and $\textsf{v}_{e}$ are some integers which are to be defined. Recall that all maps in (\ref{absquare}) are understood after restriction (\ref{zref}) which gives:
	$$
	\prod\limits_{i \in [h(e), \textbf{t}]} z_i \rightarrow \prod\limits_{i \in [h(e), \textbf{t}]} z = z^{\textsf{w}_{e}}.
	$$
and thus 
\be \label{wfin2}
W_{\textbf{t}}(z) =(-1)^{\kappa_{\textbf{t}}} \phi(x_r,  \hbar^{\textsf{v}_{r}} z^n) \prod\limits_{e\in \textbf{t}} \phi\Big(\dfrac{x_{h(e)} \varphi^{\lambda}_{t(e)}}{x_{t(e)}\varphi^{\lambda}_{h(e)}}, \hbar^{\textsf{v}_{e}} z^{\textsf{w}_{e}}\Big).
\ee	
To finish the proof we need to compute  $\textsf{v}_{e}$.
The uniqueness of the elliptic stable envelopes implies 
that these powers are fixed uniquely by the quasiperiods 
of of the corresponding sections. 
Using explicit expression (\ref{shenpart}) we compute:
	\be
	\label{strans}
	\textbf{S}^{Ell}_{\lambda}(x_k q)=-\dfrac{\hbar^{\beta(k)}}{\sqrt{q} x_k}  
\textbf{S}^{Ell}(x_k) 
	\ee
	with $\hbar$-factor given by:
	\be
	\label{bethfor}
	\hbar^{\beta(k)}=\hbar^{\delta_{c_{k}>0}} 
	\Big(\prod\limits_{{c_i=c_k-1} \atop {\h_i<\h_k}} \hbar^{-1}\Big) \Big(\prod\limits_{{c_i=c_k+1} \atop {\h_i>\h_k}} \hbar \Big) \Big(\prod\limits_{{c_i=c_k} \atop {\h_i>\h_k}} \hbar^{-1} \Big)  \Big(\prod\limits_{{c_i=c_k} \atop {\h_i<\h_k}} \hbar \Big),
	\ee
	where all products run over boxes $i\in \lambda$ with specified conditions and
	$$
	\hbar^{\delta_{c_{k}>0}}=\left\{\begin{array}{ll}
	\hbar, & c_{k}>0\\
	1, & \textrm{else}. 
	\end{array}\right. 
	$$  
	The second and the fourth product in (\ref{bethfor}) means that boxes below $k$ in the Russian Young diagram $\lambda$ and boxes with $c_i=c_k+1$ above the box $k$ contribute $+1$ to $\beta(k)$. These boxes are represented in red color in Fig. \ref{indpic}. 
	Similarly, the first and the third product say that the boxes above $k$ and the boxes with $c_i=c_k-1$ below  $k$ in the Russian diagram
	contribute $-1$ to  $\beta(k)$. These boxes represented in green color in Fig. \ref{indpic}.
	
	One can easily see that this quantity does not change if one chooses different box $k$ with the same $c_k$. Thus, $\beta(k)$ depends only on behavior of the profile of $\lambda$ at the point $c_k$. Computing the value of $\beta(k)$ for all possible profiles of $\lambda$ we find that it is given by (\ref{betefun}).   
	
Next, let $e$ be the edge in a $\lambda$-tree with $h(e)=k$  from (\ref{wfin2}) we obtain:
	\be
	\label{wtrans}
		W_{\textbf{t}}(x_k q) =
		\dfrac{
			\prod\limits_{{e^{\prime} \in \textbf{t}:}\atop {t(e^{\prime})=k}   }   z^{\textsf{w}_{e^{'}} }  \hbar^{\textsf{v}_{e^{'}}}  
		 }{	z^{\textsf{w}_{e} }  \hbar^{\textsf{v}_{e} } }
	 W_{\textbf{t}}(x_k)=\dfrac{1}{z}\, \hbar^{\sum\limits_{ {e^{\prime} \in \textbf{t}:}\atop{  t(e')=k}} \!\!\!\!\!\! \textsf{v}_{e'} -\textsf{v}_{e}} W_{\textbf{t}}(x_k)
	\ee 
	where the last equality follows from (\ref{wweight}). 
	Finally, the product of factors (\ref{strans}) and (\ref{wtrans}) must transform as (\ref{xtrans}). We conclude that $\textsf{v}_{e}$ must satisfy:
	$$
	\textsf{v}_{e}=\beta(k)+\sum\limits_{{e'\in \textbf{t}:} \atop {t(e')=k}}\textsf{v}_{e'},
	$$	
which gives (\ref{vweight}). The theorem is proven.  $\Box$

\subsection{}
For the tori specified in Section \ref{chamsec} we have the following triangle of embeddings: 
\[
\xymatrix{\ahilb^{\matC^{\times}_{\textbf{t}} \times \bA} \ar[rr]^{} \ar[dr]^{}&& \ahilb \\
	& \ahilb^{\bA} \ar[ur]^{}& }
\]
Let $\fC$, $\fC^{'}$ and $\fC^{''}$ be the chambers specified in Section \ref{chamsec}, then for each  arrow in this diagram we can associate the corresponding stable envelope map. The stable envelopes $\textrm{Stab}_{\fC}$ and 
$\textrm{Stab}_{\fC'}$ are, by definition, the maps of $\cB_{\bT,\ahilb}$-modules. We can also view $\textrm{Stab}_{\fC''}$ by composing it with
$(1\times i^{*})^*$ where $i: \ahilb^{\bA} \to \ahilb$  is the corresponding canonical map. In fact, above we always work modulo this identification - in Section \ref{restrsec} we defined the tautological bundles
$x_i$ and corresponding K\"{a}hler parameters $z_i$ on $\ahilb^{\bA}$ as
corresponding restrictions from $\ahilb$ and used the same symbols for them. This is exactly the identification by $(1\times i^{*})^*$.   
We thus have three maps of $\cB_{\bT,\ahilb}$-modules which are related by so called \textit{triangle lemma} (see Section 3.6 in \cite{AOElliptic}):
\begin{Lemma} 
	\label{tlemma}
	$$
	\textrm{Stab}_{\fC'} = \textrm{Stab}_{\fC} \circ\tau^{*}\textrm{Stab}_{\fC^{''}}
	$$
	where $\tau^{*}$ stands for a shift of K\"{a}hler parameters induced by the polarization. 
\end{Lemma}

\vspace{1cm}
\begin{figure}[h!]
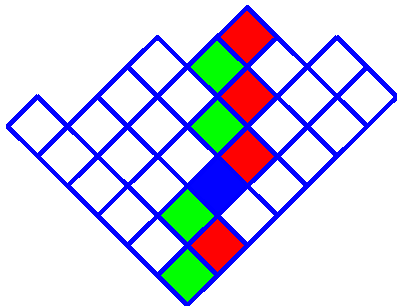

	\hskip 45mm \somespecialrotate[origin=c]{45}{\rg}
	\vspace{1cm}
	\caption{Example of the of the function $\beta_{\lambda}(\Box)$: 
		the boxes denoted in red contribute $+1$ and denoted by green 
		contribute $-1$ to $\beta_{\lambda}(\Box)$. \label{indpic}  }
\end{figure}

\section{Stable envelope in $K$-theory \label{ksec}}
\subsection{}
In the limit $q\to 0$ the section (\ref{elthm}) (normalized by product of inverse square roots of $x_i$) converges to $K$-theoretic Euler class $e_{K}(\tb)$ of the tautological bundle $\tb$:
$$
\lim\limits_{q\to 0}\Big( \prod\limits_{\i=1}^{r}  \vartheta(x_i)/\sqrt{x_i} \Big) = \prod\limits_{i=1}^{r}(1-x_i^{-1})=e_{K}(\tb),
$$
see Sections 7.1-7.2 of \cite{ell1} for discussion of relations between elliptic, K-theoretic and cohomological versions Euler class.

In this section we use this limit to obtain the explicit formulas for stable envelops in the equivariant $K$-theory of $\hilb$. Recall that
in the $K$-theory the stable envelope is a locally constant functions 
of a \textit{slope} parameter \cite{pcmilect,OS,Neg}: 
$$
s\in \textrm{Pic}(X)\otimes_{\matZ}\matR = H^{2}(X,\matR).
$$
The ``locally constant'' means that the stable envelope,
as a function of slope $s$  changes only when  $s$ crosses certain
hyperplanes in $H^{2}(X,\matR)$. These hyperplanes are called ``walls''.
The walls form a $\textrm{Pic}(X)$-periodic hyperplane arrangement in $H^{2}(X,\matR)$.

In this section we reproduce all these data for $X=\hilb$ from a limit of elliptic formulas. In this case, a slope is just a real number:
\be
\label{pic}
s\in \textrm{Pic}(\hilb)\otimes_{\matZ}\matR=\matR, 
\ee
and $\textrm{Pic}(\hilb)=\matZ \subset \matR$ is identified with the integral points. The main result of this section is the explicit formula 
for $K$-theoretic stable envelope (\ref{kthenvelope}). From this formula we find the walls in (\ref{pic}). They are described by Theorem~\ref{wallsth}.

Let us note that the normalizations for $K$-theoretic stable envelopes
used by different authors may differ from each other by a factor. In this section the $K$-theoretic envelopes are normalized as in \cite{pcmilect,OS}. 
The stable envelopes defined in \cite{Neg} differ by a line bundle factor.

\subsection{}
The $K$-theoretic stable envelope can be obtained from the elliptic stable envelope by a procedure described by the following theorem:
\begin{Theorem}[Section 3.8 in \cite{AOElliptic}] \label{klimth}
Let $z=q^{-s}$ for $s\in \matR=\textrm{Pic}(\hilb)\otimes_{\matZ}\matR$, then, in the limit $q=0$ the elliptic stable envelope $\textrm{Stab}_{\fC}$ 
twisted by the determinant of the polarization converges 
to the  $K$-theoretic stable envelope $\textrm{Stab}^{(s)}_{\fC}$ with a slope $s$:
\be \label{glim}
\lim\limits_{q\to 0} \Big(
\det(T^{1/2} \hilb)^{-1/2}  \circ \left.\textrm{Stab}_{\fC}\right|_{z=q^{-s}} \circ \det(T^{1/2 }
\hilb^{\bA})^{1/2}\Big) =\textrm{Stab}^{(s)}_{\fC}.
\ee
\end{Theorem}

The appearance of locally constant functions of slope $s$ in this limit 
is clear from the following proposition: 
\begin{Proposition}
\be
\label{adef}
\lim\limits_{q\to 0} \vartheta(x)= x^{1/2}-x^{-1/2},
\ee
if $s \in \matR\setminus\matZ$ then
\be
\label{slim}
\lim\limits_{q\to 0} \left.\dfrac{ \vartheta(x z) }{ \vartheta(z)}\right|_{z=q^{s}} =\lim\limits_{q\to 0}  \dfrac{ \vartheta(x q^s) }{ \vartheta(q^s)}= x^{\lfloor s \rfloor+1/2},
\ee
where $\lfloor s \rfloor$ denotes the integral part of $s$.
\end{Proposition}
\begin{proof} The proof is by elementary computation using infinite product representation of the theta function (\ref{thet}).
\end{proof}

Let us denote the limit of theta function (\ref{adef}) 
$$
\hat{\textsf{a}}(x)=x^{1/2}-x^{-1/2}.  
$$
The function (\ref{shenpart}) does not depend on  K\"{a}hler parameter $z$. Thus, in the limit (\ref{glim}) it takes the following explicit form:  
\be \label{kshenpart}
\begin{array}{l}
\textbf{S}^{Kth}_{\lambda}(x_1,\dots,x_n)=\\
\\ \dfrac{\prod\limits_{{\rho_j>\rho_i+1}} \hat{\textsf{a}}(x_i x_j^{-1} t_1) \prod\limits_{{\rho_j<\rho_i+1}} \hat{\textsf{a}}(x_j x_i^{-1} t_2) \prod\limits_{\rho_i\leq 0} \hat{\textsf{a}}(x_i)
	\prod\limits_{\rho_i>0} \hat{\textsf{a}}(t_1 t_2 x_i^{-1})}{\prod\limits_{\rho_i<\rho_j}\hat{\textsf{a}}(x_i x_j^{-1})\,\hat{\textsf{a}}(x_i x_j^{-1} t_1 t_2)}.
\end{array}
\ee 
Similarly, using (\ref{slim}) it is elementary to compute the limit
of the  K\"{a}hler part (\ref{wpartell}). In the limit  (\ref{glim}) takes the form:
\be \label{kthweight}
\begin{array}{l}
\textbf{W}^{Kth}({\textbf{t}}; x_1,\dots,x_n,s) =\\
\\ (-1)^{\kappa_{\textbf{t}}} \dfrac{x_r^{\lfloor n s \rfloor+1/2}}{\hat{\textsf{a}}(x_r)} \prod\limits_{e\in \textbf{t} } \hat{\textsf{a}}\Big( \dfrac{x_{h(e)} \varphi^{\lambda}_{t(e)} }{x_{t(e)} \varphi^{\lambda}_{h(e)} } \Big)^{-1} \Big( \dfrac{x_{h(e)} \varphi^{\lambda}_{t(e)} }{x_{t(e)} \varphi^{\lambda}_{h(e)} } \Big)^{{\lfloor \textsf{w}_{e} s \rfloor+1/2}}
\end{array}
\ee
This is explicitly a locally constant function of the slope $s$.  The shifts of the slope parameter by integral values, corresponding to the shifts by line bundles are especially simple to describe. Computing the limit (\ref{glim}) of identity (\ref{Wtrans}) we obtain: 
\begin{Proposition} \label{kthwtrans}
	If $\textbf{t}$ is a $\lambda$-tree then the $K$-theoretic weight functions satisfy:  	
$$
	\textbf{W}^{Kth}({\textbf{t}}; x_1,\dots,x_n,s+1) = \Big(\prod\limits_{i \in \lambda} \, \dfrac{\varphi^{\lambda}_i}{ x_i} \Big)  \, \textbf{W}^{Kth}({\textbf{t}}; x_1,\dots,x_n,s) 
$$
\end{Proposition}
\subsection{}
 As  $\hilb^{\bA}$ is a finite set of points the polarization is trivial 
$$
\det(T^{1/2}\hilb^{\bA})^{1/2}=1.
$$  The polarization for $\hilb$ is given by (\ref{polhilb}) and thus
$$
\det( T^{1/2} \hilb )^{1/2}= t_1^{n^2/2} x_1^{1/2}\cdots x_n^{1/2} .
$$ 

\subsection{} 
The Theorem \ref{klimth} for the case $X=\hilb$ gives the following result. 
\begin{Theorem} \label{kththeor} The $K$-theoretic stable envelope with a slope $s$ of a fixed point $\lambda \in \hilb^{\bT}$ has the following form:
\be \label{kthenvelope}
\begin{array}{|c|}\hline
	\\
	\ \ \ \textrm{Stab}^{(s)}_{\fC}(\lambda) = \textrm{Sym}\Big( \dfrac{ \textbf{S}^{Kth}_{\lambda}( x_1,\dots,x_n)}{ t_1^{n^2/2}x_1^{1/2}\cdots x_n^{1/2}} \sum\limits_{\delta\in \Upsilon_\lambda } \textbf{W}^{Kth}({\textbf{t}}_{\delta};x_1,\dots,x_n; s) \Big)  \\
	\\
	\hline
\end{array}
\ee
where the symbol $\textrm{Sym}$ stands for symmetrization over all variables $x_1,\dots, x_n$. 
\end{Theorem}
Let us note that despite (\ref{kthenvelope}) seemingly contains square roots, the Theorem~\ref{kththeor} implies that it is, in fact,  \textit{a rational function in all parameters and~$\hbar^{1/2}$}.   
By definition, the walls in (\ref{pic}) is the set of points such that the stable envelope changes when slope $s$ crosses one of them. 
\begin{Theorem} \label{wallsth}
For $\hilb$ the set of walls has the form:
$$
\textrm{Walls}(\hilb)=\left\{\dfrac{a}{b} \in \matQ : |b|\leq n\right\}.
$$
\end{Theorem}
\begin{proof}
The $K$-theoretic stable envelope (\ref{kthenvelope}) depends on $s$ through the functions (\ref{kthweight}). It is easy to see that these functions change value at rational points of the form $k/n$ and $k/\textsf{w}_{e}$ for integral $k$. The theorem follows from (\ref{westim}).  
\end{proof}
Note that this set is explicitly $\textrm{Pic}(\hilb)$-periodic, which means it is invariant with respect to shifts by integer numbers. 
\subsection{}
The $K$-theoretic matrix of restrictions is defined by:
$$
T^{(s)}_{\lambda \mu}=i^{*}_{\mu} \textrm{Stab}^{(s)}_{\fC}(\lambda) 
$$
where the restriction to a fixed point $i^*_{\mu}$ is the operator of substitution (\ref{restric}). By the general theory of stable envelopes in $K$-theory this matrix is triangular with respect to dominance ordering on partitions. The diagonal of this matrix does not depend on slope and is equal to:
$$
T^{(s)}_{\lambda \lambda}=\Big(\dfrac{\det N^{-}_{\lambda}}{\det T^{1/2}\hilb}\Big)^{1/2} \otimes \Lambda^{\!\!\bullet} N^{-}_{\lambda} \in K_{\bT}(pt)
$$ 
where $N^{-}_{\lambda}$ is a half of the tangent space $T_{\lambda} \hilb$ spanned by negative $a$-characters such that
$$
\Lambda^{\!\!\bullet} N^{-}_{\lambda}= \sum\limits_{k} (-1)^k \Lambda^{\!k} N^{-}_{\lambda}=\prod\limits_{\Box\in \lambda} \, (1-t_1^{l_{\lambda}(\Box)} t_2^{-a_{\lambda}(\Box)-1}). 
$$
From Proposition \ref{kthwtrans} we see that matrices of restrictions
for slopes which differ by an integer number are conjugated by the corresponding line bundles:
$$
T^{(s+1)}=Mat_{{\mathscr{O}}(1)} T^{(s)} Mat_{{\mathscr{O}}(1)}^{-1}
$$
where $Mat_{{\mathscr{O}}(1)}$ is given by (\ref{omat}). It is, however,
unknown how this matrices change under the non-integral shifts of the slope. For instance, let $s_1<s_2$ be two slopes separated by a single wall $w$ from Theorem \ref{wallsth}. The corresponding wall-crossing operator, also known as wall $R$-matrix: 
$$
{{\textsf{R}}}_{w}=(T^{(s_1)})^{-1} T^{(s_2)} 
$$
is an object of great importance and interest in representation theory and enumerative geometry. These wall $R$-matrices were considered in 
\cite{NegGor}, where several interesting conjectures about them were formulated.   We also expect that the wall $R$-matrices should describe the
monodromies of quantum differential equation for  $\hilb$ obtained in \cite{OP}.  We hope  the explicit results obtained in this paper can   help with a progress in these areas.

\section{Stable envelope in cohomology: Shenfeld's formula \label{shensec}}
The formulas for the stable envelope in the equivariant cohomology of Hilbert scheme $\hilb$ were obtained by D. Shenfeld in his PhD thesis \cite{Shenfeld}, see also \cite{GenJacks} for generalization to moduli spaces of instantons.  The Shenfeld's formula, however, differs from the our: it does not involve the summation over the $\lambda$-trees. In  this section we show that in  cohomology the sum over trees can be computed explicitly. This substantially simplifies the formula for the stable envelope. As a result we  obtain exactly the expression obtained in \cite{Shenfeld}. Thus, we give a new derivation of Shelfeld's formula.

\subsection{}
The formula for the stable envelope in the equivariant cohomology can be obtained from its $K$-theoretic version through a standard procedure:
we substitute all factors in (\ref{kthenvelope}) by their additive versions. In particular, in this section we use the additive version of the box character~(\ref{boxchar}):
$$
\varphi^{\lambda}_{\Box}=(1-j) t_1 + (1-i) t_2.
$$

To obtain the additive version of stable envelope we need to replace factors in (\ref{kshenpart}) and (\ref{kthweight}) involving function $\hat{\textsf{a}}$ by the rule:
$$
\hat{\textsf{a}}( x^n/y^m ) \to n x -m y.
$$
and the rest of the monomial factors become trivial, i.e.,
in (\ref{kthweight}) we substitute:
$$
\Big( \dfrac{x_{h(e)} \varphi^{\lambda}_{t(e)} }{x_{t(e)} \varphi^{\lambda}_{h(e)} } \Big)^{{\lfloor \textsf{w}_{e} s \rfloor+1/2}} \to 1, \ \ \ x_r^{\lfloor n s \rfloor+1/2} \to 1
$$ 
The additive version of (\ref{kshenpart}) takes the form:
\be \label{Scoh}
\begin{array}{l}
	\textbf{S}^{Coh}_{\lambda}(x_1,\dots,x_n)=\\
	\\ \dfrac{\prod\limits_{{\rho_j>\rho_i+1}}(x_i- x_j + t_1) \prod\limits_{{\rho_j<\rho_i+1}}(x_j -x_i + t_2) \prod\limits_{\rho_i\leq 0} x_i
		\prod\limits_{\rho_i>0} (t_1 +t_2- x_i)}{\prod\limits_{\rho_i<\rho_j}(x_i- x_j)(x_i- x_j+t_1+t_2)}.
\end{array}
\ee 
From (\ref{kthweight}) we obtain that the cohomological weight of a $\lambda$-tree
equals: 
\be \label{treecoh}
\textbf{W}^{Coh}({\textbf{t}}; x_1,\dots,x_n) =  (-1)^{\kappa_{\textbf{t}}} \dfrac{1}{x_r} \prod\limits_{e\in \textbf{t} } \dfrac{1}{x_{h(e)}- x_{t(e)}+\varphi^{\lambda}_{t(e)}-\varphi^{\lambda}_{h(e)}}
\ee

\subsection{}
As a result we arrive to the following explicit expression for the stable envelope in cohomology:
\begin{Theorem} \label{cohth} In the equivariant cohomology the stable envelope of a fixed point  $\lambda \in \hilb^{\bT}$ has the following form:
\be \label{cohelope}
\begin{array}{|c|}\hline
	\\
	\ \ \ \textrm{Stab}_{\fC}(\lambda) = \textrm{Sym}\Big(  \textbf{S}^{Coh}_{\lambda}(x_1,\dots,x_n)  \sum\limits_{\delta\in \Upsilon_\lambda } \textbf{W}^{Coh}({\textbf{t}}_{\delta};x_1,\dots,x_n) \Big)  \\
	\\
	\hline
\end{array}
\ee
\end{Theorem}
Note that in cohomology, in contrast with $K$-theory, the stable envelopes do not depend on the slope parameter $s$ .

\subsection{}
Let us show that the cohomological formula (\ref{cohelope}) admits 
a beautiful simplification. First, in this case the sum over trees can be computed explicitly.
\begin{Proposition} \label{trsumprop}
In cohomology, the sum over trees factorizes:
\be \label{trsum}
\sum\limits_{\delta\in \Upsilon_\lambda } \textbf{W}^{Coh}({\textbf{t}}_{\delta};x_1,\dots,x_n)=x_r^{-1} \dfrac{\prod\limits_{{(i,j)\in \lambda:}\atop {(i+1,j+1) \in \lambda}}(x_{(i+1,j+1)}-x_{(i,j)}+t_1 +t_2)}{\prod\limits_{e\in \Gamma_{\lambda}} (x_{h(e)}-x_{t(e)} - \varphi^{\lambda}_{h(e)}+ \varphi^{\lambda}_{t(e)}  ) }
\ee	
where we assume that all edges of the skeleton $\Gamma_{\lambda}$ oriented from the left to right and from the bottom to the top in the French presentation of Young diagram $\lambda$. 
\end{Proposition}
\begin{proof}
First, we note that every tree $\textbf{t}_{\delta}$ in the sum above
contains all edges in 
$$
\Gamma_{\lambda} \setminus \{ \textrm{edges in all} \ \  \reflectbox{\textsf{L}}-\textrm{shaped subgraphs} \}
$$
Thus all $\textbf{W}^{Coh}({\textbf{t}}_{\delta})$, and therefore the sum (\ref{trsum}), contain a common multiple 
corresponding to a product over these edges.  
	
Second, let us consider the edges appearing in $\reflectbox{\textsf{L}}$  -  shaped subgraphs. Recall that $|\Upsilon_\lambda|=2^m$ where $m$ is a number of  $\reflectbox{\textsf{L}}$  -  shaped subgraphs in $\Gamma_{\lambda}$. Let $\gamma=(\delta_1,\delta_2)$ be a $\reflectbox{\textsf{L}}$  -  shaped subgraph such that
$$
\delta_{1,1}=(i,j), \ \ \delta_{2,1}=\delta_{1,2}=(i+1,j), \ \ \delta_{2,2}=(i+1,j+1).
$$
The sum over $\Upsilon_\lambda$  splits to two parts: the threes containing $\delta_1$ and trees containing $\delta_2$. These two sums obviously differ in a factors corresponding to $\delta_1$ and $\delta_2$ respectively thus:
$$
\begin{array}{l}
\sum\limits_{\delta\in \Upsilon_\lambda } \textbf{W}^{Coh}({\textbf{t}}_{\delta})=\\
\Big(\sum\limits_{\delta\in \Upsilon_\lambda^{(m-1)} } \textbf{W}^{Coh}({\textbf{t}}_{\delta})\Big) 
\Big( \dfrac{1}{(x_{(i+1,j)}-x_{(i,j)}+t_2)} + \dfrac{1}{(x_{(i+1,j+1)}-x_{(i+1,j)}+t_1 )}   \Big)
\end{array}
$$
Here, the first factor, the sum over $\Upsilon_\lambda^{(m-1)}$, symbolizes the sum over $2^{m-1}$ subtrees of $\lambda$ which do not contain
$\delta_1$ and $\delta_2$ (they are not $\lambda$-trees). The first term 
and the second term of the second factor are the contributions of $\delta_1$ and $\delta_2$ to (\ref{treecoh}) respectively. We note that sum of these factors is equal to:
$$
\dfrac{x_{(i+1,j+1)}-x_{(i,j)}+t_1+t_2}{(x_{(i+1,j+1)}-x_{(i+1,j)}+t_1 )(x_{(i+1,j)}-x_{(i,j)}+t_2)}
$$ 
This means that the sum over trees contains this factor for every $\reflectbox{\textsf{L}}$  -  shaped subgraph. Next, the sum over $\Upsilon_\lambda^{(m-1)}$ factorize exactly same way and the proposition follows by induction on $m$. 
\end{proof}
\subsection{}
Let us write the function (\ref{Scoh}) in the following form:
$$
\textbf{S}^{Coh}_{\lambda}(x_1,\dots,x_n)=\textbf{S}_{\lambda}(x_1,\cdots,x_n) \textbf{S}_{\lambda}^{'}(x_1,\cdots,x_n)
$$
where the first factor represents the $\frak{S}_{\lambda}$-invariant part of $\textbf{S}^{Coh}_{\lambda}(x_1,\dots,x_n)$:
\be \label{shen}
\begin{array}{l}
\textbf{S}_{\lambda}(x_1,\cdots,x_n)=\\
\\\dfrac{\prod\limits_{{c_j>c_i+1}}(x_i- x_j + t_1) \prod\limits_{{c_j<c_i+1}} (x_j -x_i + t_2) \prod\limits_{c_i<0} x_i
	\prod\limits_{c_i>0} (t_1 +t_2- x_i)}{\prod\limits_{c_i<c_j}(x_i- x_j)(x_i- x_j+t_1+t_2)}.
\end{array}
\ee
We recall that $\frak{S}_{\lambda}$ acts by permuting the {Chern roots roots} $x_i$ with the same content~$c(i)$. Thus, the $\frak{S}_{\lambda}$-invariance of this expression is obvious: the boxes with the same content appear in (\ref{shen}) in a symmetric way. The second factor equals:
$$
\textbf{S}_{\lambda}^{'}(x_1,\cdots,x_n)=
\dfrac{\prod\limits_{ {c_j=c_i+1,}\atop {h_i>h_j}}(x_i-x_j+t_1)
\prod\limits_{ {c_j=c_i+1,}\atop {h_i<h_j}}(x_j-x_i+t_2) \prod\limits_{{c_i=0}\atop {h_i\geq 0}} x_i }{\prod\limits_{{c_i=c_j} \atop {h_i>h_j}} (x_i-x_j)(x_i-x_j+t_1+t_2)}.
$$
Let us consider the contribution of this non $\frak{S}_{\lambda}$-symmetric function and trees to the stable envelope. In other words, we consider the function the function:  
$$
S_{\Gamma_{\lambda}}(x_1,\dots,x_n)=\textbf{S}_{\lambda}^{'}(x_1,\cdots,x_n) \sum\limits_{\delta\in \Upsilon_\lambda } \textbf{W}^{Coh}({\textbf{t}}_{\delta};x_1,\dots,x_n)
$$
By Proposition \ref{trsumprop} we have:
$$
S_{\Gamma_{\lambda}}(x_1,\dots,x_n)=
\dfrac{\prod\limits_{ {{c_j=c_i+1,}\atop {h_i>h_j},} \atop (i,j)\notin \Gamma_{\lambda}}(x_i-x_j+t_1)
	\prod\limits_{ {{c_j=c_i+1,}\atop {h_i<h_j}}  \atop (i,j)\notin \Gamma_{\lambda}}(x_j-x_i+t_2) \prod\limits_{{c_i=0}\atop {h_i> 0}} x_i }{\prod\limits_{{c_i=c_j} \atop {h_i>h_j}} (x_i-x_j)   \prod\limits_{{c_i=c_j} \atop {h_i>h_j+2}}(x_i-x_j+t_1+t_2)}.
$$
\begin{Proposition} \label{cohpro}
\be \label{lsum}
\sum\limits_{\sigma\in    {\frak{S}}_{\lambda}  }\, S_{\Gamma_{\lambda}}(x_{\sigma{(1)}},\dots,x_{\sigma{(n)}})=1.
\ee	
\end{Proposition}
\begin{proof}
The proof essentially repeats the proof of Propositions \ref{vanprop1}, \ref{vanprop2} and \ref{svalues}. Indeed, we note that the numerator and denominator of $S_{\Gamma_{\lambda}}(x_1,\dots,x_n)$ 
are the rational versions of (\ref{num}) and (\ref{den}).  Arguing exactly as in the proofs of Propositions \ref{vanprop1} and \ref{vanprop2} we find that the sum in the left side of (\ref{lsum}) 
is a rational function of total degree zero which does not have poles. 
Thus, this sum is a constant. Evaluating this function at $x_i=\varphi^{\lambda}_i$ as in Proposition \ref{svalues} 
gives that the constant is equal to $1$. 
\end{proof}
\subsection{}
Let us denote
$$
\frak{z}_{\lambda} = |\frak{S}_{\lambda}|= \prod_i \textsf{d}_{i}(\lambda)!. 
$$
Let us recall that the function (\ref{shen}) is invariant with respect to the action of a subgroup $\frak{S}_{\lambda}\subset \frak{S}_{n}$.  This and  Proposition \ref{cohpro} together give:
\begin{Theorem}
The cohomological stable envelope of a fixed point $\lambda\in \hilb^{\bA}$ equals
\be
\begin{array}{|c|}
\hline  
\\
 \ \ \textrm{Stab}_{\fC}(\lambda) = \dfrac{1}{\frak{z}_{\lambda}}\, \textrm{Sym}( \textbf{S}_{\lambda}(x_1,\cdots,x_n) )\ \  \\
\\
\hline
\end{array}
\ee  	
\end{Theorem}
In this form the expression for the cohomological stable envelope appears in Shenfeld's thesis \cite{Shenfeld}. In fact, the same result can also be proved for $K$-theoretic stable envelope with \textit{integral slopes}.

\bibliographystyle{abbrv}
\bibliography{bib}

\newpage

\vspace{12 mm}

\noindent
Andrey Smirnov\\
Department of Mathematics, UC Berkeley\\
Berkeley, CA 94720-3840, U.S.A\\

\noindent
Institute for Problems of Information Transmission\\
Bolshoy Karetny 19, Moscow 127994, Russia\\

\end{document}